\documentclass{article}
\usepackage{amsmath,amssymb,amsthm}
\usepackage{enumerate}
\usepackage{xcolor, wasysym}
\usepackage{multirow}
\usepackage{comment}
\usepackage[marginpar = .6in, left=1in,right=1in]{geometry}
\usepackage{graphicx,import}
\usepackage{overpic}
\usepackage{tikz}
\usepackage{tikz-cd}
\usepackage{quiver}

\usepackage[colorlinks,citecolor=black,linkcolor=black]{hyperref}
\usepackage[nameinlink]{cleveref}

\theoremstyle{plain}
\newtheorem{THM}{Theorem}[section]

\newtheorem{PROP}[THM]{Proposition}
\newtheorem{LEM}[THM]{Lemma}
\newtheorem{COR}[THM]{Corollary}

\newtheorem{CLAIM}{Claim}
\newtheorem{CLAIME}{Claim}
\newtheorem{CASE}{Case}

\theoremstyle{definition}
\newtheorem{DEF}[THM]{Definition}
\newtheorem{RMK}[THM]{Remark}
\newtheorem{EX}[THM]{Example}
\newtheorem{Q}[THM]{Question}

\newtheorem{innercustomthm}{Theorem}
\newenvironment{customthm}[1]
  {\renewcommand\theinnercustomthm{#1}\innercustomthm}
  {\endinnercustomthm}
\newtheorem{innercustomcor}{Corollary}
\newenvironment{customcor}[1]
  {\renewcommand\theinnercustomcor{#1}\innercustomcor}
  {\endinnercustomcor}

\DeclareMathOperator{\Id}{Id}

\DeclareMathOperator{\Out}{Out}
\DeclareMathOperator{\PMap}{PMap}
\DeclareMathOperator{\PMapc}{PMap_{c}}
\DeclareMathOperator{\PHE}{PHE}

\DeclareMathOperator{\PPHE}{PPHE}
\DeclareMathOperator{\Homeo}{Homeo}
\DeclareMathOperator{\Map}{Map}
\DeclareMathOperator{\diam}{diam}
\DeclareMathOperator{\supp}{supp}
\DeclareMathOperator{\rk}{rk}
\DeclareMathOperator{\cork}{cork}
\DeclareMathOperator{\asdim}{asdim}

\newcommand{\Z}{\mathbb{Z}}
\newcommand{\R}{\mathbb{R}}

\newcommand{\cD}{\mathcal{D}}

\newcommand{\cF}{\mathcal{F}}
\newcommand{\cH}{\mathcal{H}}
\newcommand{\cL}{\mathcal{L}}
\newcommand{\cV}{\mathcal{V}}
\newcommand{\cU}{\mathcal{U}}
\newcommand{\cP}{\mathcal{P}}

\newcommand{\cR}{\mathcal{R}}
\newcommand{\cT}{\mathcal{T}}
\newcommand{\A}{\mathcal{A}}

\def\G{{\Gamma}}

\def\a{{\alpha}}
\def\e{{\epsilon}}
\newcommand{\inv}{^{-1}}
\newcommand{\arr}{\rightarrow}

\newcommand{\PMapcc}[1]{\overline{\PMapc(#1)}}
\newcommand{\defeq}{:=}
\newcommand{\la}{\langle}
\newcommand{\ra}{\rangle}

\renewcommand{\>}{\rangle}

\newcommand{\wm}[1]{\varphi_{\left(#1\right)}}

\newcommand{\lasso}{%
    \begin{tikzpicture}[scale=0.1, thick]
    {
    [yshift = -1]
    \draw circle(1);
    \draw (1.0,0.0)--(5.0,0.0);
    }
    \end{tikzpicture}%
}

\usetikzlibrary{shapes.geometric, arrows, cd}
\tikzstyle{start} = [rectangle, rounded corners, minimum width=3cm, text width=3cm, minimum height=1cm,text centered, draw=black, fill=yellow!30]
\tikzstyle{CB} = [rectangle, rounded corners, minimum width=3cm, text width=3.2cm, minimum height=1cm,text centered, draw=black, fill=blue!20]
\tikzstyle{nonCB} = [rectangle, rounded corners, minimum width=3cm, text width=3.2cm, minimum height=1cm,text centered, draw=black, fill=red!30]
\tikzstyle{locCB} = [rectangle, rounded corners, minimum width=4cm, text width=4.2cm, minimum height=1cm,text centered, draw=black, fill=blue!20]
\tikzstyle{locnonCB} = [rectangle, rounded corners, minimum width=4cm, text width=4.2cm, minimum height=1cm,text centered, draw=black, fill=red!30]
\tikzstyle{io} = [trapezium, trapezium left angle=70, trapezium right angle=110, text width=3cm, minimum width=3cm, minimum height=1cm, text centered, draw=black, fill=gray!30]
\tikzstyle{decision} = [diamond, minimum width=2.5cm, minimum height=2.5cm, text width=1.5cm, text centered, draw=black, fill=green!30]
\tikzstyle{locdecision} = [diamond, minimum width=1.5cm, minimum height=1.5cm, text width=2cm, text centered, draw=black, fill=green!30]
\tikzstyle{locdecision2} = [diamond, minimum width=1.5cm, minimum height=1.5cm, text width=2.9cm, text centered, draw=black, fill=green!30]
\tikzstyle{andBlock} = [semicircle, shape border rotate=270, minimum width=1cm, minimum height=1cm, text width = 1cm, text centered, draw=black, fill=orange!60]
\tikzstyle{arrow} = [thick,->,>=stealth]

\title{Coarse Geometry of Pure Mapping Class Groups of Infinite Graphs}
\author{George Domat, Hannah Hoganson, Sanghoon Kwak}
\date{}

\begin{document}

\maketitle

\begin{abstract}
  We discuss the large-scale geometry of pure mapping class groups of locally finite, infinite graphs, motivated by
  recent work of Algom-Kfir--Bestvina \cite{AB2021}
  and the work of Mann--Rafi \cite{mann2022large} on the large-scale geometry of mapping class groups of infinite-type surfaces. Using the framework of Rosendal \cite{rosendal2022} for coarse geometry of non-locally compact groups, we classify when the pure mapping class group of a locally finite, infinite graph is globally coarsely bounded (an analog of compact) and when it is locally coarsely bounded (an analog of locally compact). 
  
  Our techniques give lower bounds on the first integral cohomology of the pure mapping class group for some graphs and allow us to compute the asymptotic dimension of all locally coarsely bounded pure mapping class groups of infinite rank graphs. This dimension is always either zero or infinite. 
\end{abstract}


\section{Introduction}

The mapping class group of a surface, $\Map(S)$, is the group of orientation-preserving homeomorphisms of the surface up to isotopy. Mapping class groups of finite-type surfaces have been a classical field of study for several decades, and within the past decade, there has been a newfound interest in the study of mapping class groups of infinite-type, also known as \emph{big}, surfaces. See \cite{aramayona2020big} for a recent survey on the topic of mapping class groups of infinite-type surfaces. 

The study of mapping class groups of finite-type surfaces is also intimately connected with the study of outer automorphism groups of free groups, $\Out(F_{n})$. There is a rich dictionary between $\Map(S)$ and $\Out(F_{n})$ when $S$ is of finite type, and this has led to numerous results in both fields. See \cite{bestvina2014geometry} and \cite{vogtmann2002automorphisms} for an in-depth survey of $\Out(F_{n})$ and its connections to mapping class groups. 

This connection and the recent interest in mapping class groups of infinite-type surfaces begs the question: What is an appropriate ``big'' or ``infinite-type'' analog of $\Out(F_{n})$? Recent work of Algom-Kfir--Bestvina proposes such a definition. 

\begin{DEF} [{\cite[Definition 1.1]{AB2021}}]
    Let $\Gamma$ be a locally finite, connected graph. The \text{mapping class group} of $\Gamma$, denoted by $\Map(\Gamma)$, is the group of proper homotopy equivalences of $\Gamma$ up to proper homotopy. 
\end{DEF}

When $\Gamma$ is finite, this definition exactly recovers the group $\Out(F_{n})$, where $n$ is the rank of $\Gamma$. Thus, when $\Gamma$ is infinite, we obtain a version of a ``big $\Out(F_{n})$.'' In \cite{AB2021} the authors prove a version of Nielsen realization for these groups. Another natural ``big" analog of $\Out(F_{n})$ is $\Out(F_{\infty})$, where $F_{\infty}$ is a countable-rank free group. In \cite{DHK2022Aut} we show that $\Out(F_{\infty})$ has the quasi-isometry type of a point.

Mapping class groups of surfaces are topological groups with a topology coming from the compact-open topology on $\Homeo(S)$. In the finite-type setting, these groups are finitely generated, countable, and discrete \cite{Dehn1987,Lickorish1964}. In the infinite-type setting, these groups are not compactly generated, but they are Polish and homeomorphic to the irrationals \cite{aramayona2020big,PatelVlamis}. The same holds for $\Map(\Gamma)$ when $\Gamma$ is an infinite graph \cite{AB2021}. This presents a challenge in applying the tools of geometric group theory and coarse geometry to these groups. However, recent work of Rosendal \cite{rosendal2022} has established a framework for studying the coarse geometry of non-locally compact Polish groups. Rosendal introduces the notion of ``coarsely bounded'' sets which serve as a replacement for compact subsets. Rosendal proves that if a group is generated by a coarsely bounded generating set, then the quasi-isometry type with respect to the word metric induced by that generating set is well-defined, i.e., the identity map with a different coarsely bounded generating set is a quasi-isometry. If a group is only known to be locally coarsely bounded, then it still admits a well-defined coarse equivalence (a weakening of quasi-isometry) type, provided that it also \emph{has arbitrarily small subgroups} (See \Cref{ssec:pmspace}).

Within this framework, Mann-Rafi \cite{mann2022large} have begun the study of the coarse geometry of mapping class groups of infinite-type surfaces. Under a mild condition on infinite-type surfaces, they give a complete classification of surfaces whose mapping class groups are coarsely bounded (have trivial geometry), have a coarsely bounded neighborhood about the identity (well-defined coarse equivalence type), or have a coarsely bounded generating set (well-defined quasi-isometry type).

Our goal is to establish similar results as Mann-Rafi in the setting of infinite-type graphs. The \textbf{pure mapping class group}, $\PMap(\Gamma)$, of a graph $\Gamma$ is the subgroup of $\Map(\Gamma)$ which fixes the set of ends of $\Gamma$ pointwise. We give a complete classification of when these subgroups are coarsely bounded, i.e., have trivial geometry, and when these groups are locally coarsely bounded. We use $E(\G)$ to represent the end space of $\G$ and $E_{\ell}(\G)$ for the subspace of ends accumulated by loops, defined in \Cref{sec:preliminaries}. See \Cref{fig:flowchart} for a summary of our classification of coarse boundedness of $\PMap(\G)$.

\begin{customthm}{A}\label{THM:Main}
    Let $\Gamma$ be a locally finite, infinite graph. Then $\PMap(\G)$ is coarsely bounded if and only if one of the following holds. \begin{enumerate}
        \item $\G$ has rank zero, or
        \item $\G$ has rank one and has one end, i.e. $\G=\lasso$, or
        \item $\G$ satisfies both:
        \begin{itemize}
            \item $|E_{\ell}(\G)| = 1$, and
            \item $E(\G) \setminus E_{\ell}(\G)$ discrete.
        \end{itemize}
    \end{enumerate}
\end{customthm}

\begin{figure}[ht!]
\centering
\makebox[\textwidth][c]{
\scalebox{.7}{
\begin{tikzpicture}[node distance=2cm]

\node (start) [start] {Start};
\node (in) [io, below of=start,yshift=-.1cm] {$\Gamma$ = locally finite, infinite graph};
\node (decrank) [decision, below of=in, yshift=-1cm] {\mbox{\hspace{-2px}$\textrm{rank}(\Gamma)=?$}};
\draw [arrow] (start) -- (in);
\draw [arrow] (in) -- (decrank);

\node (CBrank0) [CB, below left of=decrank, xshift= -1.5cm, yshift=-.5cm] {$\textrm{PMap}(\Gamma)$ is CB (\Cref{prop:rank0})};
\node (declasso) [decision, below of=decrank, yshift=-1.5cm] {$\vphantom{\overset{\overset{?}{=}}{=}}\Gamma\underset{?}{=}\lasso$};
\node (nonCBrank234) [nonCB, below right of=decrank, xshift = 1.5cm, yshift=-.5cm] {\mbox{$\textrm{PMap}(\Gamma)$ is Not CB} (\Cref{cor:finiterank})};
\draw [arrow] (decrank) -| node[pos=.75,anchor=west] {0} (CBrank0);
\draw [arrow] (decrank) -- node[pos=.5,anchor=west] {1} (declasso);
\draw [arrow] (decrank) -| node[pos=.75,fill=white] {$\ge 2$,finite} (nonCBrank234);

\node (CBlasso) [CB, below left of=declasso, xshift= -1.5cm, yshift=-.5cm] {$\textrm{PMap}(\Gamma)$ is CB (\Cref{prop:rank1})};
\node (nonCBrank1) [nonCB, below right of=declasso, xshift = 1.5cm, yshift=-.5cm] {\mbox{$\textrm{PMap}(\Gamma)$ is Not CB} (\Cref{prop:rank1})};
\draw [arrow] (declasso) -| node[pos=.75,anchor=west] {Yes} (CBlasso);
\draw [arrow] (declasso) -| node[pos=.75,anchor=east] {No}(nonCBrank1);

\node (decnumacl) [decision, right of=decrank, xshift = 6.4cm, yshift=2.1cm] {\mbox{\hspace{-3px}$|E_{\ell}(\Gamma)|=?$}};
\node (decdiscrete) [decision, below of=decnumacl, yshift=-3cm] {$E \setminus E_{\ell}$ discrete?};
\draw [arrow] (decrank) -- ++(6,0) |-(decnumacl);
\draw [arrow] (decrank) -- node[pos=.75,fill=white] {$\infty$}++(6,0) |- (decdiscrete);

\node (nonCBloopshift) [nonCB, below of=decnumacl, yshift = -.5cm] {\mbox{$\textrm{PMap}(\Gamma)$ is Not CB} (\Cref{thm:twoendnotCB})};

\draw [arrow] (decnumacl.south) -- node[pos=.35,anchor=east] {$\ge 2$}(nonCBloopshift);

\node (nonCBnondiscrete) [nonCB, below of=decdiscrete, yshift=-.5cm] {\mbox{$\textrm{PMap}(\Gamma)$ is Not CB} (\Cref{thm:oneendnotCB})};
\draw [arrow] (decdiscrete) -- node[pos=.35,anchor=east] {No}(nonCBnondiscrete);

\node (AND) [andBlock, right of=decdiscrete, xshift = 2cm,yshift=.62cm]{AND};
\draw [arrow] (decnumacl.east) -- node[pos=.3,anchor=north]{1} ++(1.05,0) |- (AND.north west);
\draw [arrow] (decdiscrete.east) -- node[pos=.3,anchor=north] {Yes} (AND.south west);
\node (CBdiscrete) [CB, right of=nonCBnondiscrete, xshift= 3.5cm] {$\textrm{PMap}(\Gamma)$ is CB (\Cref{THM:oneendcb})};
\draw [arrow] (AND.east) -| (CBdiscrete);
\end{tikzpicture}
} 
} 
\caption{Flowchart for classifying coarsely bounded $\PMap(\G)$.}
\label{fig:flowchart}
\end{figure}
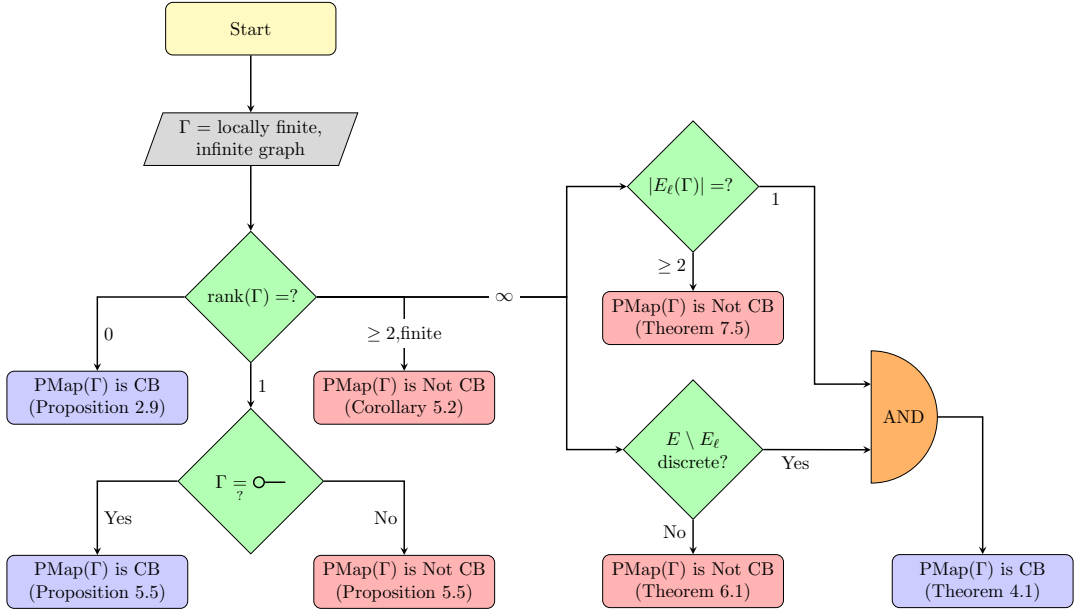

We note that this classification is \emph{not} the same classification as the one for surfaces. That is, there are surfaces with non-CB pure mapping class groups for which the ``corresponding'' graph has a CB pure mapping class group. The simplest example is the punctured Loch Ness Monster surface (one end accumulated by genus with one isolated puncture) versus the Hungry Loch Ness Monster graph (one end accumulated by loops and one other end).

We can also extend these results to the full mapping class group in some cases.

\begin{customcor}{B} 
\label{COR:MapisCB}
  If $\G$ has at least two ends accumulated by loops, and has finite end space, then the full group $\Map(\G)$ is \emph{not} coarsely bounded.
  If $\G$ has a single end accumulated by loops and $E(\G)\setminus E_{\ell}(\G)$ is discrete, then $\Map(\G)$ is coarsely bounded. 
\end{customcor}

The techniques that we use to prove \Cref{THM:Main} yield further results. The following result is analogous to a result of Aramayona-Patel-Vlamis \cite{APV2020} in the surface setting. 

\begin{customthm}{C} \label{THM:indicable} Let $\Gamma$ be a locally finite, infinite
  graph with $|E_{\ell}(\Gamma)| \geq 2$. Then $\PMap(\Gamma)$ has a nontrivial continuous
  homomorphism to $\Z$. Furthermore, if $|E_{\ell}(\G)| = n$ with
  $2\leq n< \infty$, then $\rk\left(H^{1}(\PMap(\G);\Z)\right) \geq n-1$. If
  $\lvert E_{\ell}(\G)\rvert = \infty$, then
  $H^{1}(\PMap(\G);\Z) = \bigoplus_{i=1}^{\infty} \Z$.
\end{customthm}

In particular, \Cref{THM:indicable} tells us that if $\G$ has more than one end accumulated by loops then $\PMap(\G)$ is indicable, that is, it admits a surjective homomorphism to $\Z$.
The techniques used to prove \Cref{THM:Main} also allow us to give a complete classification of graphs for which $\PMap(\G)$ is \emph{locally CB}, that is, $\PMap(\G)$ has a CB neighborhood of the identity. See \Cref{fig:flowchartloccb} for a summary of the results and proofs of \Cref{THM:LocallyCB}.
The \textbf{core graph} of $\Gamma$, denoted by $\Gamma_{c}$, is the smallest subgraph that contains all immersed loops in $\G$.

\begin{customthm}{D}\label{THM:LocallyCB}
    Let $\G$ be a locally finite, infinite graph. Then $\PMap(\G)$ is locally coarsely bounded if and only if one of the following holds. \begin{enumerate}
        \item $
        \G$ has finite rank, or
        \item $\G$ satisfies both:
        \begin{itemize}
            \item $|E_{\ell}(\G)| < \infty$, and
            \item only finitely many components of $\G \setminus \G_c$ have infinite end spaces.
      \end{itemize}
    \end{enumerate}
\end{customthm}

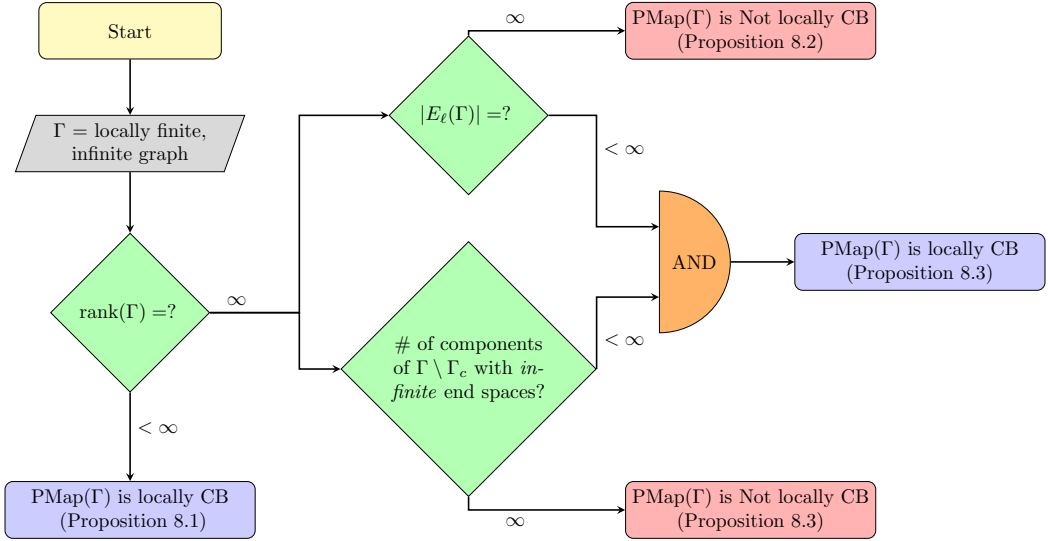
\begin{figure}[ht!]
\centering
\makebox[\textwidth][c]{
\scalebox{.75}{
\begin{tikzpicture}[node distance=2cm]
\node (start) [start] {Start};
\node (in) [io, below of=start] {$\Gamma$ = locally finite, infinite graph};
\node (decrank) [locdecision, below of=in, yshift=-1cm] {\mbox{\hspace{-2px}$\textrm{rank}(\Gamma)=?$}};
\draw [arrow] (start) -- (in);
\draw [arrow] (in) -- (decrank);

\node (CBfiniterk) [locCB, below of=decrank, yshift=-1.5cm] {\mbox{$\textrm{PMap}(\Gamma)$ is locally CB}\\~(\Cref{prop:finitelocCB})};

\draw [arrow] (decrank) --  node[pos=.4,anchor=west] {$<\infty$}(CBfiniterk);

\node (decnumacl) [locdecision, right of=decrank, xshift = 4cm, yshift = 3.5cm] {\mbox{\hspace{-3px}$|E_{\ell}(\Gamma)|=?$}};
\node (decnumofcomps) [locdecision2, below of=decnumacl,yshift=-2.5cm] {\mbox{$\#$ of components} \mbox{of $\G \setminus \G_c$ with} \emph{infinite} end spaces?};
\draw [arrow] (decrank) --++(3,0) |- (decnumacl);
\draw [arrow] (decrank) --++(3,0) node[pos=.3,anchor=south] {$\infty$} |- (decnumofcomps);

\node (AND) [andBlock, right of=decnumacl, xshift = 2cm, yshift=-2.6cm]{AND};
\node (locCBfiniteindisc) [locCB, right of=AND,xshift = 2cm] {\mbox{$\textrm{PMap}(\Gamma)$ is locally CB} (\Cref{prop:finiteendloopsloccb})};
\node (locnonCBloopshift) [locnonCB, right of=decnumacl, xshift=3cm, yshift =1.5cm] {\mbox{$\textrm{PMap}(\Gamma)$ is Not locally CB} (\Cref{prop:infendloopslocCB})};
\node (locnonCBinfindisc) [locnonCB, right of=decnumofcomps, xshift=3cm, yshift =-2.5cm] {\mbox{$\textrm{PMap}(\Gamma)$ is Not locally CB} \mbox{(\Cref{prop:finiteendloopsloccb})}};

\draw [arrow] (decnumacl) |- node[pos=.65,anchor=south] {$\infty$} (locnonCBloopshift);
\draw [arrow] (decnumofcomps) |- node[pos=.65,anchor=north] {$\infty$}(locnonCBinfindisc);

\draw [arrow] (decnumacl.east) -- ++(.85,0)|- node[pos=.15,anchor=west] {$<\infty$} (AND.north west);
\draw [arrow] (decnumofcomps.east) |- node[pos=.2,anchor=west] {$<\infty$} (AND.south west);
\draw [arrow] (AND) -- (locCBfiniteindisc);

\end{tikzpicture}
} 
} 
\caption{Flowchart for classifying locally CB $\PMap(\G)$.}
\label{fig:flowchartloccb}
\end{figure}

This condition of being locally CB, together with having arbitrarily small subgroups, is enough to guarantee that our groups have a well-defined coarse equivalence type. While this is not as strong as having a well-defined quasi-isometry type (guaranteed by CB-generation), we can still start computing some coarse invariants for these groups. In particular, the asymptotic dimension is a coarse equivalence invariant. We compute the asymptotic dimension of these locally CB groups for all infinite rank graphs. The asymptotic dimension of the mapping class groups of some infinite-type surfaces was shown to be infinite in \cite{GRV2021}. We show that for graphs, there are pure mapping class groups of infinite asymptotic dimension and those of dimension 0 (that are not themselves CB). 

\begin{customthm}{E}\label{THM:asympdim}
    Let $\G$ be a locally finite, infinite rank graph with $\PMap(\G)$ locally CB. Then
    \[
            \asdim \PMap(\G) =
            \begin{cases}
                     0 & \text{if $|E_\ell(\G)|=1$} \\
                     \infty & \text{otherwise; $|E_\ell(\G)| >1$.}
            \end{cases}
    \]
\end{customthm}

To see when $\PMap(\G)$ has asymptotic dimension $0$, we build a simplicial tree for $\PMap(\G)$ to act on. Elements of $\PMap(\G)$ coarsely embed as the set of leaves in this tree. This tree comes equipped with a natural height function with $\PMap(\G)$ acting transitively on each level set. The tree is one-ended so that $\PMap(\G)$ also fixes the point at infinity.

To begin getting a feel for how these groups function, we highlight two key differences between mapping class groups of infinite graphs and surfaces:

\textbf{$\Map(\Gamma)$ always displaces finite subgraphs of $\Gamma$ when $\Gamma$ has infinite rank.}  Nondisplaceable subsurfaces were integral to the arguments of Mann and Rafi \cite{mann2022large}. To see why any subgraph can be displaced by this mapping class group, let $\Gamma$ be a graph with infinite rank and $\Delta$ be any finite subgraph. The complement of $\Delta$ in $\Gamma$ is still infinite rank, so we can find a subgraph $\Delta'$ of the same rank as $\Delta$, but disjoint from $\Delta$. Both $\Delta$ and $\Delta'$ deformation retract onto a rose of equal rank. Thus, we can interchange $\Delta$ and $\Delta'$ by a proper homotopy equivalence to displace $\Delta$. When $\Gamma$ has nonzero \emph{finite} rank $r$, this is not the case. For example, we can choose a subgraph of rank greater than $\frac{r}{2}$, which must be nondisplaceable.

\textbf{The restriction of a (proper) homotopy equivalence may not be a homotopy equivalence.}
For example, consider a rose with two loops labeled by $a$ and $b$. The map defined by $a \mapsto a$ and $b \mapsto ab$ is a homotopy equivalence, but it maps a loop labeled by $b$ of rank 1 to a subgraph of rank 2, while homotopy equivalences induce isomorphisms on fundamental groups. This contrasts with the fact that the restriction of a surface homeomorphism to a subsurface is still a homeomorphism. This observation demonstrates that there is no analogous change of coordinates principle for graphs, which is a commonly used technique in the field of mapping class groups of surfaces.

\section*{Outline}
In \Cref{sec:preliminaries}, we give the necessary background. In particular, we
review some of the known facts about locally finite graphs of infinite type and
their end spaces, and some basic facts on coarse structures on groups.
\Cref{SEC:Elements} gives the reader a hands-on introduction to different
elements that exist in $\PMap(\G)$. \Cref{sec:CBPMCG} concerns graphs with one
end accumulated by loops.  In \Cref{sec:finiteposrk}, we use the structure of
$\PMap(\G)$ for graphs with finite, positive rank, established in \cite{AB2021},
to show that almost all of these groups are not CB. In \Cref{SEC:Length}, we
define length functions on a large class of graphs---graphs with infinite trees
or combs---showing that their pure mapping class groups are not CB and act on
simplicial trees. We define flux maps in \Cref{sec:TwoEnds} and use them to
prove \Cref{THM:indicable} and that if $\G$ has at least two ends accumulated by
loops, then $\PMap(\G)$ is not CB. The results of Sections \ref{sec:CBPMCG}
through \ref{sec:TwoEnds} together prove \Cref{THM:Main}. In
\Cref{ssec:LocalCB}, we use the previously established techniques to classify
graphs for which $\PMap(\G)$ is locally CB, proving Theorem \ref{THM:LocallyCB}.
Finally, in \Cref{sec:asymdim},  we compute the asymptotic dimension of
$\PMap(\G)$ that are locally CB with $\G$ of infinite rank, proving
\Cref{THM:asympdim}. We include an appendix in which we give an explicit proof that ultrametric spaces are $0$-hyperbolic.

\section*{Acknowledgements}

  The authors are grateful to Mladen Bestvina for suggesting this project, reading the first draft, and providing thoughtful comments. The authors would also like to thank Priyam Patel for a careful reading of an earlier draft and especially correcting the condition for \Cref{LEM:compositionTongues}. Also we thank Yael Algom-Kfir for attentive comments resulting in both the simpler proof of \Cref{LEM:multiplebps} and the clarification of the proof of \Cref{LEM:folding}.
  Thank you to Ryan Dickmann and Brian Udall for their helpful conversations, and Ty Ghaswala and Anschel Schaffer-Cohen for useful comments. We also thank the referee for numerous helpful comments. In addition, the authors acknowledge support from NSF grants DMS--1906095 (Hoganson), DMS--1905720 (Domat, Kwak), and RTG DMS--1840190 (Domat, Hoganson). 

\section{Preliminaries}
\label{sec:preliminaries}
\subsection{Infinite Locally Finite Graphs}
\label{ss:InfLocFinGraphs}

Let $\Gamma$ be a locally finite, infinite, connected graph. We often forget about the actual graph structure of $\Gamma$ and regard it simply as its underlying topological space. Since $\Gamma$ can be realized as the direct limit of nested finite graphs, the fundamental group of $\Gamma$ is free. We define the \textbf{rank} of $\Gamma$, denoted by $\rk(\Gamma) \in \Z_{\ge 0} \cup \{\infty\}$, to be the rank of its fundamental group.  The \textbf{space of ends} $E(\Gamma)$ of $\Gamma$ is the inverse limit:
\[
  E(\Gamma):= \varprojlim_{K\subset \Gamma} \pi_{0}(\Gamma \setminus K),
\] 
where the limit runs over the compact subsets $K \subset \Gamma$.  Equipped with the usual inverse limit topology, $E(\Gamma)$ becomes a totally disconnected compact metrizable space, so is homeomorphic to a closed subset of the Cantor set. Also, $E(\Gamma)$ compactifies $\Gamma$ in the union $\Gamma \cup E(\Gamma)$ (sometimes referred to as the ``end compactification'' or Freudenthal compactification). This allows us to define the neighborhood of an end in the ambient graph $\Gamma$ by picking a neighborhood in the union first then taking the intersection with $\Gamma$. Note this is different from the neighborhood of an end in the end space $E(\Gamma)$.

We define the \textbf{space of ends accumulated by loops}, denoted by $E_{\ell}(\Gamma)$, to be the subspace of $E(\Gamma)$ consisting of ends for which every neighborhood in $\Gamma$ is of infinite rank.  When no confusion will arise we sometimes refer to $(E(\G),E_{\ell}(\G))$ as simply $(E,E_{\ell})$. Any end that is accumulated by ends accumulated by loops is itself accumulated by loops, and thus $E_{\ell}$ is a closed subspace of $E$. Observe that the rank of $\Gamma$ is infinite if and only if $E_{\ell}$ is nonempty.  We remark that elements of $E_{\ell}$ are called \textit{non-$\infty$-stable ends} in \cite{ayala1990proper}, and \textit{ends accumulated by genus} in \cite{AB2021}.

We will consider the homeomorphism type of the pair $(E,E_{\ell})$, where we say $(E,E_{\ell}) \cong (E',E'_{\ell})$ when there is a homeomorphism $h:E \arr E'$ which restricts to a homeomorphism $h\vert_{E_{\ell}}: E_{\ell} \arr E'_{\ell}$. Now we define the \textbf{characteristic triple} of $\Gamma$ as the triple $\left(\rk (\Gamma), E(\Gamma), E_{\ell}(\Gamma)\right)$. We say two characteristic triples $(r,E,E_{\ell})$ and $(r',E',E_{\ell}')$ are \textbf{isomorphic} if $r=r'$ and $(E,E_{\ell})$ is homeomorphic to $(E',E'_{\ell})$.

As Ker{\' e}kj{\' a}rt{\' o} \cite{Kerekjarto1923} and Richards's \cite{Richards1963} classification theorem for surfaces is the foundation for the study of infinite-type surfaces, we use the following classification for locally finite graphs established by Ayala--Dominguez--M{\'a}rquez--Quintero:

\begin{THM}[{\cite[Theorem 2.7]{ayala1990proper}, Proper Homotopy Classification of Locally Finite Graphs}]
  \label{thm:InfGraphClass}
  An isomorphism $(r,E,E_\ell) \to (r',E',E'_\ell)$ of characteristic pairs of two locally finite, connected graphs $\Gamma$ and $\Gamma'$ extends to a proper homotopy equivalence $\Gamma \to \Gamma'$. If $\G$ and $\G'$ are trees, then this extension is unique up to proper homotopy.
\end{THM}
Conversely, by definition of ends, a proper homotopy equivalence yields an isomorphism between characteristic pairs. Therefore, two connected graphs $\Gamma$ and $\Gamma'$ have the same proper homotopy type if and only if they have isomorphic characteristic triples of end spaces.

For the remainder of the paper, we assume all of our graphs to be connected, infinite, and locally finite. For a graph $\Gamma$, we let $\Gamma_{c}\subset \Gamma$, the \textbf{core graph} of $\Gamma$, be the smallest subgraph that contains all immersed loops.

We will make use of the fact that connected graphs are $K(G,1)$ spaces for free groups. In particular, we use the following proposition.  

\begin{PROP}[{\cite[Proposition 1B.9]{Hatcher2002}}] \label{PROP:Hatcher}
  Let $X$ be a connected CW complex and let $Y$ be $K(G,1)$. Then every homomorphism $\pi_1(X,x_0) \arr \pi_1(Y,y_0)$ is induced by a map $(X,x_0)\arr (Y,y_0)$ that is unique up to homotopy fixing $x_0$.
\end{PROP}

A consequence of this proposition is the following lemma. 

\begin{LEM} \label{LEM:multiplebps}
    Let $\Gamma$ be a connected, finite graph of rank at least $2$ and let $x_{1},\ldots,x_{n}$ be distinct points in $\Gamma$. Let $f,g: \Gamma \rightarrow \Gamma$ be two homotopy equivalences that both fix $x_1,\ldots,x_n$. If $f_{*,i},g_{*,i}:\pi_{1}(\Gamma,x_{i}) \rightarrow \pi_{1}(\Gamma,x_{i})$ are the same induced maps for all $i$ then $f$ is homotopic to $g$ via a homotopy fixing $x_{1},\ldots, x_{n}$. 
\end{LEM}

\begin{proof}
    First, applying \Cref{LEM:multiplebps} to $f_{*,1}=g_{*,1}$, we find a homotopy $H_1:\G \times I \to \G$ between $f$ and $g$, fixing $x_1$. Fixed by $f$ and $g$, the point $x_2$ describes a loop $\alpha$ in $\G$ during the homotopy $H_1$. Formally, $H_1(x_2,I) = \alpha$. We claim $\alpha$ is nulhomotopic. Indeed, setting $x_2$ as the base point and from the homotopy $H_1$ between $f$ and $g$, we have (e.g. by \cite[Lemma 1.19]{Hatcher2002}): 
    \[
        f_{*,2}([\gamma]) = [\alpha] \cdot g_{*,2}([\gamma]) \cdot [\alpha]^{-1}, \qquad \text{for every $[\gamma] \in \pi_1(\G,x_2)$,}
    \]
    which is actually equal to $g_{*,2}([\gamma])$ by the assumption. Since $g_{*,2}$ is an isomorphism, it follows that $[\alpha]$ is in the center of $\pi_1(\G,x_2)$, which is trivial as $\G$ has rank at least 2. Hence, $[\alpha]=1$.
    
    Now using the nulhomotopy of $\alpha$, we illustrate the way to modify $H_1$ to obtain a homotopy \emph{rel} $x_1,x_2$ between $f$ and $g$. Consider first the pair:
    \[
    (\G \times I \times I,\ \left(\{x_1,x_2\} \times I \times I\right) \cup \left(\G \times \{0,1\} \times I\right)) =: (\G \times I \times I,\ A),
    \]
    on which we will use the Homotopy Extension Property. For this, we specify the homotopy $\cH : (\G \times I \times \{0\} \cup A) \to \G$ ``between homotopies'' we want to extend. First, we define $\cH|_{\G \times I \times \{0\}} = H_1$. On $\{x_1\} \times I \times I$, define $\cH$ as the constant map to $x_1 \in \G$. Using the nulhomotopy of $\alpha$, define $\cH$ on $\{x_2\} \times I \times I$ so that $\cH(\{x_2\} \times I \times \{0\}) = \alpha$ and $\cH(\{x_2\} \times I \times \{1\}) = \{x_2\}$. On $\G \times \{0\} \times I$ and $\G \times \{1\} \times I$, define $\cH$ as $f$ and $g$ respectively. To be precise, for every $t \in I$:
    \[
        \cH(x,0,t) := f(x), \qquad \cH(x,1,t) := g(x).
    \]
    Then we apply Homotopy Extension Property on the pair $(\G \times I \times I,\ A)$ with $\cH$ to obtain the extension of $\cH$, defined on the full domain $\G \times I \times I$. Denoting the extension by $\cH$ again, the map $\cH|_{\G \times I \times \{1\}}$ is the desired homotopy $H_2$ between $f$ and $g$, rel $x_1$ \emph{and} $x_2$.

    Inductively we repeat the same process for each of the remaining fixed points.
    Namely, given the homotopy $H_k$ between $f$ and $g$ rel $x_1,\ldots,x_k$, we obtain a new homotopy $H_{k+1}$ between $f$ and $g$ rel $x_1,\ldots,x_k$ \emph{and} $x_{k+1}$ using the Homotopy Extension Property on the pair
    \[
        (\G \times I \times I, \ \left(\{x_1,\ldots,x_{k+1}\} \times I \times I\right) \cup \left(\G \times \{0,1\} \times I\right)) =: (\G \times I \times I,\ A),
    \]
    with the ``initial homotopy'' being $H_k$.
    Iterating the process, ultimately we obtain the homotopy $H_n$ betwen $f$ and $g$ rel $x_1,\ldots, x_n$, concluding the proof.
 \end{proof}

\subsection{Big Mapping Class Groups of Infinite Graphs}
\label{ss:BMCGofGraphs}

A continuous map is \textbf{proper} if the inverse image of every compact set is compact.

\begin{DEF}

For any (infinite) locally finite graph $\Gamma$, we define $\PHE(\Gamma)$ as
the group of proper homotopy equivalences (or PHEs). Recall that we say a map $f:\Gamma \rightarrow \Gamma$ is a \textbf{proper homotopy equivalence} if $f$ is proper and there exists some $g:\Gamma \rightarrow \G$ that is also \emph{proper} such that $gf$ and $fg$ are \emph{properly} homotopic to the identity. 

We define the \textbf{mapping class group} of $\Gamma$, $\Map(\Gamma)$,
as the group of proper homotopy classes of proper homotopy equivalences on $\Gamma$:

\[
    \Map(\Gamma) = \PHE(\Gamma)/\text{proper homotopy}.
\]
\end{DEF} 

Note that when $\Gamma$ is a finite graph this definition recovers $\Out(F_{n})$ where $n = \rk(\G)$. Thus we see that by taking $\G$ to be infinite we obtain a type of ``big $\Out(F_{n})$''.

\begin{RMK}
  \label{RMK:notPHE}
    To ensure that $\PHE(\G)$ is a group, we do need to include in the definition of a proper homotopy equivalence that the inverse homotopy equivalence is also \textit{proper}. That is, there are examples of homotopy equivalences that are proper but whose homotopy inverses are never proper.
    To illustrate, let $\Gamma$ be the graph with one end which is accumulated by loops. Label each loop by $a_1,a_2,a_3 \cdots$, which we identify with the corresponding elements in $\pi_1(\Gamma)$. Consider a map $f : \Gamma \to \Gamma$, whose induced map $f_*$ on $\pi_1(\Gamma)$ is defined by
    $a_1 \mapsto a_1$, and $a_i \mapsto a_{i-1}a_i$ for $i \ge 2$.

    Since $f_*$ is an isomorphism, $f$ is a homotopy equivalence. Moreover, the inverse homotopy equivalence $g$ of $f$ induces $f_*^{-1}:\pi_1(\Gamma) \to \pi_1(\Gamma)$, defined by
    \[
    a_1 \mapsto a_1, \quad a_2 \mapsto a_1^{-1}a_2, \quad a_3 \mapsto a_2^{-1}a_1a_3, \quad a_4 \mapsto a_3^{-1}a_1^{-1}a_2 a_4 \cdots.
    \]

    It can be seen that $f^{-1}_*$ maps every loop in $\Gamma$ around $a_1$. Therefore, the preimage of $a_1$ under the any representative of inverse homotopy equivalence of $f$ is never compact, so $f$ has no proper inverse homotopy equivalence. This forces $f \not\in \PHE(\G).$
\end{RMK}

\begin{DEF} \label{def:support} 
    For $\phi \in \PHE(\Gamma)$ we say that $\phi$ is \textbf{totally supported} on $K \subset \Gamma$ if $\phi(K) = K$ and $\phi\vert_{\G\setminus K} = \Id$. We say that $[\phi] \in \Map(\Gamma)$ is \textbf{totally supported} on $K$ if there is a proper homotopy representative of $[\phi]$ that is totally supported on $K$. 
\end{DEF}

\begin{RMK}
    We will use the term support in its usual way. That is, if $\phi \in \PHE(\Gamma)$, then the \textbf{support} of $\phi$ is the closure of the set of $x \in \Gamma$ such that $\phi(x) \neq x$. Note that for homeomorphisms (e.g.\ of a surface), being supported on $K$ is equivalent to being totally supported on $K$. This is not true for homotopy equivalences, since they are not necessarily injective. 
\end{RMK}

We would like to endow $\Map(\Gamma)$ with a topology that comes from the topology of $\hat{\Gamma} = \Gamma \cup E(\Gamma)$, the end compactification of $\Gamma$.
To do so, we put the compact-open topology on the set $\mathcal{C}(\hat{\Gamma})$ of continuous maps on $\hat{\Gamma}$.
A proper homotopy equivalence on $\Gamma$ extends to a continuous map on $\hat{\Gamma}$, so we can embed $\PHE(\Gamma)$ into $\mathcal{C}(\hat{\Gamma})$, from which $\PHE(\Gamma)$ inherits the subspace topology.
Algom-Kfir and Bestvina show in \cite[Corollary 4.3]{AB2021} that the map
\[
    q: \PHE(\Gamma) \to \Map(\Gamma)
\]
is an open map. Thus, $\Map(\Gamma)$ inherits the quotient topology from the topology on $\PHE(\G)$.

With this topology, a neighborhood basis about the identity map in $\Map(\Gamma)$ is given as follows. For each finite subgraph $K$ of $\Gamma$, we have an open neighborhood $\cV_K$ of the identity given by:
\begin{align*}
    \cV_K = &\{[f] \in \Map(\Gamma):\ \exists f' \in [f] \quad \text{s.t.}
    \quad f'|_K = \Id_K, \\
    &\text{ and $f'$ preserves each complementary component of $K$.}\}
\end{align*}

Algom-Kfir and Bestvina prove that these sets are clopen subgroups \cite[Proposition 4.7]{AB2021}. They also show \cite[Proposition 4.11]{AB2021} that this topology makes $\Map(\Gamma)$ into a \textit{Polish}(separable and metrizable) group, with the underlying space homeomorphic to $\Z^\infty$.

\subsection{Pure Mapping Class Groups and Homeomorphism Groups of End Spaces}
\label{ss:PMAPandHomeo}

Recall as in \Cref{ss:InfLocFinGraphs}, every proper homotopy equivalence of a graph $\Gamma$ extends to a homeomorphism of the space of ends $(E,E_{\ell})$ of $\Gamma$. Thus we see that $\Map(\Gamma)$ acts on the space of ends via homeomorphisms. 

\begin{DEF}
    The \textbf{pure mapping class group}, $\PMap(\Gamma)$, is the kernel of the action of $\Map(\Gamma)$ on the space of ends of $\Gamma$. 
\end{DEF}

The pure mapping class group is a closed subgroup of $\Map(\G)$ and thus is Polish with respect to the subspace topology. These groups fit into the following short exact sequence:
\[
    1 \longrightarrow \PMap(\Gamma) \longrightarrow \Map(\Gamma) \longrightarrow \Homeo(E,E_{\ell}) \longrightarrow 1.
\]

In particular, if $\G$ is a tree (i.e., of rank 0), then by \Cref{thm:InfGraphClass} every self proper homotopy equivalence of $\G$ arises from the homeomorphism of the end space $E(\G)$.
By definition every element in $\PMap(\G)$ induces the identity map on $E(\G)$, so we deduce that $\PMap(\G)=1$. Also, from the short exact sequence we have $\Map(\G)\cong \Homeo(E(\G))$. We record these observations as follows:
\begin{PROP}
\label{prop:rank0}
Let $\G$ be a locally finite, infinite tree. Then $\PMap(\G)=1$ and $\Map(\G) \cong \Homeo(E(\G))$.
\end{PROP}

We equip $\Homeo(E,E_{\ell})$ with the compact-open topology. Note that for a finite clopen partition $\cP = \{P_{1}, P_{2}, \ldots, P_{n}\}$ of $E$ the sets of the form
\begin{align*}
    \cU_{\cP} = \{f \in \Homeo (E,E_{\ell})\ \vert\ f(P_{i})=P_{i} \text{ for all } i\}
\end{align*}
give a neighborhood basis about the identity in $\Homeo(E,E_{\ell})$. 

\begin{RMK}
    \label{RMK:TwoTops}
    Note that the map $q: \Map(\Gamma) \to \Homeo(E,E_{\ell})$ is an open map, so in particular a quotient map. Hence, the quotient topology on $\Homeo(E,E_{\ell})$ induced by $q$ coincides with the compact-open topology on $\Homeo(E,E_{\ell})$ generated by the maps fixing partitions of end space. Indeed, consider a basic set $\cV_{K}$ in $\Map(\Gamma)$ for some compact set $K \subset \Gamma$. Then by definition of $q$, the image $q(\cV_{K})$ is the set of homeomorphisms on $E$ which preserve the partition of $(E,E_{\ell})$ induced by $\Gamma \setminus K$, which forms a basic set of the compact-open topology of $\Homeo(E,E_{\ell})$, so $q$ is an open map. 
\end{RMK}

\subsection{Stallings Folds}

We will make use of the notion of Stallings folds defined and used in \cite{stallings1983topology} throughout Section \ref{sec:CBPMCG}.

\begin{DEF}
    A \textbf{morphism} of graphs is a continuous map that sends vertices to vertices and edges to edges. An \textbf{immersion} is a locally injective morphism of graphs. 
\end{DEF}

\begin{DEF}
    Let $\Gamma$ be a graph and $e_{1},e_{2}$ two edges in $\Gamma$ sharing a vertex. Form a new graph $\Gamma' = \Gamma / e_{1} \sim e_{2}$. The natural quotient morphism $\Gamma \rightarrow \Gamma'$ is a \textbf{fold}. 
\end{DEF}

Folds come in two flavors depending on whether $e_{1}$ and $e_{2}$ share only a single vertex or both vertices. Type 1 folds are the folds between edges sharing only one vertex and Type 2 folds are the folds between edges sharing both vertices. While Type 1 folds are $\pi_{1}$-isomorphisms, Type 2 folds are only $\pi_1$-surjective and not $\pi_1$-injective. In fact, a Type 1 fold is a proper homotopy equivalence.

\begin{THM} [\cite{stallings1983topology}]
    Let $f:\Gamma \rightarrow \Gamma'$ be a morphism between two finite graphs. Then $f$ can be factored as: 
    \begin{align*}
        \Gamma = \Gamma_{0} \overset{\phi_1}{\longrightarrow} \Gamma_{1} \overset{\phi_2}{\longrightarrow} \Gamma_{2} \overset{\phi_3}{\longrightarrow} \cdots \overset{\phi_n}{\longrightarrow} \Gamma_{n} \overset{h}{\longrightarrow} \Gamma'
    \end{align*}
    where the last map $h$ is an immersion and all the other maps $\{\phi_i\}_{i=1}^n$ are folds. 
\end{THM}

While Stallings' theorem is for finite graphs, we will be partially folding morphisms on \emph{infinite} graphs in order to obtain an immersion on a finite subgraph.
When we are folding proper homotopy equivalences, we do not use Type 2 folds as they are not $\pi_1$-isomorphisms.

\subsection{Coarse Structures on Groups}
\label{ss:coarse}

In this section, we give some definitions and basic results about coarse structures on spaces, first introduced by Roe \cite{roe2003lectures}. For more details on this section, refer to \cite[Chapter 2]{rosendal2022}. In particular, we do not state the formal definition of coarse boundedness here but only the relevant equivalent definitions as worked out in \cite{rosendal2022}.

\begin{DEF}[{\cite[Proposition 2.15]{rosendal2022}}]
  \label{PROP:RosendalCB}
    Let $A$ be a subset of a Polish group $G$. Then we say that $A$ is \textbf{coarsely bounded (CB)} in $G$ if one of the following equivalent conditions is satisfied.
    \begin{enumerate}[(1)]
        \item (Rosendal's Criterion)
            For every neighborhood $\cV$ of the identity in $G$, there is a finite subset $\cF$ of $G$ and some $n \geq 1$ such that $A \subset (\cF\cV)^{n}$.
        \item
            For every continuous action of $G$ on a metric space $X$ and every $x \in X$, $\diam(A \cdot x) < \infty$.
    \end{enumerate}
  \end{DEF}

\begin{EX}
    \label{ex:finiteCB}
    Any finite group equipped with the discrete topology is coarsely bounded. Similarly, any compact topological group is coarsely bounded.
\end{EX}

    Thanks to the following observation deduced from \Cref{PROP:RosendalCB}, whenever we have the conclusion that $\PMap(\G)$ is CB \emph{in itself} we can extend it to the fact that $\PMap(\G)$ is CB \emph{in $\Map(\G)$}.
    
  \begin{COR}
    \label{COR:CBinLargerGroups}
    Let $G$ be a Polish group and $H$ be a Polish subgroup. If $H$ is CB in itself, then $H$ is CB in $G$.
  \end{COR} 
  
  \begin{proof}
      Any continuous action of $G$ on $X$ will restrict to a continuous action of $H$ on $X$, so this follows from $(2)$ of \Cref{PROP:RosendalCB}.
  \end{proof}

  In the category of coarse spaces, isomorphisms are given by coarse equivalences.
  We will not state the definition here as it can be quite technical, but note that it extends the notion of a quasi-isometry to the larger class of spaces equipped with a coarse structure. In \Cref{ssec:pmspace} we will state a version of the definition for pseudo-metric spaces. Again we refer to \cite{rosendal2022} for details on this and we collect a few facts below that will be useful to us. All of the proofs of the statements below are either contained in or given by elementary arguments using the definitions in \cite[Chapter 2]{rosendal2022}.
    
    \begin{PROP}[Coarse boundedness is a coarse equivalence invariant]
      \label{PROP:CBisCEI}
      If $X$ and $Y$ are coarsely equivalent, then $X$ is CB if and only if $Y$ is CB.
  \end{PROP}

The following variant of a proposition of Rosendal tells us how coarse geometries of groups in a short exact sequence are related to one another. 

\begin{PROP}[cf. {\cite[Proposition 4.37]{rosendal2022}}]
  \label{PROP:SES} Suppose $K$ is a closed normal subgroup of a Polish group $G$ and assume that $K$ is coarsely bounded in $G$. Then the quotient map
\[
    \pi: G \to G/K
\]
is a coarse equivalence. In particular, $G$ is CB if and only if $G/K$ is CB. 
\end{PROP}
This together with \Cref{COR:CBinLargerGroups}, the short exact sequence in \Cref{ss:PMAPandHomeo} and \Cref{RMK:TwoTops} allows us to conclude:

\begin{COR} \label{COR:OutCEHomeo}
  Let $\Gamma$ be a locally finite, infinite graph with end space $(E,E_{\ell})$.
  If $\PMap(\Gamma)$ is CB, then $\Map(\Gamma)$ is coarsely equivalent to $\Homeo(E,E_{\ell})$.
\end{COR}

Finally, we verify that the property of being CB is closed under passing to (open) finite index subgroups and extensions.

\begin{PROP}[cf. {\cite[Proposition 5.67]{rosendal2022}}]
    \label{PROP:finiteindexCB}
    Let $G$ be a Polish group and $H \le G$ be a finite index open Polish subgroup. Then $H$ is CB if and only if $G$ is CB.
\end{PROP}

\begin{proof}
    Let $[G:H] = n$ and $G/H=\{g_1H,\ldots,g_nH\}$. Assume first that $H$ is CB. By \Cref{COR:CBinLargerGroups}, $H$ is CB in $G$. Then for any identity neighborhood $\cV$ in $G$, there exist a finite set $\cF \subset G$ and $m>0$ such that $H \subset (\cF \cV)^m$. Now, taking $\cF'=\{g_1,\ldots,g_n\}\cup \cF$:
    \[
        G=\bigcup_{i=1}^n(g_iH) \subset \bigcup_{i=1}^n g_i(\cF \cV)^m \subset (\cF'\cV)^{m+1},
    \]
    which implies that $G$ is CB.

    Conversely, suppose that $G$ is CB. Since $G$ is a topological group and $H$ is a finite index open subgroup, it is also closed. Then by \cite[Proposition 5.67]{rosendal2022} $H$ is \emph{coarsely embedded} in $G$: a subset $A \subset H$ is CB in $H$ if and only if $A$ is CB in $G$. Since $G$ is CB, $H$ is CB in $G$ by definition, which further implies that $H$ is CB in itself, concluding the proof.
\end{proof}

Because any closed finite index subgroup of a topological group is open, we have the following. 

\begin{COR}
  \label{COR:finiteEndCB}
  Let $\Gamma$ be a locally finite, infinite graph with a finite end space. Then $\PMap(\Gamma)$ is CB if and only if $\Map(\Gamma)$ is CB. 
\end{COR}

\subsection{Groups as Pseudo-metric Spaces} \label{ssec:pmspace}
This section will provide a background for \Cref{ss:asdim0} and \Cref{ssec:asdimInf}, in which we compute the asymptotic dimension of locally CB $\PMap(\G)$. See \cite[Chapter 2 and Section 3.6]{rosendal2022} for more details on this in the CB setting and \cite{cornulier2016} in the locally compact setting. The goal of this section is to show the following proposition.

\begin{PROP}
\label{prop:WDCEpmap}
Let $\G$ be a locally finite, infinite graph with locally CB $\PMap(\G)$. Then there exists a pseudo-metric on $\PMap(\G)$ that is well-defined up to coarse equivalence.
\end{PROP}

\begin{DEF}
A group $G$ is said to \textbf{have arbitrarily small open subgroups} if for any identity neighborhood $V$ of $G$, there exists an open subgroup $K$ of $G$ such that $K \subset V$. 
\end{DEF}

\begin{DEF}[{\cite[Definition 2.51]{rosendal2022}}]
  Let $G$ be a Polish group and $d$ a pseudo-metric on $G$. Then $d$ is said to be \textbf{coarsely proper} if for every $x_{0} \in G$ and $R\ge 0$, the metric ball $\{x \in G\ \vert\ d(x_{0},x)\le R\}$ is CB in $G$.
\end{DEF}

\begin{PROP}[Existence, {cf. \cite[Theorem 2.38]{rosendal2022}}] \label{prop:LocCBCPmetric}
  Let $G$ be a separable, metrizable, locally CB, topological group that has arbitrarily small open subgroups. Then $G$ admits a continuous left-invariant and coarsely proper pseudo-metric.
\end{PROP}

\begin{proof}
  Let $\cV$ be a CB neighborhood of the identity in $G$. By taking smaller subgroups, we may assume $\cV$ is an open subgroup of $G$.
  Since $G$ is separable and metrizable, it is Lindel{\" o}f. Hence, $G$ can be covered by countably many (disjoint) cosets of $\cV$. Write $G = \bigcup_{i=0}^{\infty}g_{i}\cV$, with $g_{0} = \Id$ and $g_{i} \in G$ for $i>0$.
  Define a length function $\ell$ on $G$ by $\ell(g_{i})=\ell(g_{i}^{-1})=i$ for $i \ge 0$ and $\ell(u)=0$ for $u \in \cV$.
  Set $S= \{g_{i}^{\pm 1}\}_{i=0}^{\infty} \cup \cV$.
  Now, for a general element $g \in G$, define:
  \[
    \ell(g):= \inf\left\{\left.\sum_{j=1}^{m} \ell(s_{j})\ \right|\ g = s_{1}\cdots s_{m}, \quad \text{ for some } m \in \Z^{+},\ s_{i} \in S \right\}.
  \]
  Then $\ell$ induces a left-invariant pseudo-metric $d$ on $G$. To check $\ell$ is continuous on $G$, we check that $\ell^{-1}(\{n\})$ is closed for each $n \in \Z_{\ge 0}$. Choose a convergent sequence $\{h_i\}_{i \in \Z^{+}}$ with $h_i \arr h$, such that $\ell(h_i)=n$. We want to show $\ell(h)=n$. In fact, the sequence $\{h^{-1}h_i\}_{i \in \Z^+}$ converges to $\Id$. As $\cV$ is open, by taking the tail of the sequence we may assume $h^{-1}h_i \in \cV$ for all $i \in \Z^+$, so $\ell(h^{-1}h_i)=0$. However, by the triangle inequality:
  \[
    n= \ell(h_i) - \ell(h_i^{-1}h) \le \ell(h) \le \ell(h_i) + \ell(h_i^{-1}h) = n,
  \]
  so $\ell(h)=n$, concluding $\ell$ is continuous.
   
  Now to check that $d$ is coarsely proper, it suffices to check that the metric balls centered at the identity are CB. Namely, let $B_{R}=\{g \in G| \ell(g) \le R\}$ for $R \ge 0$.
  We have $B_{0}=\cV$, which is CB in $G$. Now assume $R>0$. Observe that
  \[
    B_{R} \subset (\cV \cup \{g_{1},\ldots,g_{R}\})^{\lfloor R\rfloor}.
  \]
  Since a finite union of CB-sets is CB, and a finite power of a CB-set is CB, it follows that $(\cV \cup \{g_{1},\ldots,g_{R}\})^{R}$ is CB in $G$, so $B_{R}$ is CB in $G$, concluding the proof.
\end{proof}

To prove that the continuous, left-invariant, coarsely proper, pseudo-metric is well-defined, we introduce the definition of a coarse equivalence between pseudo-metric spaces.

\begin{DEF}[Coarse Equivalence]
    \label{def:metricCE}
      Let $f:(X,d_X) \to (Y,d_Y)$ be a map between two pseudo-metric spaces. Then $f$ is said to be \textbf{coarsely Lipschitz} if there exists a non-decreasing function $\Phi_{+}:[0,\infty) \to [0,\infty)$, called an \textbf{upper control function} of $f$, such that
      \[
      d_Y(f(x),f(x')) \le \Phi_{+}(d_X(x,x')),
      \]
      for all $x,x' \in X$. Similarly, $f$ is said to be \textbf{coarsely expanding} if there exists a non-decreasing function $\Phi_{-}:[0,\infty) \to [0,\infty]$, called a \textbf{lower control function} of $f$, 
      such that $\lim_{r \to \infty}\Phi_{-}(r) \to \infty$ and
      \[
      \Phi_{-}(d_X(x,x')) \le d_Y(f(x),f(x')),
      \]for all $x,x' \in X$. We say $f$ is a \textbf{coarse embedding} if it is coarsely Lipschitz and coarsely expanding. 
      Further, $f$ is said to be \textbf{coarsely surjective} if there exists a $C \ge 0$ such that for any $y \in Y$ there exists an $x \in X$ such that $d_Y(y,f(x))\le C$.
      Finally, the map $f$ is a \textbf{coarse equivalence} if it is a coarse embedding and coarsely surjective.
    \end{DEF}

\begin{PROP}[Uniqueness, {cf. \cite[Lemma 2.52]{rosendal2022}}]
  \label{prop:WDCEtype}
  Let $G$ be a separable, metrizable, locally CB, topological group that has arbitrarily small subgroups. Then a continuous, left-invariant, coarsely proper, pseudo-metric on $G$ is well-defined up to coarse equivalence.
  More generally, if $d,d'$ are continuous, left-invariant, pseudo-metrics on $G$ and $d$ is coarsely proper, then the identity map $\Id:(G,d) \to (G,d')$ is coarsely Lipschitz.
\end{PROP}

 \begin{proof}
    It suffices to show the latter statement because if $\Id:(G,d') \to (G,d)$ is coarsely Lipschitz, then the Lipschitz constants give a lower control function of the inverse map, $\Id:(G,d) \to (G,d')$, showing that $\Id: (G,d) \to (G,d')$ is a coarse embedding. Because the identity map is (coarsely) surjective, it follows that $\Id:(G,d) \to (G,d')$ is a coarse equivalence.
    
   Hence, we prove $\Id:(G,d) \to (G,d')$ is coarsely Lipschitz: There exists a non-decreasing function $\Phi_{+}:[0,\infty) \to [0,\infty)$ such that
     $d'(\Id,g) \le \Phi_{+}(d(\Id,g))$
     for every $g \in G$. Define $\Phi_{+}$ as:
     \[
       \Phi_{+}(m) = \sup\{d'(\Id,g)\ \vert\ \text{for $g \in G$ s.t. } d(\Id,g)\le m\}.
     \]
     Then by definition $\Phi_{+}$ is non-decreasing. Hence, it suffices to prove that $\Phi_{+}$ only admits a finite value. Suppose for the sake of contradiction $\Phi_{+}(m)=\infty$ for some $m>0$. This implies that there exists a sequence $\{g_{n}\}_{n=1}^{\infty}$ of elements in $G$ such that $d'(\Id,g_{n}) \to \infty$ as $n \to \infty$, but $d(\Id,g_{n})\le m$ for all $n\ge 1$. Note $d$ is coarsely proper, so $B_{m}:= \{g \in G\ \vert\ d(\Id,g)\le m\}$ is CB in $G$. However, $G$ admits a continuous length function $\ell'$, which is unbounded on $B_{m}$ obtained from $d'$ and this contradicts the assumption that $B_{m}$ is CB in $G$. nTherefore, $\Phi_{+}(m) < \infty$ for every $m\ge 0$ and this concludes the proof.
 \end{proof}

Now we are ready to prove \Cref{prop:WDCEpmap}.

\begin{proof}[Proof of \Cref{prop:WDCEpmap}]
    In \Cref{ss:BMCGofGraphs}, we have seen $\PMap(\G)$ is Polish and has arbitrarily small open basic subgroups $\cV_K$.
    Therefore, by \Cref{prop:LocCBCPmetric} we obtain a coarsely proper pseudo-metric on $\PMap(\G)$, which is well-defined up to coarse equivalence by \Cref{prop:WDCEtype}.
\end{proof}

\section{Elements of $\PMap(\G)$} \label{SEC:Elements}

In this section we call attention to different types of elements in $\PMap(\G)$, specifically \emph{word maps}, \emph{loop swaps}, and \emph{loop shifts}. We use word maps and loop swaps in \Cref{sec:CBPMCG} and loop shifts in \Cref{sec:TwoEnds}. In order to define these maps we first introduce some standard forms and notation for graphs. Additionally, we hope this section provides the reader with a better hands-on understanding of the groups $\PMap(\G)$.

Throughout this section we use \emph{loop} in the graph theoretic sense, that is an edge whose initial and terminal vertices are the same. We use \emph{based loop} to refer to an element of a fundamental group.

\subsection{Standard Forms of Graphs}
\label{ss:stdforms}

Standard models for locally finite, infinite graphs were introduced in \cite{ayala1990proper} and are used in \cite{AB2021}.  Our arguments do not require graphs be standard models. In particular we do not require the underlying tree to be binary, and sometimes we introduce artificial vertices of valence two. We will instead use graphs in \emph{standard form}, defined as follows. \begin{DEF}
    A locally finite graph, $\G$, is in \textbf{standard form} if $\G$ is a tree with loops attached at some of the vertices. We endow $\G$ with the path metric that assigns each edge length $1$.
\end{DEF} Standard form is strictly weaker than standard model, that is, every standard model is a graph in standard form. Note that for a graph in standard form, the underlying tree is a spanning tree and it is unique. Another benefit of standard form is that it allows us to talk about the fundamental group in a very concrete way. Specifically, we can orient and enumerate the loops as $\{\alpha_i\}_{i \in I}$ for some $I\subset \Z_{\geq 0}$, and call the vertices to which they are incident $\{v_i\}_{i \in I}$. For a basepoint $x_0$ in the tree, let $a_i$ be the based loop resulting from pre- and post-concatenating each loop $\a_{i}$ with the geodesic from $x_0$ to $v_{i}$. Then the collection $\{a_i\}_{i\in I}$ forms a basis for $\pi_1(\G,x_0)$.

 \Cref{sec:CBPMCG} focuses on graphs whose end spaces contain only one end accumulated by loops. The simplest such graph is the Loch Ness Monster graph, which is the graph with exactly one end and that end is accumulated by loops.  We named this graph in analogy with the Loch Ness Monster surface, which has a single end that is accumulated by genus. The next simplest class of graphs are those whose end space contains finitely many points, one of which is accumulated by loops. We call these graphs \emph{Hungry Loch Ness Monsters,} and include \Cref{fig:hungry} to demonstrate why. 
\begin{figure}[ht!]
	    \centering
	    \includegraphics[width=.8\textwidth]{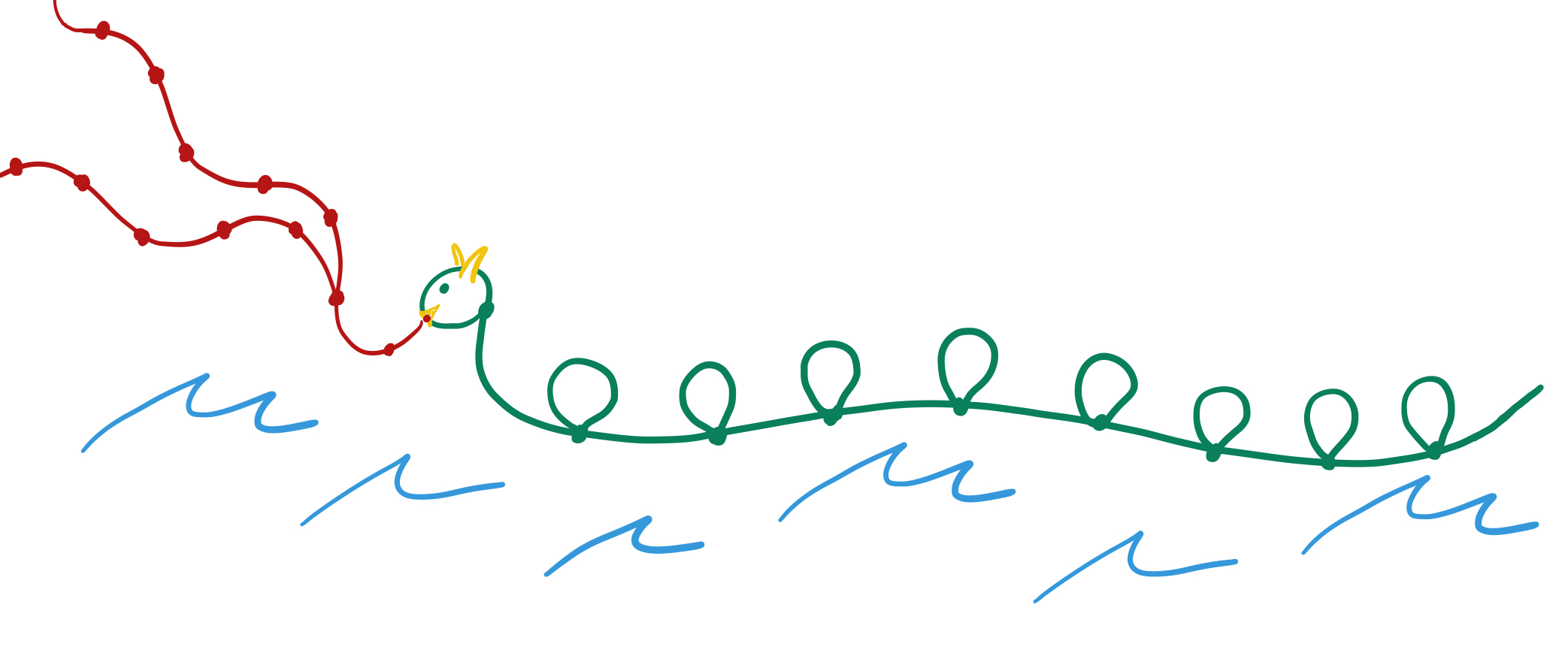}
		    \caption{A Hungry Loch Ness Monster graph, with two tongues.}
	    \label{fig:hungry}
\end{figure}

Let $\G_N$ for $N \in \Z_{\ge 0}$ refer to the locally finite graph with $|E_{\ell}|=1$ and $|E| = N+1$.  Let $\G_{\infty}$ refer to the locally finite graph with $|E_{\ell}|=1$, $|E|=\infty$ but $E \setminus E_{\ell}$ having no accumulation points. We call the graph $\Gamma_{\infty}$ the \emph{Millipede Monster graph} (see \Cref{fig:lochness}) and following the above, $\Gamma_{0}$ is the Loch Ness Monster graph and $\Gamma_{N}$ for $N \in \Z_{>0}$ are the Hungry Loch Ness Monster graphs. We also note that the core graph of \emph{any} graph $\Gamma$ with $|E_{\ell}(\G)|=1$ is properly homotopic to $\Gamma_{0}$. 

\begin{figure}[ht!]
	    \centering
	    \def\svgwidth{.8\textwidth}
		    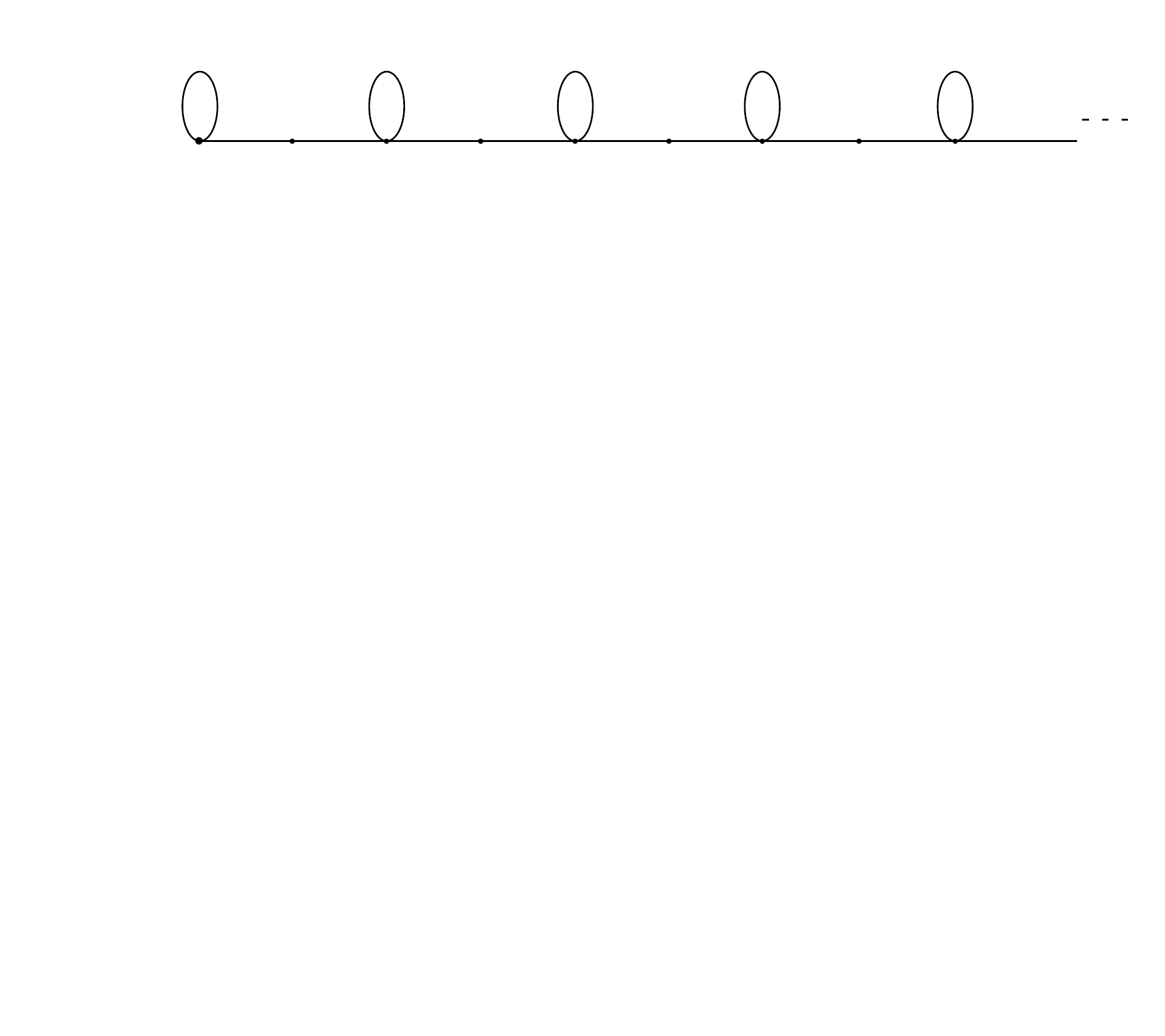
		    \caption{Standard forms and labels for the Loch Ness Monster graph ($\Gamma_{0}$), the Hungry Loch Ness Monster graph with $N$ rays attached ($\Gamma_{N}$), and the Millipede Monster graph ($\Gamma_{\infty}$).}
	    \label{fig:lochness}
\end{figure}

We use the vertex labeling $\{v_i\},\{w_i\}$ from \Cref{fig:lochness} to introduce the following notation for any graph $\Gamma$ in standard form with $|E_{\ell}(\G)|=1$. Note that if $\Gamma \neq \Gamma_{N}$ for some $N$ then we are only labeling $\Gamma_{c} \subset \G$ following the labeling on $\Gamma_{0}$. The notation $(v_{i},v_{j})$ is used to designate the geodesic in $\G_c$ connecting $v_{i}$ and $v_{j}$. The notation $[v_{i},v_{j}]$ designates the subgraph consisting of $(v_{i},v_{j})$ together with the loops $\alpha_{k}$ for all $k$ between $i$ and $j$, inclusive of $i$ and $j$. We can replace $v_i$ with $w_i$ and still use parenthesis to indicate the line segment in the core graph and closed brackets to include any loops which are incident to the line segment. We use $A_{i,j}$ to denote the free factor of $\pi_1(\G,x_0)$ (for any basepoint) coming from $[v_i,v_j]$, that is $A_{i,j}=\<a_k\>_{k=i}^{j}$.

\subsection{Loop Swaps} \label{ssec:loopswaps}
We will make ample use of a specific class of maps that swap sets of loops. First, we define them explicitly on the graphs $\G_N$ with $N \in \Z_{\ge 0}$ equipped with the path metric. Then we extend the definition to the millipede monster graph $\G_\infty$.

\begin{DEF}
    Given a triple $(n,m_{1},m_{2}) \in (\Z_{\geq 0})^3$ satisfying $m_{2}-m_{1} \ge n$ we define the \textbf{loop swap} determined by the triple $(n,m_{1},m_{2})$, denoted by $\cL(n,m_{1},m_{2})$, to be the map which swaps the $n$ loops starting at $v_{m_{1}}$ with the $n$ loops starting at $v_{m_{2}}$. That is, $\cL(n,m_{1},m_{2})$ is the map that interchanges $[v_{m_{1}},v_{m_{1}+n-1}]$ and $[v_{m_{2}},v_{m_{2}+n-1}]$ isometrically, stretches the following edges to the following paths,
    \begin{align*}
        (w_{m_{1}-1},v_{m_{1}}) &\mapsto (w_{m_{1}-1},v_{m_{2}}) \\
        (v_{m_{1}+n-1},w_{m_{1}+n-1}) &\mapsto (v_{m_{2}+n-1},w_{m_{1}+n-1}) \\
        (w_{m_{2}-1},v_{m_{2}}) &\mapsto (w_{m_{2}-1},v_{m_{1}}) \\
        (v_{m_{2}+n-1},w_{m_{2}+n-1}) &\mapsto (v_{m_{1}+n-1},w_{m_{2}+n-1}),
    \end{align*}
    and is the identity everywhere else. If the graph is the Loch Ness Monster and $m_{1}=1$ then $\cL(n,1,m_{2})$ is defined in the same way without the first stretch map, as there is no edge $(w_{0},v_{1})$ to stretch.

\end{DEF} 

\begin{figure}[ht!]
	    \centering
	    \def\svgwidth{.8\textwidth}
		    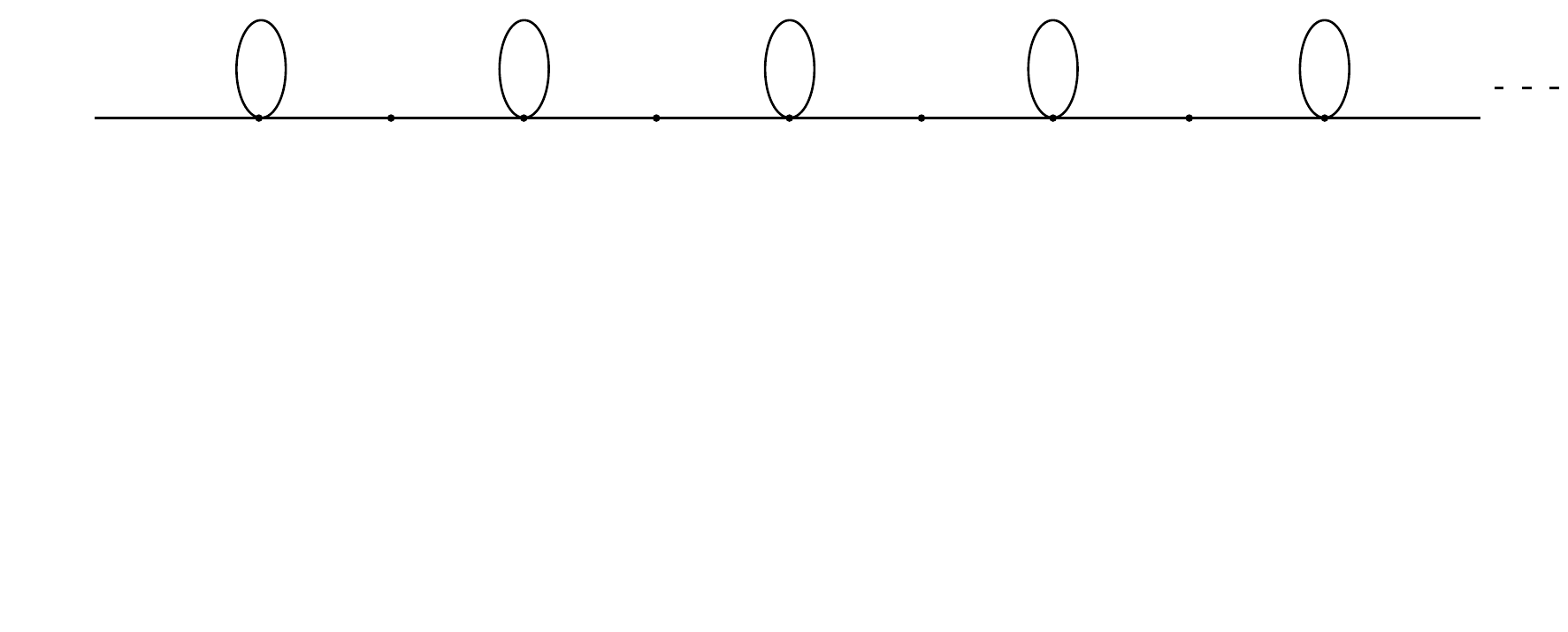
		    \caption{Example of the loop swap $\cL \left(2,1,4\right)$ on the Hungry Loch Ness Monster.} 
	    \label{fig:loopswapexample}
\end{figure} 

See \Cref{fig:loopswapexample} for an example of a loop swap on a Hungry Loch Ness Monster. We now make a few remarks about loops swaps.

\begin{RMK} 
\begin{enumerate}[(1)]
    \item 
    Loop swaps are always proper homotopy equivalences. 
    \item
    $\cL(n,m_{1},m_{2})^{2}$ is properly homotopic to the identity, so the corresponding mapping classes of loop swaps have order two. 
    \item
    The vertices $w_{m_{1}-1},w_{m_{1}+n-1}, w_{m_{2}-1},$ and $w_{m_{2}+n-1}$ are all fixed points of $\cL(n,m_{1},m_{2})$. 
    \item
    If $K = [v_{m_{1}},v_{m_{1}+n-1}]$, then $\cL(n,m_{1},m_{2})(K) \cap K$ is empty. 
    \item
    If $x_{0} \in \Gamma$ is fixed by some $\cL(n,m_{1},m_{2})$ then the induced map on $\pi_{1}(\Gamma,x_{0})$ is given by:
    \begin{align*}
        \cL(n,m_{1},m_{2})_{*}: a_{i}\mapsto
        \begin{cases}  a_{m_{2}+(i-m_{1})} &\text{ if } m_{1}\leq i < m_{1}+n, \\
        a_{m_{1}+(i-m_{2})} &\text{ if } m_{2} \leq i < m_{2}+n,\\ 
        a_{i} &\text{ otherwise.} 
        \end{cases}
    \end{align*}
\end{enumerate}
\end{RMK}

One can extend the definition of $\cL(n,m_{1},m_{2})$ to the millipede monster graph $\G_\infty$ by similarly interchanging the subgraphs $[v_{m_{1}},v_{m_{1}+n-1}]$ and $[v_{m_{2}},v_{m_{2}+n-1}]$ isometrically and now stretching subsegments of each of the $R_{i}$ that are incident to these subgraphs along the spanning tree of $\G_{\infty}$.

With some care one can also define loop swaps on any graph $\G$ with large enough rank, but we do not need them for the arguments in this paper.

\subsection{Word Maps}
 Let $\G$ be a graph with $\rk(\G) >0$. Consider its standard form as given in \Cref{ss:stdforms} and orient and enumerate the loops $\{\alpha_i\}_{i \in I}$ with $I \subset \Z_{\geq 0}$, and call the vertices at which they are based $\{v_i\}_{i \in I}$.
  Pick a base point $x \in \G$. Now identify $\pi_1(\G,x)$ with $\<a_i\>_{i\in I}$ where $a_i$ is the based loop that traverses $\alpha_i$. Let $w\in \pi_1(\G ,x)$ and write $w=a_{i_1}^{\pm}a_{i_2}^{\pm}\cdots a_{i_m}^{\pm}$. The \textbf{word path} associated to $w$ in $\G$ is the path that begins at $v_{i_1}$ and traverses $\alpha_{i_1}$ in the forward or backward orientation according to the sign of $a_{i_1}^{\pm}$ in $w$, then travels along the tree to $v_{i_2}$ and traverses $\alpha_{i_2}$ according to the sign and continues in this manner, ending at $v_{i_m}$.

\begin{DEF}\label{DEF:WordMap}
    Let $I\subset e$ be a connected subset of an edge $e\in \G$ and identify $I$ with the interval $[0,1]$, and further subdivide $I$ into $[0,\frac{1}{4}] \cup [\frac14, \frac34] \cup [\frac34, 1]$. If $I$ is contained in an edge of $\G \setminus \G_{c}$ then by convention we orient $I=[0,1]$ so that $0$ is farther from $\G_{c}$ than $1$. See \Cref{RMK:I_ConventionMatters} for why we follow this convention.
    
    We can define a \textbf{word map}, denoted by $\wm{w,I}$, supported on $I$ as follows: the interval $[0,\frac14]$ is mapped to the path in the tree from $0$ to $v_{i_1}$, the interval $[\frac14, \frac34]$ is mapped to the word path associated to $w$, and the interval $[\frac34,1]$ is mapped to the path in the tree from $v_{i_m}$ to $1$. The word map is the identity on the rest of $\G$, see \Cref{fig:wordexample}. 
\end{DEF}
    
\begin{figure}[ht!]
	    \centering
	    \def\svgwidth{.8\textwidth}
		    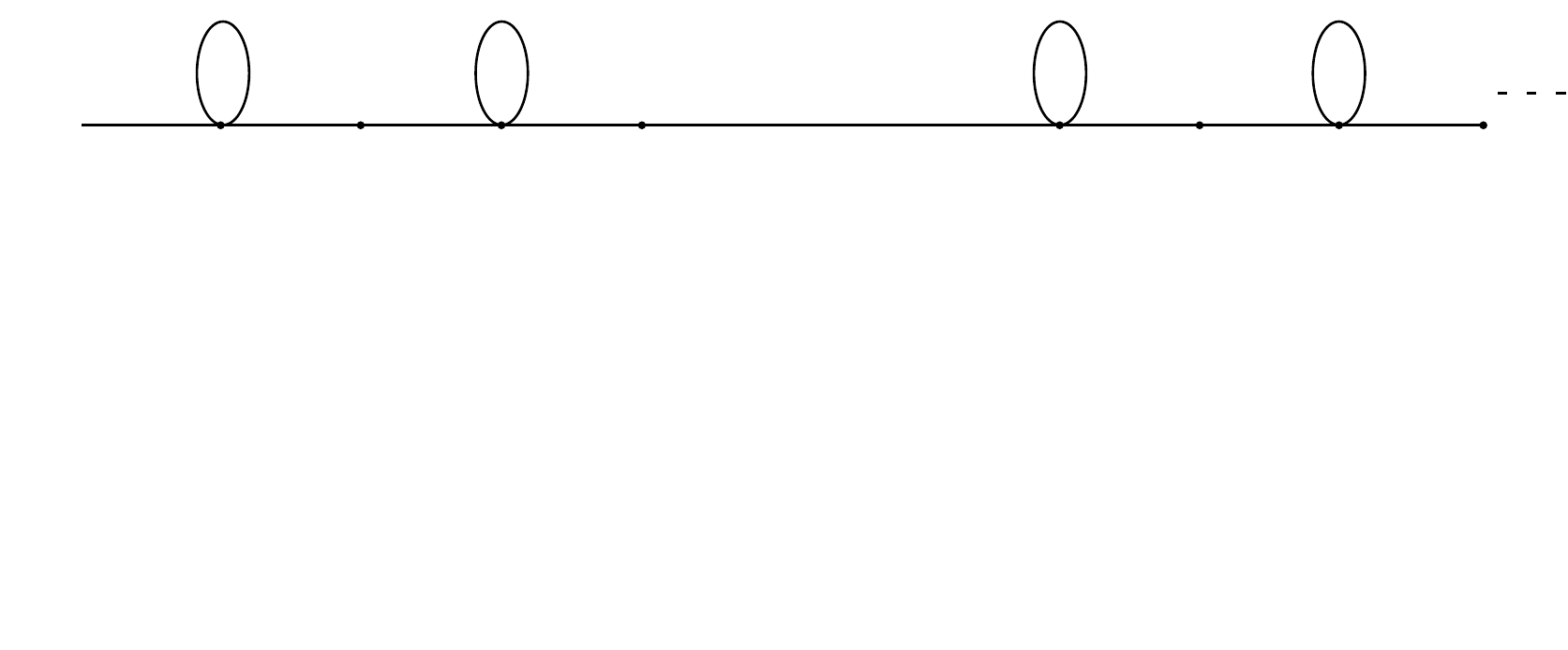
		    \caption{A word map $\wm{a_3a_4a_1,I}$} 
	    \label{fig:wordexample}
\end{figure}
    
To see that $\wm{w,I}$ is proper, note that $\wm{w,I}$ is compactly supported on $I$, and $\wm{w,I}(I)$ is also compact. To see that $\wm{w,I}$ is a non-trivial element of $\Map(\G)$ observe that it induces a non-trivial automorphism of at least one of $\pi_1(\G,0)$ and $\pi_1(\G,1)$. In fact, the induced map will be conjugation on a free factor of $\pi_1(\G,1)$.

In particular, if $\G=\G_N$ and $I \subset (v_j,v_{j+1})$ positively oriented, then the induced map on $\pi_1(\G,v_j)$ is the partial conjugation
\begin{align*}
   \left( \wm{w,I}\right)_* (a_i) &=
   \begin{cases} a_i &\text{ if } i\leq j, \\
   w a_i w\inv &\text{ if } i> j.
   \end{cases}
\end{align*}

The composition of two word maps on the same edge is again a word map. In fact, we have a nice composition rule when the words are supported on $\G \setminus \G_c$:

\begin{LEM}[Composition rule] \label{LEM:compositionTongues}
If $I$ is contained in an edge of $\G \setminus \G_{c}$ and $w_{1},w_{2}$ are two words in $\pi_{1}(\G,x)$, then 
\[
\wm{w_1,I}\circ \wm{w_2,I} \simeq \wm{w_1w_2,I},
\]
and the homotopy is proper.
\end{LEM}

\begin{proof}
    First, note that both maps are supported on $I$, so once we show that they are homotopic then they are properly homotopic.
    To show they are homotopic, we keep track of the image of $I=[0,\frac14]\cup[\frac14,\frac34]\cup[\frac34,1]$ under each map. Say $v_{i,1}$ and $v_{j,1}$ are the vertices in $\G_c$ incident to the loops corresponding to the first and last letters of $w_1$. Define $v_{i,2}$ and $v_{j,2}$ similarly using $w_2$. In this proof only, for vertices $x,y$ of $\G$ we denote by $[x,y]$ a geodesic from $x$ to $y$ in $\G$. Also, denote by $\cP(w_k)$ for $k=1,2$ the word path corresponding to $w_k$, starting at $v_{k,1}$, traversing the loops representing $w_k \in \pi_1(\G,x)$ and then ending at $v_{k,2}$. Then by definition of word maps, $\wm{w_1,I} \circ \wm{w_2,I}$ maps each subinterval of $I$ as:
    \begin{align*}
        &\left[0,\frac14\right]
        &&\xrightarrow{\wm{w_2,I}}
        \quad \left[0,v_{i,2}\right]
        &&\xrightarrow{\wm{w_1,I}}
        \quad\left[0,v_{i,1}\right]\cup \cP(w_1)\cup \left[v_{j,1},1\right]\cup\left[1,v_{i,2}\right], \\
        &\left[\frac14,\frac34\right]
        &&\xrightarrow{\wm{w_2,I}}
        \quad\cP(w_2)
        &&\xrightarrow{\wm{w_1,I}}
        \quad\cP(w_2), \\
        &\left[\frac34,1\right]
        &&\xrightarrow{\wm{w_2,I}}
        \quad\left[v_{j,2},1\right]
        &&\xrightarrow{\wm{w_1,I}}
        \quad\left[v_{j,2},1\right].
    \end{align*}
    Here to get the first line, we decomposed $[0,v_{i,2}]$ into $[0,1] \cup [1,v_{i,2}]$. Since the path $[v_{j,1},1] \cup [1,v_{i,2}]$ is homotopic to $[v_{j,1},v_{i,2}]$ we can homotope $\wm{w_1,I} \circ \wm{w_2,I}$ as:
    \[
        [0,1] \xrightarrow{\wm{w_1,I} \circ \wm{w_2,I}}
        \left[0,v_{i,1}\right]\cup \cP(w_1) \cup \left[v_{j,1},v_{i,2}\right] \cup \cP(w_2) \cup [v_{j,2},1],
    \]
    which is, after reparametrization, exactly the same as $\wm{w_1w_2,I}$ because $\cP(w_1w_2) = \cP(w_1) \cup \left[v_{j,1},v_{i,2}\right] \cup \cP(w_2)$. Therefore, we conclude $\wm{w_1,I} \circ \wm{w_2,I} \simeq \wm{w_1w_2,I}$.
\end{proof}

\begin{RMK}\label{RMK:I_ConventionMatters}
    Here to have \Cref{LEM:compositionTongues} it is crucial to have the convention of the orientation on $I=[0,1]$ so that $0$ is further from $\G_c$ than $1$. Otherwise, if we had the opposite orientation of $I$ so that $1$ is further from $\G_c$ than $0$, we would have the \emph{reversed} composition rule:
    \[
    \wm{w_1,I}\circ \wm{w_2,I} \simeq \wm{w_2w_1,I},
    \]
    this is because now $[\frac34,1]$ traverses over $I$ under the map $\wm{w_2,I)}$, whereas it was $[0,\frac14]$. Hence, with this orientation of $I$ in $\wm{w_1,I} \circ \wm{w_2,I}$ the word path $\cP(w_1)$ follows \emph{after} the word path $\cP(w_2)$.
\end{RMK}

The composition rule fails when $I$ is supported on a core graph of $\G$ and $\cP(w_2)$ traverses over $I$, even though the resulting map is still a word map supported on $I$. For example, consider $\wm{a_3a_4a_1,I}$ from \Cref{fig:wordexample}. Then by similar analysis as in \Cref{LEM:compositionTongues}, one can check that
\[
    \wm{a_2,I} \circ \wm{a_3a_4a_1,I} \simeq \wm{a_2a_3a_4a_2^{-1}a_1a_2,I}.
\]
For the same reason, word maps supported on disjoint intervals do not necessarily commute. Instead we introduce the notion of a \emph{multi-word map}, which we think of as doing multiple word maps simultaneously. It is defined as follows for disjoint $I,J\subset \G$:

\begin{align*}
    \left(\varphi_{(w_1,I)\sqcup (w_2,J)}\right)(x)&=
   \begin{cases} \wm{w_1,I}(x) &\text{ if } x\in I, \\
   \wm{w_2,J}(x) &\text{ if } x\in J, \\
   x &\text{ else.}
   \end{cases}
\end{align*} 

\begin{RMK} \label{rmk:RaymapsareWordmaps}
Word maps show up more often than one might initially expect. To see this, let $e$ be an edge between two vertices, $v_0$ and $v_1$, in a locally finite graph $\G$. If $\psi\in \PHE(\G)$ fixes each of $v_0$ and $v_1$, then $\psi\vert_{e}$ must be (properly) homotopic to a word map $\wm{w,I}$ with $I\subset e$. Note also that any proper homotopy equivalence supported on a ray, $R$, can be realized as a word map with $I\subset R$. Properness implies that any proper homotopy equivalence supported on a component of $\G \setminus \G_{c}$ can be realized as product of \emph{finitely} many word maps. Further, by homotoping to fix vertices, any compactly supported proper homotopy equivalence can be realized as a multi-word map supported on the interior of finitely many edges.
\end{RMK}

Manipulation of word maps is an essential tool in \Cref{sec:CBPMCG}, so we establish the subsequent lemmas. We say that an $n$-point set $\{x_i\} \subset \G$ \emph{span} an $n$-pod if one of the connected components of $\G-\cup\{x_i\}$ is a finite $n$-pod graph $K_{n,1}$ --- the complete $(n,1)$-bipartite graph --- whose boundary points are exactly $\{x_i\}$. \Cref{LEM:push} tells us how to homotope the support of a word map past a vertex of $\G$, as shown in \Cref{fig:push}.

\begin{figure}[h]
    \centering
    \includegraphics[width=0.6\textwidth]{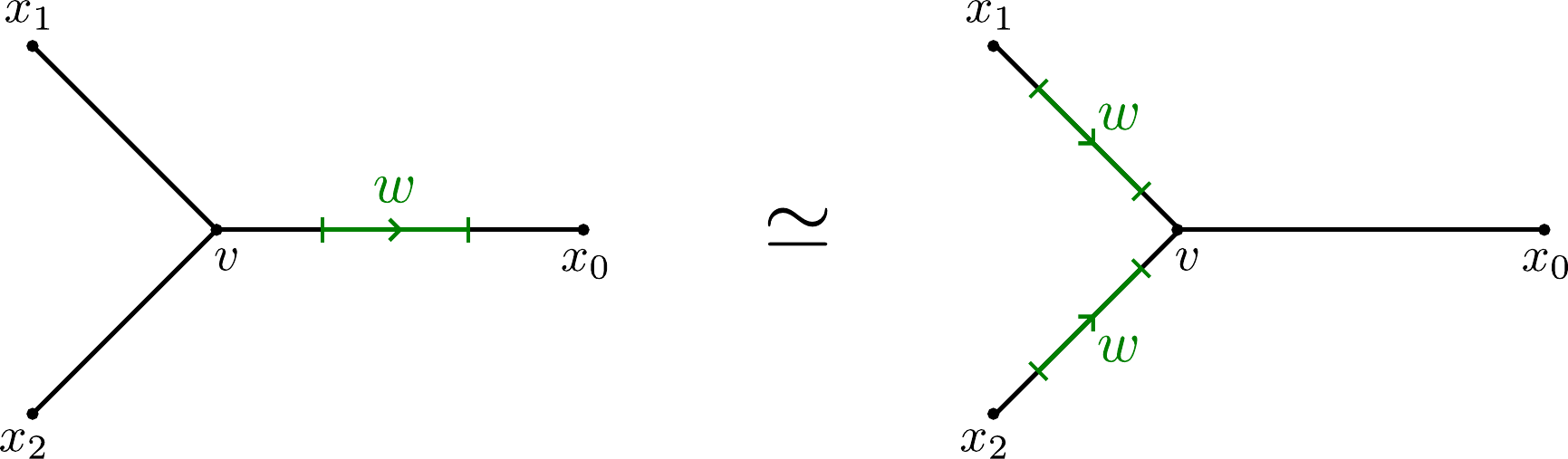}
    \caption{How to ``push'' a word map past a vertex $v$.}
    \label{fig:push}
\end{figure}

\begin{LEM}\label{LEM:push}
    Let $\{x_i\}_{i=0}^{n-1}$ span a finite $n$-pod in $\G$ and call the central vertex $v$. Let $I_0$ be an interval in $(v,x_0)$ and $I_i$ be an interval in $(x_i,v)$ for $1\leq i \leq n-1.$ Then for any $w\in \pi_1(\G)$ we have, \[\wm{w,I_0}\simeq \varphi_{(w,I_1)\sqcup (w,I_2)\sqcup \dots \sqcup (w,I_{n-1})}  \] and the homotopy is proper.
\end{LEM}

\begin{proof}
    First assume $n=3$, as in \Cref{fig:push}. Observe that $\varphi_{(w,I_1)\sqcup (w,I_2)}$ and $\wm{w,I_0}$ induce the same maps on $\pi_1(\G,x_i)$ for each of $i=0,1$ and $2$. Thus, \Cref{LEM:multiplebps} tells us $\wm{w,I_0} \simeq \varphi_{(w,I_1)\sqcup (w,I_2)}$. Because the support of the homotopy is the tripod, it is proper.  
    Now to see that the lemma holds for larger $n$, one can either blow up a valence $n$ vertex into a sequence of valence $3$ vertices, or apply \Cref{LEM:multiplebps} directly with $n$ basepoints. 
\end{proof}

\begin{LEM}\label{LEM:RaysLoops}
    Let $\G=\G_N$ with $N\in \Z_{\ge 0} \cup \{\infty\}$. If $u\in \PMap(\Gamma)$ is compactly supported, then $u$ can be homotoped to have support on the loops and rays of $\G$. That is, its support will be disjoint from the spanning tree of the core graph $\G_{c}$. Furthermore, we can write $u=u_{R_m}\circ \dots \circ u_{R_1}\circ u_{\ell}$ where $u_{\ell}$ is supported on the loops, $u_{R_i}$ is supported on the $i$th ray, and $m\leq N$ is finite.
\end{LEM}

\begin{proof}
    First homotope $u$ to fix the vertices. Because $u$ is compactly supported, this homotopy is proper.
    Let $\Delta$ be a compact subgraph of $\G_N$ on which $u$ is totally supported. If necessary, expand $\Delta$ so that it is connected, contains $[v_1,v_n]$ for some $n$, and contains a segment of every ray incident to $[v_1,v_n]$. Note that if $\G=\G_{\infty}$, then $\Delta$ will intersect at most finitely many of the rays, $R_1,\dots R_m$; otherwise, $m=N$.   
    
    Now because $u$ fixes the vertices it must be acting as simultaneous word maps on each edge of $\Delta$.  Applying \Cref{LEM:push} we can push these word maps off of the spanning tree $(v_1,v_n)$ (or $(w_0,v_n)$ if $\G=\G_N$ for some $N \in \Z_{>0}$) in $\Delta$ and supported on the ray segments and loops of $\Delta$. See \Cref{fig:PushingToRaysAndLoops} for the illustration of pushing word maps when $\G=\G_3$.
    
    \begin{figure}[ht!]
    \centering
        \includegraphics[width=.55\textwidth]{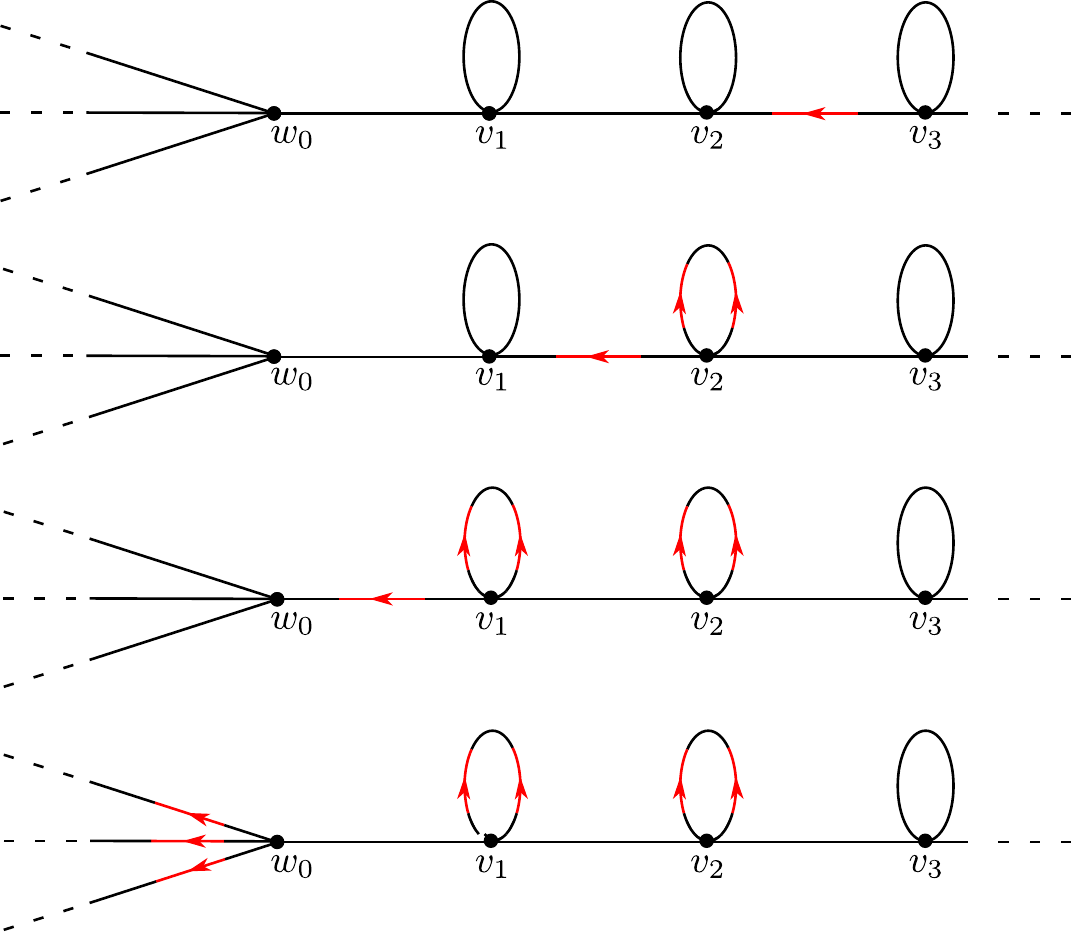}
        \caption{Pushing a word map on $(v_2,v_3)$ to be supported on the loops and rays in $\G=\G_3$.}
        \label{fig:PushingToRaysAndLoops}
    \end{figure}
    
    Thus, we can choose $I_i\subset R_i$ and pairs of intervals $J_i,J_i'\subset \a_i$ and write
    \[u=\varphi_{(w_1,I_1)\sqcup \dots \sqcup (w_m,I_m) \sqcup (\hat{w}_1,J_1)\sqcup (\hat{w}_1,J_1') \dots \sqcup (\hat{w}_n,J_n) \sqcup (\hat{w}_n,J_n')} \] for some $w_i,\hat{w}_i\in A_{1,n}$.  Let $u_{R_i}=\wm{w_i,I_i}$ and $u_{\ell}=\varphi_{(\hat{w}_1,J_1)\sqcup \dots \sqcup (\hat{w}_n,J_n')}$ . To see that $u$ splits as the desired composition, note that each interval $I_i\subset R_i$ is disjoint from every word path of $w_j$ and $\hat{w}_j$. 
\end{proof}

In particular, \Cref{LEM:RaysLoops} tells us that if $u \in \PMap(\G_{0})$ is compactly supported, then there is a homotopy representative of $u$ which is supported only on the loops. 

For any word map supported outside of the core graph of $\Gamma$, we can also compute the effect of conjugating the map by a mapping class totally supported on the core graph of $\Gamma$.

\begin{LEM}[Conjugation Rule] \label{LEM:conjugatetongue}
    Suppose $\G$ is a locally finite, infinite graph with nonempty complement of the core graph. 
    Let $\wm{w,I} \in \PMap(\G)$ be a word map supported outside the core graph $\G_c$, and $\psi \in \PMap(\G)$ be totally supported on a compact subgraph of the core graph. Then 
    \begin{align*}
        \psi \circ \wm{w,I} \circ \psi^{-1} = \wm{\psi_{*}(w),I}.
    \end{align*}
\end{LEM}

\begin{proof}
Pick a compact set $K\subset \G$ so that both $\wm{w,I}$ and $\psi$ are totally supported on $K$.
In particular, we have $I \subset K$. Since $I = [0,1]$ is outside the core graph, we may trim $K$ a bit such that the end point $0$ of $I$ lies on the boundary $\partial K:=K \cap \overline{\G \setminus K}$. Then label the boundary points $\partial K =\{x_0:=0,\ x_1,\ x_2,\ldots,x_n\}$, which will serve as base points of the fundamental group of $\G$. 

Observe that both $\psi \circ \wm{w,I} \circ \psi^{-1}$ and $\wm{\psi_*(w),I}$ fix each $x_i \in \partial K$, by choice of $K$. Hence, by \Cref{LEM:multiplebps}, it suffices to show that $\psi \circ \wm{w,I} \circ \psi^{-1}$ and $\wm{\psi_*(w),I}$ induce the same automorphisms on $\pi_1(\G,x_i)$ for each $i\in\{ 0, 1,\ldots,n\}$ to conclude the proof.

For each such $i\neq 0$, and any word $w\in \pi_1(\G)$, observe that $(\wm{w,I})_*: \pi_1(\G,x_i) \to \pi_1(\G,x_i)$ is just the identity map, as the geodesic from $x_i$ to the core graph is disjoint from $I$, and $I$ is also disjoint from the core graph. Therefore, on $\pi_1(\G,x_i)$:
\[
    (\psi \circ \wm{w,I} \circ \psi^{-1})_* = \psi_* \circ \psi_*^{-1} = \Id = \wm{\psi_*(w),I}.
\]

Now assume $i=0$. Recall the observation made earlier in this subsection that a word map induces the conjugation by the word associated to the map. Namely, we have for each generator $a_j \in \pi_1(\G,x_0)$:
\begin{align*}
    (\psi \circ \wm{w,I} \circ \psi^{-1})_*(a_j) &= \psi_*\left(w^{-1}\psi_*^{-1}(a_j)w\right) \\
    &= \psi_*(w)^{-1}a_j\psi_*(w) = \left(\wm{\psi_*(w),I}\right)_*(a_j),
\end{align*}
which concludes the proof.
\end{proof}

\subsection{Loop Shifts}
\label{ss:loopshifts}

Loop shifts are the graph equivalent of handle shifts on surfaces, which were introduced by Patel and Vlamis in \cite{PatelVlamis}. Let $\Lambda$ be the graph in standard form with exactly two ends, each of which are accumulated by loops, as in \Cref{fig:laddergraph}. The simplest example of a loop shift is the right translation of $\Lambda$ by one loop. The graph $\Lambda$ also admits loop shifts which omit some loops from their support. It takes more care to define these loops shifts, as well as loop shifts on more general graphs.

\begin{figure}[h]
	    \centering
	    \def\svgwidth{.8\textwidth}
		    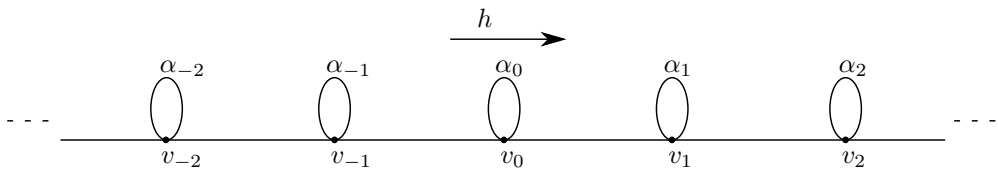
		    \caption{The graph $\Lambda$ on which we first define a loop shift.} 
	    \label{fig:laddergraph}
\end{figure}

Let $\G$ be any graph in standard form with at least two ends accumulated by loops, endowed with the path metric. Pick two distinct ends $e_{-}, e_{+}\in E_{\ell}(\G)$. There is a unique oriented line, $L$, in the spanning tree of $\G$ between these ends, and there are countably many loops of $\Gamma$ incident to $L$.  For any infinite subset of these loops, $\mathcal{A}$, which accumulates onto each of $e_{-}$, $e_{+}$, there is a natural well-ordered bijection between the loops of $\mathcal{A}$ and $\Z$. So, enumerate the loops $\mathcal{A}=\{\alpha_i\}_{i\in \Z}$, and call the vertices to which they are incident ${v_i}$. We use parenthesis to denote the geodesic between two points in $\G$, as in \Cref{ss:stdforms}.
\begin{DEF}
    The \textbf{loop shift} $h\in \PMap(\G)$, associated to $\mathcal{A}$, is constructed piecewise as follows. \begin{enumerate}
        \item $h\vert_{L\cup \mathcal{A}}=\begin{cases}
        v_i &\mapsto v_{i+1} \\
        \alpha_i &\mapsto \alpha_{i+1} \\
        (v_i,v_{i+1}) &\mapsto (v_{i+1},v_{i+2}).
        \end{cases}$
        \item Choose $\epsilon < \frac12$, and let $N_{\epsilon}(L\cup \mathcal{A})$ be the $\epsilon$-neighborhood of $L\cup \mathcal{A}$; define \[h\vert_{\G \setminus N_{\e}(L\cup \mathcal{A})}=\Id.\]
        \item Now $N_{\e}(L\cup \A)\setminus (L\cup \A)$ is a disjoint union of $\epsilon$-intervals incident to $L$. For each such interval $I=(x,y)$ with $y\in L$, we define \[h((x,y))=(x,h(y)),\] as $h(y)$ was defined in the first step.
    \end{enumerate}
    We also refer to the two ends $e_{-}$ and $e_{+}$ as $h_{-}$ and $h_{+}$. 
\end{DEF}
Note that different choices of $\epsilon<\frac12$ result in properly homotopic loop shifts, so they are the same element of $\Map(\G)$.

\section{Graphs of Infinite Rank with CB Pure Mapping Class Groups}
\label{sec:CBPMCG}
In this section we prove the following.

\begin{THM} \label{THM:oneendcb}
    Let $\Gamma$ be a locally finite graph with exactly one end accumulated by loops. If $E(\Gamma) \setminus E_{\ell}(\Gamma)$ does not contain an accumulation point, $\PMap(\Gamma)$ is CB. 
\end{THM}

We will prove this theorem in steps, beginning with the simplest such graphs and increasing in complexity at each step. First, in \Cref{ssec:lochness}, the Loch Ness Monster graph will be treated, then in \Cref{ss:hungrylochness} we consider the Hungry Loch Ness Monster graphs. Finally, when $\G$ has infinite end space, the only case where $E(\G) \setminus E_\ell(\G)$ has no accumulation point is when $E(\G)$ has countably many ends with the unique accumulation point of $E(\G)$ coinciding with the point in $E_{\ell}(\G)$. This graph, unique up to proper homotopy equivalence, is the Millipede Monster graph.

Each proof in this section emulates the proof that surfaces with self-similar end space have coarsely bounded mapping class groups, \cite[Proposition 3.1]{mann2022large} due to Mann and Rafi. The same method can be used to see that $S_{\infty}$, the group of bijections of a discrete countable set (with possibly infinite support) equipped with the compact-open topology (or equivalently the point-open topology, since the underlying set is discrete), is coarsely bounded. We present this proof here as a warm up, although this fact can also be found in \cite[Example 9.14]{roelcke1981uniform}. There the authors actually show that $S_{\infty}$ is Roelcke precompact, a stronger condition. 

\begin{PROP} \label{THM:permutationCB}
    The topological group of bijections on a countable set, $S_{\infty}$, is CB.
\end{PROP}

\begin{proof}
    We fix $\Z^+$ as our countable set. We first note that the compact-open topology on $S_{\infty}$ has a neighborhood basis about the identity given by sets of the form $\cV_{K} = \{\sigma \vert \sigma(i) = i \text{ for all } i \in K\}$ where $K$ is a finite subset of $\Z^+$. We will use Rosendal's criterion to see that $S_\infty$ is CB. 
    
    Let $\cU$ be a neighborhood of the identity. We can find a basis element $\cV_{K} \subset \cU$ where $K$ has the form $K=\{1,\ldots, n\}$. Let 
    \begin{align*}
        f = (1, n+1) (2,n+2) \cdots (n,2n),
    \end{align*}
    and set $\cF = \{f\}$. We will show that $S_{\infty} = (\cF\cV_{K})^{3}$, which implies $S_{\infty}=(\cF\cU)^3$. 
    
    For any $\phi \in S_{\infty}$, let $u$ be a finite permutation such that 
    \begin{align*}
        u(\phi(i)) = i, \;\; \text{ for all }1 \leq i \leq n.
    \end{align*}
    Then $u\phi \in \cV_{K}$. Let $m = \max\{\supp(u) \cup \{2n\}\}$ and define 
    \begin{align*}
        g = (n+1,m+1)(n+2,m+2)\cdots (2n,m+n).
    \end{align*}
    By our choice of $m$, we have that $g \in \cV_{K}$. We can also check that $fgugf \in \cV_{K}$. Indeed, 
    \begin{align*}
        fgugf(i) = fgu(i+n+m) = fg(i+n+m) = i,
    \end{align*}
    for all $i \in \{1,\ldots,n\}$. Rearranging these inclusions, we conclude that $\phi \in (\cF\cV_{K})^{3}$. 
\end{proof}

\subsection{Notation and a few Lemmas}
\label{ss:CB_notation_lemmas}

Throughout this section we consider the graphs $\Gamma_N$ with $N\in \Z_{\ge 0} \cup \{\infty\}$. We will refer to the graphs and labels from \Cref{fig:lochness} and use the notation established in \Cref{SEC:Elements}. Each time we are given a compact set $K\subset \Gamma$ we first expand it to contain $[v_1,v_n]$ and we set $f$ to be the loop swap $\mathcal{L}(n,1,n+1)$ on the appropriate graph $\G_N$. 
\begin{RMK}
   For any set $K=[v_1,v_n]$, any element in $\cV_K$ must be the identity on each tree component of $\G_N \setminus K$. So, expanding $K$ to include any of portion of the rays that it disconnects from $\G_c$ does not change the set $\cV_K$. This also simplifies the definition of the sets $\cV_{K}$. Namely, for such $\Gamma$ and $K$, any $\phi \in \PMap(\G)$ is contained in $\cV_{K}$ if and only if $\phi$ is homotopic to some $\phi'$ such that $\phi'\vert_{K} =\Id_{K}$ and $\phi'(\G\setminus K) \cap K = \emptyset$. Note that this is doable exactly when $\phi \in \PMap(\G).$
\end{RMK}

We next prove two lemmas that will be used throughout this section. Note that these lemmas hold for any general graph with only a single end accumulated by loops, not just the examples mentioned above. The first lemma shows that we can ``finitely approximate'' an inverse, as in the warm-up above.

\begin{LEM}[Folding to Approximate] \label{LEM:folding}
    Let $\G$ be a graph with $|E_{\ell}(\G)| =1$ and $K = [v_{1},v_{n}] \subset \Gamma$. For $\phi \in \PMap(\Gamma)$, there exists some $u \in \PMap(\Gamma)$ and $K' \subset \Gamma$ with $K' \supset K$ such that the following holds. 
    \begin{enumerate}[(a)]
        \item
            $u$ is totally supported on $K'$, i.e., $u(K') = K'$ and $u|_{\Gamma \setminus K'} = \Id_{\Gamma \setminus K'}$,
        \item
            $u\phi \in \cV_{K}$, i.e., $u\phi$ is properly homotopic to a map $v$ such that 
            \begin{enumerate}[(i)]
            \item
                $v\vert_{K} = \Id_{K}$, 
            \item
                $v(\Gamma \setminus K) \cap K = \emptyset$.
            \end{enumerate}
    \end{enumerate}
\end{LEM}

\begin{proof}
    Let $\psi$ be a proper homotopy inverse of $\phi$. We will make use of Stallings folds to see that we can approximate $\psi$ on $K$. After a proper homotopy, we may assume $\psi$ maps vertices to vertices. We next modify $\psi$ via a proper homotopy so that $\psi^{-1}(x)$ is a totally disconnected set for every $x \in \Gamma$. To do so, subdivide every edge that is collapsed and modify $\psi$ (via a proper homotopy) to send this new midpoint to a vertex adjacent to the original image of the edge. Note that the two endpoints of the edge will still have the same image, but the edge itself will traverse over an entire edge in the target twice. If there is a subinterval of an edge that is collapsed (as opposed to an entire edge) we can perform the same proper homotopy on the subinterval. Note that this modification does not change the fact that $\psi$ maps vertices to vertices.
    
    The complete pre-images of vertices under $\psi$ give us a subdivision $\Gamma_S$ of $\Gamma$. Note that $\G_{S}$ is locally finite since $\psi$ is proper. We will show that we can factor $\psi$ through finitely many folds so that the resulting map is injective on the image of $K$.

    \begin{CLAIM}\label{claim:vertex_id}
        If $\psi(v)=\psi(w)$ for some vertices $v,w$ of $\Gamma_{S}$, then $v$ and $w$ can be identified after finitely many folds.
    \end{CLAIM}
    \begin{proof}[Proof of \Cref{claim:vertex_id}]
Consider a path $\gamma$ from $v$ to $w$ in $\Gamma_S$. Since $\psi(v)=\psi(w)$, the image $\psi(\gamma)$ represents an element in $\pi_1(\G_S,\psi(v))$. Since $\psi_*$ is a $\pi_1$-isomorphism, there exists $[\alpha] \in \pi_1(\G_S,v)$ such that $\psi_*([\alpha]) = [\psi(\gamma)]$. Thus $[\psi(\alpha^{-1} * \gamma)]=1$, where here $\alpha^{-1} * \gamma$ is the concatenated path from $v$ to $w$, first following $\alpha^{-1}$ and then $\gamma$. The existence of a nullhomotopy of $\psi(\alpha^{-1}*\gamma)$ in $\Gamma_S$ suggests that we can fold the path $\gamma$ to wrap around $\alpha$ to identify $w$ with $v$.
    \renewcommand{\qedsymbol}{$\triangle$}
    \end{proof}

    \begin{CLAIM}\label{claim:edge_id}
        If $\psi(e)=\psi(e')$ for distinct edges $e,e'$ of $\Gamma_S$, then $e$ and $e'$ can be identified after finitely many folds.
    \end{CLAIM}
    \begin{proof}[Proof of \Cref{claim:edge_id}]
    By the previous claim, we may assume the two edges $e$, $e'$ share a vertex in $\Gamma_S$. Since $\psi$ induces a $\pi_1$-isomorphism, the two edges cannot share \textit{both} vertices, otherwise $\psi$ collapses the nontrivial loop bounded by $e,e'$. Hence, there can be only one vertex that $e,e'$ share, from which we can perform a Type $1$ fold to identify $e=e'$.
    \renewcommand{\qedsymbol}{$\triangle$}
    \end{proof}

    Apply these two claims to every edge and vertex in $K$ as well as those mapped into $K\cup [v_n,w_n]$ via $\psi$. Let $F:\Gamma_{S} \rightarrow \Gamma_{S}'$ be the product of all of these folds, where $\Gamma_{S}'$ is the resulting graph. Note that any word maps on trees which are adjacent to $K$ will be completely folded in this procedure. Because each fold is only defined on two edges, $\Gamma_{S}$ contains some connected compact subgraph $K'$ containing $K$ such that $F\vert_{\Gamma_{S} \setminus K'}$ is the identity map. In other words, $K'$ is the part where the folds happen. Now $\psi$ factors as $\psi = h \circ F$ where $h:\Gamma_{S}' \rightarrow \G_{S}$ is injective on $F(K)$. Since $K'$ witnesses all the folds of edges mapped into $K$, we have that $h(\Gamma_{S}' \setminus F(K')) \cap K = \emptyset$. Our original map $\psi$ was a (surjective) proper homotopy equivalence, so $h$ is also surjective onto $K$. Thus, $h$ restricts to a graph isomorphism from $F(K)$ to $K$.
    
    Next, we define a homotopy equivalence $\sigma: \G_{S}' \rightarrow \G_{S}$. First define $\sigma$ on $F(K) \subset \G_{S}'$ to be $h\vert_{F(K)}$, which is a graph isomorphism onto $K$. At most one connected component of $\overline{F(K') \setminus F(K)}$ has positive rank (the component containing $w_n$), while the others are compact sub-trees of the trees incident to $K$. Define $\sigma$ on the tree components to be the identity.  Define $\sigma$ on the positive rank component to be any homotopy equivalence onto the ``next'' subgraph of equal rank to the right of $K$ in $\G_S$. That is, if $K = [v_{1},v_{n}]$ then $\sigma$ maps the positive rank component to $[v_{n+1},v_{n+m}]$ as a homotopy equivalence, where $m$ is the rank of $\overline{F(K') \setminus F(K)}$. Finally, define $\sigma$ on $\G'_{S} \setminus F(K')$ to be the identity map.
    
    Let $u = \sigma \circ F$. We have the commutative diagram in \Cref{fig:folding} which commutes up to homotopy.
    
    \begin{figure}[ht!]
    \makebox[\textwidth][c]{%
    \begin{overpic}[percent,width=.9\textwidth]{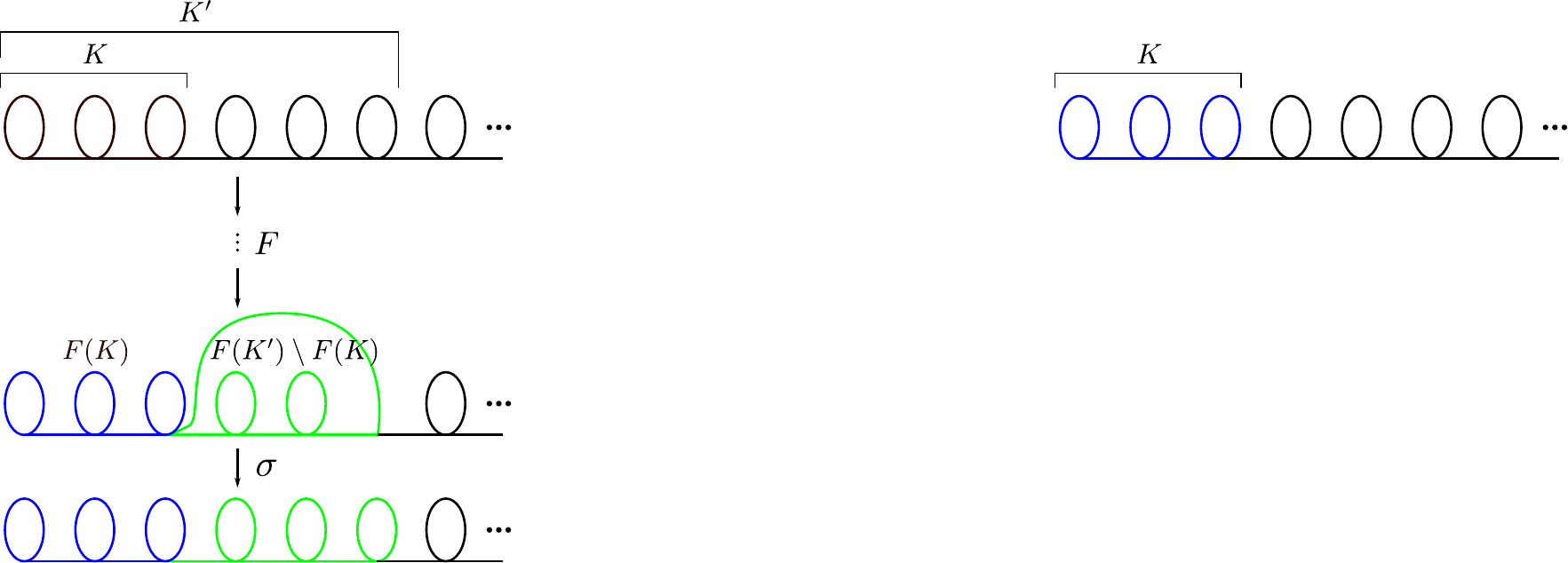}
    \put(-3,15){\parbox{.9\textwidth}{%
    \[
    \begin{tikzcd}[ampersand replacement=\&]
	{\G_S} \& \& {\G_S} \\
	\vdots \\
	{\G'_S} \\
	{\G_S}
	\arrow[from=2-1, to=3-1]
	\arrow["\sigma", from=3-1, to=4-1]
	\arrow["h", from=3-1, to=1-3]
	\arrow["\bar{h}", curve={height=-15pt},from=1-3, to=3-1]
	\arrow["{\psi}"', from=1-1, to=1-3]
	\arrow["{\phi}"', curve={height=10pt},from=1-3, to=1-1]
	\arrow[from=1-1, to=2-1]
	\arrow["F", shift left=1, draw=none, from=1-1, to=3-1]
	\arrow[curve={height=18pt}, dashed, from=1-1, to=4-1]
	\arrow["{u:=\sigma F}"'{pos=0.3}, curve={height=24pt}, draw=none, from=1-1, to=4-1]
    \end{tikzcd}
    \]
    }}
    \end{overpic}
    }
    \caption{Commutative diagram for the folds $F$ from $\psi$ and the homotopy equivalence $\sigma$.}
    \label{fig:folding}
    \end{figure}
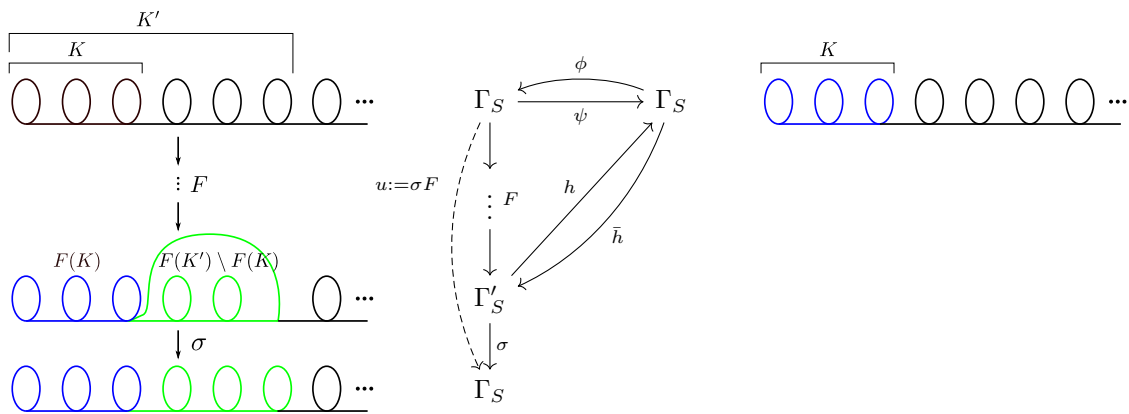
    We can now check that $u$ has the desired properties. By construction, $u\vert_{\Gamma_{S} \setminus K'} = \Id_{\Gamma_{S}\setminus K'}$ and $u(K') = K'$, so (a) is satisfied. Now $u \phi \simeq \sigma \bar{h}$, where $\bar{h}$ is a homotopy inverse of $h$. Since $h$ does not map anything from $\G_{S}' \setminus F(K')$ into $K$, we can pick $\bar{h}$ to map nothing from $\G_{S} \setminus K$ into $F(K')$. Because $h\vert_{F(K)}$ is a graph isomorphism onto $K$, we can choose $\bar{h}$ to be the inverse graph isomorphism from $K$ onto $F(K)$. Now we see that $\sigma \circ \bar{h}$ is exactly the identity on $K$, so that (b-i) is satisfied. Finally, $\bar{h}$ maps $\G_{S} \setminus K$ into $\G'_{S} \setminus F(K')$, and $\sigma$ maps this set back into $\G_{S} \setminus K' \subset \G_S \setminus K$, so that (b-ii) is also satisfied. 
\end{proof}

The previous lemma says that we can realize $\PMap(\Gamma)$ as the closure of the compactly supported mapping classes of $\G$. That is, letting $\PMapc(\G)$ be the subgroup of $\PMap(\G)$ consisting of all proper homotopy classes of proper homotopy equivalences having a compactly supported representative, then the previous lemma gives the following.

\begin{COR} \label{COR:closurecompactsup}
    Let $\G$ be a graph with $\lvert E_{\ell}\rvert=1$. Then $\overline{\PMapc(\G)} = \PMap(\G)$.
\end{COR} 

\begin{proof}
    Let $\phi\in \PMap(\G)$. Take a compact exhaustion $\{K_{i}\}$ of $\G$ with $K_{i}$  consisting of the subgraph $[v_{1},v_{i}]$ together with larger and larger portions of the trees extending to the other ends. We thus obtain via \Cref{LEM:folding} a sequence $\{u_{i}^{-1}\}$ of elements in $\PMapc(\G)$ which converges to $\phi$ in $\PMap(\G)$.
\end{proof}

Recall that we use the term \emph{loop} to mean a single edge whose end points are the same.

\begin{LEM}
    \label{lem:supp_on_loops}
    Let $\G$ be a graph with $\lvert E_{\ell}(\G)\rvert =1$, and let $K=[v_1,v_n]$. If $u \in \PMapc(\Gamma)$ is supported only on the loops of $\Gamma$, then $u \in (\cF \cV_{K})^{3}$, where $\cF = \{f\}$ with $f=\cL(n,1,n+1)$.
\end{LEM}

\begin{proof}
     Without loss of generality we can assume that $u$ is totally supported on $K' = [v_{1},v_{m}]$ where $m \geq 2n$. Let $g = \cL(n,n+1,m+1)$. Note that $g \in \cV_{K}$ and $(gf)(K) \cap K = \emptyset$. See \Cref{fig:lochnesssetup} for a schematic of the setup in the case of the Loch Ness Monster.
    
    \begin{figure}[ht!]
	    \centering
	    \def\svgwidth{.8\textwidth}
		    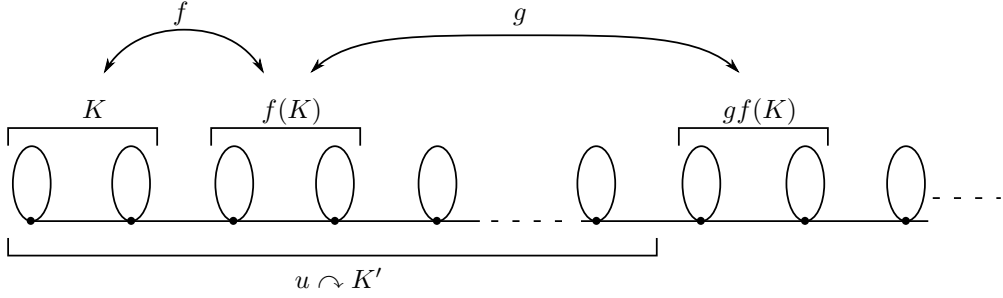
		    \caption{Setup of the loop swap maps for the Loch Ness Monster graph} 
	    \label{fig:lochnesssetup}
\end{figure}
    
    We claim that $\nu=(gf)^{-1}u(gf) \in \cV_{K}$. First note that $\nu$ is totally supported on $[v_{1},v_{m+n}]$. We seek to apply \Cref{LEM:multiplebps}. In order to apply this lemma we will need to consider fundamental groups with basepoints in each of the complementary components of $[v_{1},v_{m+n}]$. Pick a such a finite collection of basepoints and note that a basis for the fundamental group based at each one is given by pre- and post-concatenating each loop $\a_{i}$ with the unique geodesic from $v_{i}$ to the respective basepoint. Now, since $u$ is supported on the loops and $g$ and $f$ are loop swaps we see that $\nu$ symbolically induces the same map on the fundamental groups from the perspective of each of these basepoints. Therefore, we will slightly abuse notation and write $\pi_{1}(\Gamma)$ to refer to the fundamental group with any of these basepoints and $\{a_{i}\}_{i=1}^{\infty}$ to denote a free basis. 
    
    As in \Cref{ss:stdforms}, $A_{i,j}$ for $i\leq j$ denotes the free factor of $\pi_{1}(\Gamma)$ generated by the basis elements $\{a_{k}\}_{k=i}^{j}$. We claim that $\nu_{*}\vert_{A_{1,n}} = \Id\vert_{A_{1,n}}$ and $\nu_{*}(A_{n+1,m+n}) \subset A_{n+1,m+n}$. Note that $gf$ induces the following map on $\pi_{1}(\G)$.
    \begin{align*}
        (gf)_{*}: a_{i} \mapsto \begin{cases} a_{m+i} &\text{ if } 1\leq i \leq n, \\
        a_{i-n} &\text{ if } n+1\leq i \leq 2n,\\
        a_{i-(m-n)} &\text{ if } m+1 \leq i \leq m+n. \\ 
        a_{i} &\text{ otherwise.} 
        \end{cases}
    \end{align*} 
    In particular, $(gf)_{*}$ acts as the permutation on the free factors $A_{1,n} \rightarrow A_{m+1,m+n} \rightarrow A_{n+1,2n} \rightarrow A_{1,n}$, by sending ordered sets of generators to ordered sets of generators. Similarly, $(gf)_{*}^{-1}$ acts as the permutation $A_{1,n} \rightarrow A_{n+1,2n} \rightarrow A_{m+1,m+n} \rightarrow A_{1,n}$. 
    
    We also have that $u_{*}(a_{i}) = a_{i}$ for all $i> m$, and that $u_{*}(a_{j}) \in A_{1,m}$ for all $j\leq m$. Putting everything together, we have the following equalities for $j\leq n$. 
    \begin{align*}
        \nu_{*}(a_{j}) &= (gf)_{*}^{-1}u_{*}(gf)_{*}(a_{j}) \\
        &= (gf)_{*}^{-1}u_{*}(a_{m+j}) \\
        &= (gf)_{*}^{-1}(a_{m+j}) \\
        &= a_{j}.
    \end{align*}
    This shows that $\nu_{*}\vert_{A_{1,n}} = \Id\vert_{A_{1,n}}$ as desired. 
    
    Next we check that $\nu_{*}(a_i)\in A_{n+1,\infty}$ for $i>n$, which is equivalent to checking that $(u gf)_*(a_i) \in A_{1,m}*A_{m+n+1,\infty}$. The only generators that are mapped into $A_{m+1,m+n}$ by $u_{*}$ are exactly $a_{m+1},\dots,a_{m+n}$. Since we assumed that $i>n$ we have that $(gf)_{*}(a_{i}) \in A_{1,m}*A_{m+n+1,\infty}$ and we conclude that $\nu_{*}(a_{i}) \in A_{n+1,\infty}$ for $i>n$. Thus, we can apply \Cref{LEM:multiplebps} to see that $\nu\in \cV_{K}$. Then by rearranging and noting that $f^{-1}=f$ we obtain $u = gf\nu fg \in (\cF\cV_{K})^{3}$.
\end{proof}

\subsection{The Loch Ness Monster Graph}
\label{ssec:lochness} 

The Loch Ness Monster graph has exactly one end, so $\PMap(\G_0)=\Map(\G_0)$. The lemmas we have prepared in \Cref{ss:CB_notation_lemmas} and \Cref{SEC:Elements} are sufficient to prove that this group is coarsely bounded.

\begin{PROP} \label{THM:lochnessCB}
    $\Map(\G_0)$ is CB. 
\end{PROP}

\begin{proof}
    We will make use of Rosendal's criterion. Given any neighborhood of the identity in $\PMap(\G_{0})$, we pass to a basis element $\cV_{K}$ for some compact $K$, and then expand $K$ so that $K=[v_{1},v_{n}]$ for some $n\in \Z^+$. Let $\phi \in \Map(\G_0)$. Find a map $u$ using \Cref{LEM:folding}, such that $u\phi \in \cV_K$. Apply \Cref{LEM:RaysLoops} to write $u=u_{\ell}$ with $u_{\ell}$ supported on the loops of $\G_{0}$. Now we apply  \Cref{lem:supp_on_loops} to show that $fgugf \in \cV_K$. Then $u = gfhfg$ for some $h \in \cV_K$. Hence, $\phi \in (gfh^{-1}fg)\cV_K \subset (\cF\cV_K)^3$ for $\cF=\{f\}$, as $g \in \cV_K$.
\end{proof}

\subsection{The Hungry Loch Ness Monster Graphs}
\label{ss:hungrylochness}

Next we show that for $N\in \Z^+$, the group $\PMap(\G_N)$ is coarsely bounded. Recall that $\G_N$ denotes a Hungry Loch Ness Monster graph, as in \Cref{fig:lochness}. Unlike elements of $\PMap(\G_0)$, compactly supported elements of $\PMap(\G_N)$ may have support on the rays, so we start by developing a method to fit these maps into Rosendal's criterion. Recall that maps supported on rays can be homotoped to word maps as pointed out in \Cref{rmk:RaymapsareWordmaps}. Thus, on its own, \Cref{LEM:RayMap} says that the subgroup of $\PMap(\G_N)$ which consists of elements supported on rays, is coarsely bounded. 

\begin{LEM}\label{LEM:RayMap}
    Let $\G=\G_N$ for some $N\in \Z^+ \cup \{\infty\}$, and let $K=[v_1,v_n]$. Let  $\wm{w,I}$ be a word map with $I\subset R$ for any ray $R\subset \G$. Then we can realize $\wm{w,I}\in \left(\cF\cV_K\right)^5$, where $\cF=\{f, \wm{a_{n+1},I}^{\pm}\}$ and $f=\mathcal{L}(n,1,n+1)$.
\end{LEM}

\begin{proof}
    We first modify $\wm{w,I}$ to ensure that the word $w$ is a basis element in $F_{\infty} = \pi_{1}(\Gamma_{N})$. Freely reduce $w$ and let $m = \max\left\{\{i \vert a_{i} \text{ appears in } w\} \cup \{2n+1\} \right\}$. Set $h = \cL(1,n+1,m+1) \in \cV_{K}$. Then we define
    \begin{align*}
        \phi' = \wm{a_{n+1},I} h  \wm{w,I}  h = \wm{a_{n+1}w',I}
    \end{align*}
    where $w' = h_{*}(w)$. The final equality follows from \Cref{LEM:compositionTongues} and \Cref{LEM:conjugatetongue}. Note that $w'$ does not contain any instances of $a_{n+1}$. So, $a_{n+1}w'$, the word defining $\phi'$, only contains a single instance of $a_{n+1}$, and is thus a basis element for $F_{\infty}$. 
    
    Next we modify $\phi'$ again to get a word map $\phi''$ whose defining word is a basis element completely contained in $A_{n+1,\infty}$. That is, $\phi''$ does not hit any of the loops in $K$. Let $g = \cL(n,n+1,m+2) \in \cV_{K}$ and set
    \begin{align*}
        \phi'' &= f  g  \phi'  g  f \\ 
        &= \wm{(fg)_{*}(a_{n+1}w'),I} \\
        &= \wm{a_{m+2}w'',I}.
    \end{align*}
    where $w'' = (fg)_{*}(w') = (fgh)_{*}(w)$. Note that $(fg)_{*}$ only maps the basis elements $a_{m+2},\ldots,a_{m+n+1}$ into $A_{1,n}$ so that $w'' \in A_{n+1,\infty}$, by the choice of $m$. Notice this does not show $\phi'' \in \cV_K$ because the complementary component of $K$ may not be preserved by the word map $\phi''=\wm{a_{m+2}w'',I}$.
    
    Next we choose $\rho \in \PMap(\Gamma_{N})$ such that 
    \[
        \rho_{*} =
        \begin{cases}
        a_{m+2} \mapsto a_{m+2}w'' & \phantom{} \\
        a_{i} \mapsto a_{i} & \text{ for all } i \neq m+2. 
        \end{cases}
    \]
    Note that such a homotopy equivalence exists since $a_{m+2}w''$ is a basis element for $F_{\infty}$ and it can be taken to be proper since $\rho_{*}$ is the identity outside of the finite-rank free factor $A_{n+1,m+n+1}$ of $F_{\infty}$. This also shows that $\rho \in \cV_{K}$. 
    
    Finally, we conjugate $\phi''$ to get $\wm{a_{n+1},I}$, the word map in $\cF$:
    \begin{align*}
        g \rho^{-1} \phi''  \rho  g &= g \rho^{-1} \wm{a_{m+2}w'',I}  \rho  g \\
        &= g  \wm{a_{m+2},I}  g \\
        &= \wm{a_{n+1},I}.
    \end{align*}
    Therefore, after substituting and rearranging we have
    \begin{align*}
        \wm{w,I} =
        \underbrace{\vphantom{\wm{a_{n+1},I}^{-1}} h}_{\cV_{K}} 
        \underbrace{\vphantom{\wm{a_{n+1},I}^{-1}} \wm{a_{n+1},I}^{-1}}_{\cF} 
        \underbrace{\vphantom{\wm{a_{n+1},I}^{-1}} g}_{\cV_{K}} 
        \underbrace{\vphantom{\wm{a_{n+1},I}^{-1}} f}_{\cF} 
        \underbrace{\vphantom{\wm{a_{n+1},I}^{-1}} \rho g}_{\cV_{K}}
        \underbrace{\vphantom{\wm{a_{n+1},I}^{-1}} \wm{a_{n+1},I}}_{\cF}
        \underbrace{\vphantom{\wm{a_{n+1},I}^{-1}} g \rho^{-1}}_{\cV_{K}}
        \underbrace{\vphantom{\wm{a_{n+1},I}^{-1}} f}_{\cF}
        \underbrace{\vphantom{\wm{a_{n+1},I}^{-1}} g h}_{\cV_{K}}
        \in (\cF\cV_{K})^{5}. \qquad \qedhere
    \end{align*}
\end{proof}

\begin{PROP}
  \label{thm:HLNCB}
    For any $N\in \Z^+$, the group $\PMap(\G_{N})$ is CB.
\end{PROP}

\begin{proof}
    We will again use Rosendal's criterion. Given any open set $\cU$ first set $K=[v_1,v_n]$ so that $\cV_K \subset \cU$.  Let $f=\mathcal{L}(n,1,n+1)$, and choose intervals $I_i$ in each ray $R_i$. Now set $\cF=\{f, \wm{a_{n+1},I_1}^{\pm}, \ldots, \wm{a_{n+1},I_N}^{\pm} \}$, we will show that any $\phi\in \PMap(\G_N)$ is in $\left( \mathcal{F}\cV_K\right)^{4+5N}$. 
    
    First apply \Cref{LEM:folding} to get an element $u\in \PMap(\G_N)$ such that $u\phi \in \cV_K$. Now use \Cref{LEM:RaysLoops} to write $u=u_{R_N}\circ \dots \circ u_{R_{1}}\circ u_{\ell}$ where $u_{R_i}$ is supported on $R_i$. By \Cref{lem:supp_on_loops} we know $u_\ell\in \left(\cF\cV_K\right)^3$.
    On the other hand, each $u_{R_i}$ has a homotopy representative as a word map $\wm{w_i,I_i}$, to which we will apply \Cref{LEM:RayMap}. That is, $u_{R_i}\in \left(\mathcal{F}\cV_K\right)^5$.
    All in all, we can now write $u\in \left( \cF \cV_K \right)^{3+5N}$.  Combining this with the expression $u\phi \in \cV_K$ we get that $\phi \in \left(\cF \cV_K \right)^{4+5N}$. 
\end{proof}

Since $\G_N$ with $N \in \Z^+$ has finitely many ends, \Cref{COR:finiteEndCB} implies:
\begin{COR}
    \label{cor:HLNCB}
    For any $N \in \Z^+$, the full mapping class group $\Map(\G_N)$ is coarsely bounded.
\end{COR}

\subsection{The Millipede Monster Graph} \label{ssec:millipede}

Let $\G_\infty$ be the graph with infinite rank whose end space is homeomorphic to $\{\frac{1}{2^n}: n\in \Z^+\}\cup\{0\}$ with $E_{\ell}(\G_{\infty})=\{0\}$, as shown in \Cref{fig:lochness}. The next proof again uses Rosendal's criterion, but we note that this is the only case where we don't show uniformity in the size of $\mathcal{F}$ and $n$ across different open neighborhoods of the identity. Uniformity of these constants is always present in the surface case \cite{mann2022large}.

\begin{PROP}
  \label{PROP:MillipedeCB}
    $\PMap(\G_\infty)$ is coarsely bounded.
\end{PROP}

\begin{proof}
    We will again use Rosendal's criterion. Given any open neighborhood of the identity $\mathcal{V}$, choose $K=[v_1,v_n]$ so that $\cV_K\subset \cU$. Let $f = \cL(n,1,n+1)$ and choose intervals $I_{i}$ in each ray $R_{i}$ for $i=1,\ldots,n$. Set $\cF = \{f,\wm{a_{n+1},I_{1}}^{\pm},\ldots,\wm{a_{n+1},I_{n}}^{\pm}\}$. We will show that any $\phi \in \PMap(\Gamma_{\infty})$ is in $(\cF \cV_{K})^{7+5n}$. 
    
    Once again we apply \Cref{LEM:folding} to get an element $u \in \PMap(\Gamma_{\infty})$ such that $u\phi \in \cV_{K}$ and $u$ is totally supported on $K'$, such that $K'$ is a compact neighborhood of $[v_1,v_m]$ in $ [v_{1},v_{m}] \cup R_1\cup \ldots\cup R_m$,
 for some $m \geq n$. Use \Cref{LEM:RaysLoops} to write $u =\left( u_{R_{m}}\circ \cdots \circ u_{R_{n+1}} \right) \circ u_{R_{n}} \circ \cdots \circ u_{R_{1}} \circ u_{\ell}$ where each $u_{R_{i}}$ is supported on $R_{i}$ and $u_{\ell}$ has support only on the loops of $\Gamma_{\infty}$. Just as before, we apply \Cref{lem:supp_on_loops} to see that $u_{\ell} \in (\cF_{0}\cV_{K})^{3}$, for $\cF_{0} = \{f\}$, and \Cref{LEM:RayMap} to see that $u_{R_{i}} \in (\cF_{i}\cV_{K})^{5}$, for $\cF_{i} = \{f,\wm{a_{n+1},I_{i}}^{\pm}\}$ and $i\in \{1,\ldots, n\}$. Note that we chose $\cF = \bigcup_{i=1}^{n} \cF_{i}$ so that it only remains to check the following claim.
    
    \begin{CLAIME}
         $u_{R_{m}}\circ \cdots \circ u_{R_{n+1}} \in (\cF\cV_{K})^{3}$. 
    \end{CLAIME}
    
    \begin{proof}
        Each $u_{R_{i}}$ is homotopic to a word map $\wm{w_{i},I_{i}}$ for $I_{i}$ an interval in each $R_{i}$. Let $M=\max\{\{j \vert a_{j} \text{ appears in } w_{i}\}_{i=n+1}^{m} \cup \{n\}\}$ and let $g = \cL(n,n+1,M+1)$.
        
        Then using Lemma \ref{LEM:conjugatetongue} we have
        \begin{align*}
            fgu_{R_{i}}gf &= fg \wm{w_{i},I_{i}} gf \\
            &= \wm{(fg)_{*}(w_{i}),I_{i}}
        \end{align*}
        for all $i \in\{ n+1,\ldots,m\}$. Since $g$ was chosen so that $(fg)_{*}(w_{i}) \in A_{n+1,\infty}$, we have that $fgu_{R_{i}}gf \in \cV_{K}$ for all $i$. Finally, we note that \[(fgu_{R_{m}}gf)\circ \cdots \circ (fgu_{R_{n+1}}gf) = fg \circ \left(u_{R_{m}} \circ \cdots \circ u_{R_{n+1}}\right) \circ gf,\] and rearrange to get the claim.
        \renewcommand{\qedsymbol}{$\triangle$}
    \end{proof}
    
    Combining each of these allows us to conclude that $u \in (\cF\cV_{K})^{6+5n}$, and as $u\phi \in \cV_{K}$ we obtain that $\phi \in (\cF\cV_{K})^{7+5n}$. 
\end{proof}

We included \Cref{THM:permutationCB} as a warm up proof earlier in the section, but it is also the other key ingredient to proving the following corollary.

\begin{COR}\label{cor:millipedeCB}
    $\Map(\G_{\infty})$ is coarsely bounded. 
\end{COR}

\begin{proof}
    By \Cref{PROP:SES}, because $\PMap(\G_{\infty})$ is coarsely bounded, $\Map(\G_{\infty})$ is coarsely equivalent to $\Homeo\left(E(\G_{\infty}),E_{\ell}(\G_{\infty})\right) \cong S_{\infty}$. By \Cref{THM:permutationCB}, $S_{\infty}$ is coarsely bounded.
\end{proof}

\section{Graphs of Finite Positive Rank}
\label{sec:finiteposrk}
In this section we will see that $\PMap(\Gamma)$ is not CB for graphs $\Gamma$ of finite positive rank except when $\Gamma = \lasso$. Recall when $\G$ has rank 0, \Cref{prop:rank0} shows that $\PMap(\G)$ is trivial, hence CB.

Note any locally finite graph $\Gamma$ is an Eilenberg-Maclane space, $K(\pi_1(\Gamma),1)$, so there is a natural homomorphism 
\begin{align*}
    \Psi: \Map(\Gamma) \rightarrow \Out(\pi_{1}(\Gamma))
\end{align*}
that associates $g \in \Map(\Gamma)$ with the corresponding outer automorphism class of $g_{*}:\pi_{1}(\Gamma) \rightarrow \pi_{1}(\Gamma)$. We refer the reader to \cite[Chapter 3]{AB2021} for an in-depth discussion of the map $\Psi$ and its kernel. Note $\Psi$ is surjective when $\Gamma$ has finite rank $n$ and that $\pi_{1}(\Gamma) \cong F_{n}$, the free group of rank $n$.
The restriction $\Psi|_{\PMap(\Gamma)}$ to $\PMap(\Gamma)$ still surjects onto $\Out(\pi_1(\Gamma))$ when $\Gamma$ has finite rank. This is because the fundamental group, $\pi_1(\Gamma)\cong \pi_1(\Gamma_c)$, only captures the finite core graph $\Gamma_c$ of $\Gamma$, so we can choose the extension of the map $\Gamma_{c} \rightarrow \Gamma_{c}$ corresponding to a given outer automorphism to fix the ends of $\Gamma$.

In general when $\G$ has infinite rank, however, $\Psi$ is not surjective as there are automorphisms not realized by proper homotopy equivalences, such as the automorphism induced by the inverse homotopy equivalence in the example in \Cref{RMK:notPHE}.

\subsection{Graphs of Rank $>1$}
We first deal with the generic case, when a graph has rank larger than 1.

\begin{LEM}
 If $\Gamma$ is a locally finite graph of finite rank $n>1$, then $\Psi: \Map(\Gamma) \rightarrow \Out(F_{n})$ is continuous.
\end{LEM}

\begin{proof}
 Since $\Map(\Gamma)$ is a topological group and $\Out(F_{n})$ is a discrete group, it is sufficient to check that $\ker{\Psi}=\Psi^{-1}(\{[id]\})$ contains an open subgroup. 
 Observe that $\cV_{\Gamma_{c}}$ is an open subgroup of $ \ker{\Psi}$.
\end{proof}

\begin{COR}
\label{cor:finiterank}
  If $\Gamma$ is a locally finite graph of finite rank $n>1$ then $\Map(\Gamma)$ is not CB. In particular, $\PMap(\Gamma)$ is not CB.
\end{COR}

\begin{proof}
 Note that $\Out(F_{n})$ is not CB by Condition (2) of \Cref{PROP:RosendalCB}. Indeed, $\Out(F_{n})$ has many unbounded continuous actions on metric spaces, e.g., on its Cayley graph or on the free factor complex. We can thus precompose this action with the continuous surjective map $\Psi$ to obtain an unbounded continuous action of $\Map(\Gamma)$.
 Further precompose with the inclusion $\PMap(\Gamma) \hookrightarrow \Map(\Gamma)$ to deduce the latter assertion.
\end{proof}

\subsection{Graphs of Rank 1}
\label{ss:rank1}

The technique used above fails when $\Gamma$ has rank one because the target group of $\Psi$ is now $\Out(\Z) \cong \Z/2\Z$, a finite (hence CB) group.
Instead we will make use of the semi-direct product description of $\PMap(\Gamma)$ from \cite{AB2021} to see that $\PMap(\Gamma)$ will also not be CB, other than one exceptional case $\G=\lasso$.

For the sake of brevity we state definitions and results from \cite[Section 3]{AB2021} in the specific case that $\Gamma$ has finite positive rank; however, it is important to note that these definitions and results can be stated more generally. 

Let $\Gamma$ have finite positive rank and choose some $e_{0} \in E(\Gamma)$.  Let $\pi_{1}(\Gamma,e_0)$ be the group of proper homotopy classes of lines $\sigma:\R \rightarrow \Gamma$ with $\lim_{t\rightarrow \infty}\sigma(t) = \lim_{t\rightarrow - \infty}\sigma(t) = e_0$, with group operation given by concatenation.
In \cite{AB2021} it is noted that given any $x_{0} \in \Gamma$ there is an isomorphism $\pi_{1}(\Gamma,x_{0}) \rightarrow \pi_{1}(\Gamma,e_0)$. Let $\Gamma_{c}^{*} \subset \Gamma$ be the subgraph consisting of the core graph $\Gamma_{c}$ together with a choice of ray in $\Gamma$ that limits to $e_0$ and intersects $\Gamma_{c}$ in exactly one point.
\begin{DEF}[{\cite[Definition 3.3]{AB2021}}]
    The group $\cR$ as a set is the collection of maps $h:E(\Gamma) \rightarrow \pi_{1}(\Gamma_{c}^{*},e_{0})$ satisfying
    \begin{enumerate}
        \item[(R0)]
            $h(e_{0})=1$, and 
        \item[(R1)]
            $h$ is locally constant. 
    \end{enumerate}
    The group operation in $\cR$ is given by pointwise multiplication in $\pi_{1}(\Gamma,e_{0})$.
\end{DEF}

Algom-Kfir and Bestvina use this group to give a description of $\PMap(\Gamma)$ as a semidirect product. 

\begin{THM}[{\cite[Corollary 3.9]{AB2021}}] \label{THM:ABsemidirect}
  If $\Gamma$ has finite positive rank then
  \begin{align*}
      \PMap(\G) \cong \cR \rtimes \PMap(\G_c^*),
  \end{align*}
  Moreover, the natural homomorphism $\PMap(\G_c^*) \to \Out(\pi_1(\G_c^*))$ is a surjection when $\rk(\G) \ge 2$ and is an isomorphism when $\rk(\G) =1$.
\end{THM}

We can now prove the following. 

\begin{PROP}
\label{prop:rank1}
    Let $\Gamma$ be a locally finite, infinite graph of rank 1. Then $\PMap(\G)$ is CB if and only if $|E(\Gamma)|=1$, that is, if $\G = \lasso$.
\end{PROP}

\begin{proof}
    First note that if $|E(\G)|=1$, then $\cR=1$ so it follows that $\PMap(\G) \cong \PMap(\G_c^*) \cong \Out(F_1) \cong \Z/2\Z$, which is CB.
    
    Now conversely, assume $|E(\G)|>1$ and
    we claim that $\PMap(\G)$ is not CB.
    By \Cref{THM:ABsemidirect}, and that $\PMap(\G_c^*) \cong \Z/2\Z$, we have that $\cR$ is a clopen index 2 subgroup of $\PMap(\Gamma)$. Therefore, $\cR$ is Polish, and by \Cref{PROP:finiteindexCB} it suffices to show that $\cR$ is not CB.
    
    Let $e \in E(\Gamma)$ with $e \neq e_0$ and $\phi_{e}: \cR \rightarrow \pi_{1}(\Gamma_{c}^{*},e)$ be the evaluation homomorphism at $e$. Since $\cR$ is nontrivial, $\phi_e$ is a nontrivial homomorphism from $\cR$ to $\Z$. A classical result of Dudley \cite{dudley1961} states that any homomorphism from a Polish group to $\Z$ is continuous. Thus $\phi_{e}$ defines a continuous nontrivial homomorphism from $\cR$ to $\Z$ showing that $\cR$ is not CB. 
\end{proof}

\section{Length Functions: Graphs with Infinite Combs and Trees} \label{SEC:Length}

In this section we consider graphs, $\Gamma$, of rank at least one and with $E(\Gamma) \setminus E_{\ell}(\Gamma)$ having an accumulation point. Note that this includes the class of graphs with one end accumulated by loops and infinite end space with the one exception of the Millipede Monster graph, $\G_{\infty}$, described in \Cref{ssec:millipede}. We show that for any such graph $\G$, its pure mapping class group is neither coarsely bounded nor generated by a coarsely bounded set.

\begin{THM}\label{thm:oneendnotCB}
    Let $\G$ be a locally finite, infinite graph with $\rk(\G) > 0$ and $E(\Gamma) \setminus E_\ell(\Gamma)$ containing an accumulation point. Then $\PMap(\G)$ is not CB and is neither algebraically nor topologically CB-generated.
\end{THM}

We will also show that these groups act on simplicial trees. This will be used again in \Cref{ss:asdim0}. We first prove $\PMap(\G)$ is not CB by showing that these mapping class groups admit an unbounded length function. 

\begin{DEF}
    A \textbf{length function} on a topological group $G$ is a continuous function $\ell: G \arr [0, \infty)$ satisfying  \begin{enumerate}[(a)]
        \item $\ell(\Id)=0$,
        \item $\ell(g) = \ell(g\inv)$,
        \item $\ell(gh)\leq \ell(g)+\ell(h)$ for all $g,h\in G$.
    \end{enumerate}
\end{DEF}

\begin{RMK}
\label{rmk:unbddlengthftn}
Having an unbounded length function is in direct contradiction with Rosendal's criterion, \Cref{PROP:RosendalCB}. This is because a length function on $G$ induces a left-invariant pseudo metric $d(g,h):= \ell(g^{-1}h)$ on $G$. Note $G$ continuously acts on $(G,d)$ by left multiplication. Hence if $H \le G$ is unbounded under $\ell$, then the orbit $H \cdot \Id_H=H$ is unbounded, so $H$ is not CB in $G$.
\end{RMK}

In order to define the length function on $\PMap(\G)$ we consider the geometric realization of $\G$ so that each edge is identified with $[0,1]$ and put the path metric on $\G$. 
Consider the connected components $\{T_i\}$ of $\G \setminus \G_c$. Because there exists an accumulation point in $E(\G)\setminus E_\ell(\G) = \bigsqcup E(T_i)$, some $T_i$ has infinite end space. For each such $T_i$ we can define an unbounded length function on $\PMap(\G)$. 

\begin{PROP}
\label{prop:lengthfunction}
Let $\G$ be a locally finite, infinite graph with $\rk(\G)\geq 1$ and $E\setminus E_{\ell}$ having an accumulation point.
Let $T$ be a connected component of $\G \setminus \G_c$ with infinite end space. Then the map $\ell$ defined on $\PMap(\G)$ as:
\begin{align*}
    \ell:&\PMap(\G) \arr \Z_{\ge 0},  \\
    &g \mapsto \min\{k \in \Z_{\ge 0}\ \vert\ g(L)\text{ is properly homotopy equivalent to }L, \text{ $\forall\;$geodesic lines $L\subset T$ with } d(L,\G_{c}) \geq k\}
\end{align*}
is a well-defined unbounded length function on $\PMap(\G)$. Moreover, $\ell$ is an \textit{ultranorm}, namely $\ell(gh) \le \max\{\ell(g),\ell(h)\}$,
for all $g,h \in \PMap(\G)$. 
\end{PROP}

\begin{proof}

First, for any properly homotopic $g,g' \in \PHE(\G)$ and geodesic line $L \subset T$, we have $g(L) \simeq L$ if and only if $g'(L) \simeq L$. Hence $\ell$ is well-defined on $\PMap(\G)$. Also, as discussed in \Cref{rmk:RaymapsareWordmaps} any proper homotopy equivalence on $\G$ cannot have infinite support on $T$. This means that for any mapping class $g \in \PMap(\G)$, there is a representative of $g$ that when restricted to $T$ is supported on a finite collection of intervals. In particular, there is some geodesic line $L$ that is fixed by $g$, up to proper homotopy equivalence.

One can immediately see that $\ell(\Id) = 0$ and $\ell(g^{-1})=\ell(g)$. Any pure proper homotopy equivalence on $\G$ can be homotoped so that on $T$ it is a finite collection of disjointly supported word maps.
The composition of two word maps on the same edge is also a word map on that edge, so $\ell(gh)\leq \max\{\ell(g),\ell(h)\}\le \ell(g)+\ell(h)$ for all $g,h \in \PMap(\G)$. Next we show that $\ell:\PMap(\G) \to \Z_{\ge 0}$ is continuous.

Let $v$ be the vertex in $\G_c$ that connects $T$ to $\G_c$; that is, $\overline{T}\cap \G_c=\{v\}$. Let $K$ be any finite, positive rank subgraph of $\G_c$ that contains $v$. Since $K$ disconnects $T$ from the rest of $\G$, any element $u\in\cV_K$ is totally supported in $\G \setminus T$, so $\ell(u)=0$. Now to show that $\ell$ is continuous we will show that for any $n\in \Z_{\geq 0}$ the preimage of $\{n\}$ is open in $\PMap(\G)$.
Let $g\in \ell\inv(\{n\})$. Then $g\cV_K\subset \ell\inv(\{n\})$ because for any $u\in \cV_K$, we have $\ell(gu)=\ell(g)=n$. Here we have equality because $u$ is totally supported on $\G\setminus T$, so cannot ``undo'' any part of $g$ and decrease its length.

Finally, because $T$ has an infinite end space, $E(T)$ contains an accumulation point. Let $I_{i}$ be a sequence of intervals in $T$ converging to this end. Then $\ell(\wm{w,I_{i}}) \rightarrow \infty$ for any nontrivial word $w \in \pi_1(\G) \neq 1$.
\end{proof}

\begin{proof}[Proof of \Cref{thm:oneendnotCB}]
    By \Cref{prop:lengthfunction}, $\PMap(\G)$ admits an unbounded length function $\ell$, and \Cref{rmk:unbddlengthftn} implies that $\PMap(\G)$ is not CB. Now let $A$ be a CB set in $\PMap(\G)$. Then by \Cref{rmk:unbddlengthftn} again, there exists $L$ such that $\ell(a)\le L$ for each $a \in A$. Because $\ell$ is an ultranorm, products of elements in $A$ also have length at most $L$. However, $\PMap(\G)$ has elements of arbitrarily long length so cannot be generated by $A$. This proves $\PMap(\G)$ is not CB-generated. Moreover, as $\ell$ is continuous, a CB set $A$ cannot topologically generate $\PMap(\G)$ either. 
\end{proof}

Now we would like to show that $\PMap(\Gamma)$ acts on a tree.
Define a pseudo-metric $\hat{d}$ on $\PMap(\Gamma)$ as $\hat{d}(g,h):= \ell(g^{-1}h)$. Then $\PMap(\Gamma)$ acts on itself by left multiplication isometrically: for $g,h,k \in \PMap(\Gamma)$,
\[
    \hat{d}(kg,kh) = \ell(g^{-1}k^{-1} \cdot kh) = \ell (g^{-1}h) = \hat{d}(g,h).
\]

Now we consider the quotient space $\cP = \PMap(\Gamma)/{\sim}$, where $g \sim h$ for $g,h \in \PMap(\G)$ iff $\hat{d}(g,h)=0$, namely, identify all the elements in $\PMap(\Gamma)$ that are at distance 0 from each other. Then the pseudo-metric $\hat{d}$ induces a metric $d$ on $\cP$, and $\PMap(\Gamma)$ isometrically acts on $(\cP,d)$ by the left multiplication.

To show $(\cP,d)$ forms the vertex set of some tree, we first claim that $(\cP,d)$ is 0-hyperbolic.
From the ``ultranorm'' nature of $\ell$ as in \Cref{prop:lengthfunction}, we observe that $d$ is an \textit{ultrametric} on $\cP$, i.e., $d(g,h) \le \max\{d(g,k),d(h,k)\}$ for any $g,h,k \in \cP$. Indeed,
\begin{align*}
    d(g,h) &= \ell(g^{-1}h) = \ell(g^{-1}k \cdot k^{-1}h) \le \max\{\ell(g^{-1}k),\ell(k^{-1}h)\} \\
    &= \max\{\ell(g^{-1}k),\ell(h^{-1}k)\}
    = \max\{d(g,k), d(h,k)\}.
\end{align*}

Next we apply the following well known lemma \cite[Exercise 1]{Duchin2017} to see that our metric space is $0$-hyperbolic. As of yet we have not found a written proof of this lemma and so we include a proof in the appendix.

\begin{LEM}[\Cref{lem:ultram0hyp2}]
\label{lem:ultram0hyp}
Every ultrametric space is Gromov 0-hyperbolic.
\end{LEM}

\begin{COR} \label{cor:actionOnTree}
    $\PMap(\Gamma)$ acts on a simplicial tree $\cT$.
\end{COR}

\begin{proof}
We have seen that $\PMap(\Gamma)$ acts on $(\cP,d)$, which is an ultrametric space. By \Cref{lem:ultram0hyp}, it follows that $(\cP,d)$ is 0-hyperbolic. 
Note that every 0-hyperbolic space isometrically embeds into an $\R$-tree and any group action extends to an action on this $\R$-tree via the ``Connecting the Dots'' Lemma (See e.g.\ \cite[Lemma 2.13]{bestvina2002rtrees}). More precisely, to build the $\R$-tree, we start with the collection of based arcs $I_x=[[\Id],x]$ of length $\ell(x)$, and then identify $I_x$ and $I_y$ along $[[\Id],z]$, where $z \in \cP$ has $\ell(z)=(x,y)_{\Id}$, the Gromov product.
Then the distance between the vertices of this $\R$-tree is a half-integer, so the vertex set is discrete. All in all, we obtain a simplicial tree $\cT$.
\end{proof}

\begin{PROP}
    The $\PMap(\G)$-action on $\cP$ is continuous.
\end{PROP}

\begin{proof}
    Let $F:\PMap(\G) \times (\cP,d) \rightarrow (\cP,d)$ be the action. Note that $(\cP,d)$ is discrete, so we will show that $F^{-1}([g])$ is open. Let $K \subset \G$ be compact so that $K$ separates the $T$ used to define our length function from the core graph, $\G_{c}$. Now let $(f,[h])$ be an arbitrary point in $F^{-1}([g])$. Thus we have
    \begin{align*}
        F(f,[h])=[fh]=[g],
    \end{align*}
    showing that $g^{-1}fh \in \cV_K$, so $f \in g\cV_Kh^{-1}$. Note that since both left and right multiplication are continuous, $g\cV_Kh^{-1} \times [h]$ is an open neighborhood of $(f,[h])$ in $\PMap(\G) \times (\cP,d)$. To conclude,
    we now claim that $g\cV_{K}h^{-1}\times[h] \subset F^{-1}([g])$.  Let $a\in \cV_{K}$, then $a$ fixes $T$ and $\ell(a)=0$. Thus, we have 
    \begin{align*}
       F(gah^{-1},h)=[ga]=[g], \qquad \text{since $[a]=[\Id]$,}
    \end{align*}
    proving $g\cV_Kh^{-1} \times [h] \subset F^{-1}([g])$.
    Therefore, $F^{-1}([g])$ is open, and $F$ is continuous.
    
    We note that since the action on $(\cP,d)$ is continuous, as the natural extension, the action on the tree is also continuous.
\end{proof}

\begin{PROP}
    \label{prop:PformsLeaves}
    Each vertex from $\mathcal{P}$ is a leaf in $\mathcal{T}$.
\end{PROP}

\begin{proof}
    Note the $\PMap(\G)$-action on $\cP$ is transitive since $\PMap(\G)$ acts on itself transitively.
    In effect, it suffices to show the vertex corresponding to the identity element has valence one, as then the vertices corresponding to the elements in $\cP$ must have valence one by transitivity.
    
    We used the connect the dots lemma to construct $\mathcal{T}$, so the identity element has valence greater than one only if there are non-identity elements $g,h\in \mathcal{P}$ for which $(g,h)_{\Id}=0$. However, our length function is an ultranorm, so
    \[ (g,h)_{\Id}= \frac12 \left(\ell(g) +\ell(h) -\ell(g\inv h) \right) \geq \frac12 \left(\ell(g) +\ell(h) -\max\{\ell(g),\ell(h)\} \right). \]
    Thus, $(g,h)_{\Id} \ge \frac{1}{2} \min\{\ell(g),\ell(h)\}$ and is zero only if one of $g$ and $h$ is in the same equivalence class as the identity in $\mathcal{P}$. 
\end{proof}

\begin{PROP} \label{prop:1-endedTree}
    The tree $\mathcal{T}$ is one-ended. In particular, the action of $\PMap(\G)$ on $\cT$ fixes a point at infinity.
\end{PROP}

\begin{proof}
    If $\mathcal{T}$ has more than one end, then we can find sequences of vertices $\{x_i\}$ and $\{y_i\}$ such that \begin{enumerate}[(1)]
        \item $\{d(\Id, x_i)\} \arr \infty$ and $\{d(\Id,y_i)\} \arr \infty$, and
        \item $\{(x_i,y_i)_{\Id}\}$ is bounded.
    \end{enumerate}
    We first deal with the special case where we have sequences $\{g_i\}$ and $\{h_i\}$ corresponding to elements of $\mathcal{P}$ satisfying (1). Our length function is an ultranorm, so we get the following contradiction with (2).
    \begin{align*}
        (g_i,h_i)_{\Id} &= \frac12 \left( \ell(g_i) +\ell(h_i) -\ell(g_i^{-1}h_i) \right) \\ 
        &\geq \frac12 \left( \ell(g_i) +\ell(h_i) -\max \{\ell(g_i),\ell(h_i)\} \right) \\
        &=\frac12 \min\{\ell(g_i),\ell(h_i)\} \arr \infty.
    \end{align*}

Now observe that by construction, any non-leaf vertex $z$ in $\mathcal{T}\setminus \mathcal{P}$ appears as a bifurcation point for some two different geodesics $[\Id,g]$ and $[\Id,h]$, with distinct $g,h\in \mathcal{P}$; namely, $[\Id,g] \cap [\Id,h] = [\Id,z]$. In particular, $z$ has valence at least 3.
Then by the above observation, we can find leaves $g_i,h_i \in \cP$, such that $[\Id,g_i] \cap [\Id,h_i] =[\Id,z_i]=[\Id,x_i] \cap [\Id,y_i]$ for some vertex $z_i$ of $\cT$.
We may choose $g_i,h_i$ such that $[\Id,g_i] \supset [\Id,x_i]$ and $[\Id,h_i] \supset [\Id,y_i]$. Therefore, we have sequence of leaves $g_i,h_i$ such that 
\[
    d(\Id,g_i) \ge d(\Id,x_i) \to \infty, \qquad d(\Id,h_i) \ge d(\Id,y_i) \to \infty
\]
but
$(x_i,y_i)_{\Id} = (g_i,h_i)_{\Id}$ is not bounded by the special case, contradiction.
\end{proof}

\subsection{Combinatorial Model}

In this section we will give a combinatorial description of the tree constructed above for a specific graph $\G$. While we will not use this description to obtain any new results, we hope it helps to give the reader some intuition about what these trees look like.

Let $\Gamma$ be the graph with infinite rank and end space homeomorphic to $\left(\{0,2\}\cup\{\frac{1}{n}: n \in \Z^+\}, \{2\}\right)$. That is, $\Gamma$ is a Loch Ness monster graph with an ``infinite comb'' attached at the first loop as in \Cref{fig:lochnessandcomb}.

\begin{figure}[h]
	    \centering
	    \def\svgwidth{.8\textwidth}
\begingroup%
  \makeatletter%
  \providecommand\color[2][]{%
    \errmessage{(Inkscape) Color is used for the text in Inkscape, but the package 'color.sty' is not loaded}%
    \renewcommand\color[2][]{}%
  }%
  \providecommand\transparent[1]{%
    \errmessage{(Inkscape) Transparency is used (non-zero) for the text in Inkscape, but the package 'transparent.sty' is not loaded}%
    \renewcommand\transparent[1]{}%
  }%
  \providecommand\rotatebox[2]{#2}%
  \newcommand*\fsize{\dimexpr\f@size pt\relax}%
  \newcommand*\lineheight[1]{\fontsize{\fsize}{#1\fsize}\selectfont}%
  \ifx\svgwidth\undefined%
    \setlength{\unitlength}{468.47622898bp}%
    \ifx\svgscale\undefined%
      \relax%
    \else%
      \setlength{\unitlength}{\unitlength * \real{\svgscale}}%
    \fi%
  \else%
    \setlength{\unitlength}{\svgwidth}%
  \fi%
  \global\let\svgwidth\undefined%
  \global\let\svgscale\undefined%
  \makeatother%
  \begin{picture}(1,0.10966108)%
    \lineheight{1}%
    \setlength\tabcolsep{0pt}%
    \put(0,0){\includegraphics[width=\unitlength,page=1]{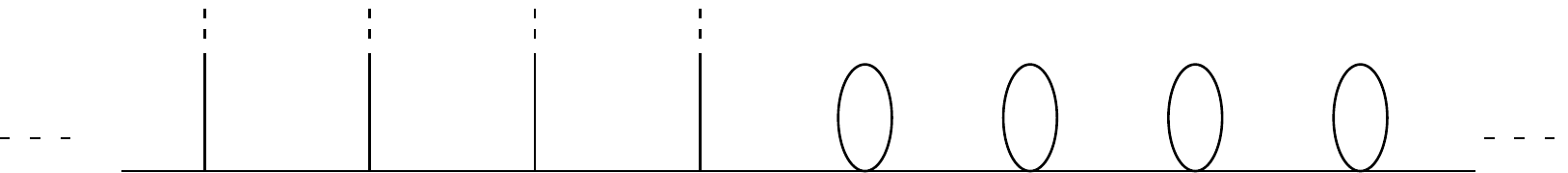}}%
    \put(0.40032422,0.06346862){\color[rgb]{0,0,0}\makebox(0,0)[lt]{\lineheight{1.25}\smash{\begin{tabular}[t]{l}$R_1$\end{tabular}}}}%
    \put(0.29463125,0.06346862){\color[rgb]{0,0,0}\makebox(0,0)[lt]{\lineheight{1.25}\smash{\begin{tabular}[t]{l}$R_2$\end{tabular}}}}%
    \put(0.18893823,0.06346862){\color[rgb]{0,0,0}\makebox(0,0)[lt]{\lineheight{1.25}\smash{\begin{tabular}[t]{l}$R_3$\end{tabular}}}}%
    \put(0.08324519,0.06346862){\color[rgb]{0,0,0}\makebox(0,0)[lt]{\lineheight{1.25}\smash{\begin{tabular}[t]{l}$R_4$\end{tabular}}}}%
  \end{picture}%
\endgroup%

		    \caption{The Loch Ness Monster graph with an ``infinite comb" attached.} 
	    \label{fig:lochnessandcomb}
\end{figure}
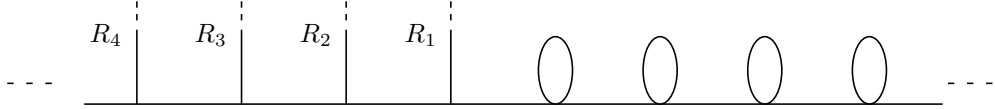

We define $\mathcal{T}$ to be the graph with the following vertex and edge sets.
\begin{itemize}
    \item Vertex Set: Pairs of the form $(n,f)$ where $n \in \Z^+$ and $f:\Z^+ \cap [n,\infty) \rightarrow F_{\infty}$ is a ``prescribing'' function which maps only finitely many points to non-identity elements. 
    \item Edge Set: Two vertices, $(n,f)$ and $(m,g)$, are joined by an edge if and only if $|m-n| = 1$ and $f$ and $g$ agree on the overlap of their domains. 
\end{itemize}

The intuition behind these definitions is that the vertices are similar to ``cylinder sets''. That is, there is a natural height function on the vertices (given by the $\Z^+$ coordinate) and vertices at height $1$ are exactly prescribing a word map on each of the rays on the comb, vertices at height 2 are prescribing word maps starting on the second ray while allowing for anything on the first ray, and vertices at height $n$ are prescribing word maps starting on the $n$th ray. 

Observe that $\mathcal{T}$ is a connected tree. For connectedness note that given any vertex $(n,f)$ there is a path from $(n,f)$ to some $(m,\Id)$ with $m>n$ since $f$ has finite ``support'' and all of the $(m,\Id)$ lie on a connected ray. For treeness we note that the descending link of any vertex is a single edge. See \Cref{fig:combmodel} for a schematic picture of the tree. After endowing the edges with appropriate lengths, this $\cT$ is \emph{isometric} to the tree $\cT$ constructed in \Cref{cor:actionOnTree}.
To achieve such an isometry, one defines a map from the height 1 vertices to $(\cP,d)$ and verifies that the Gromov products on each space are equivalent. 

\begin{figure}[ht!]
	    \centering
	    \def\svgwidth{.8\textwidth}
		    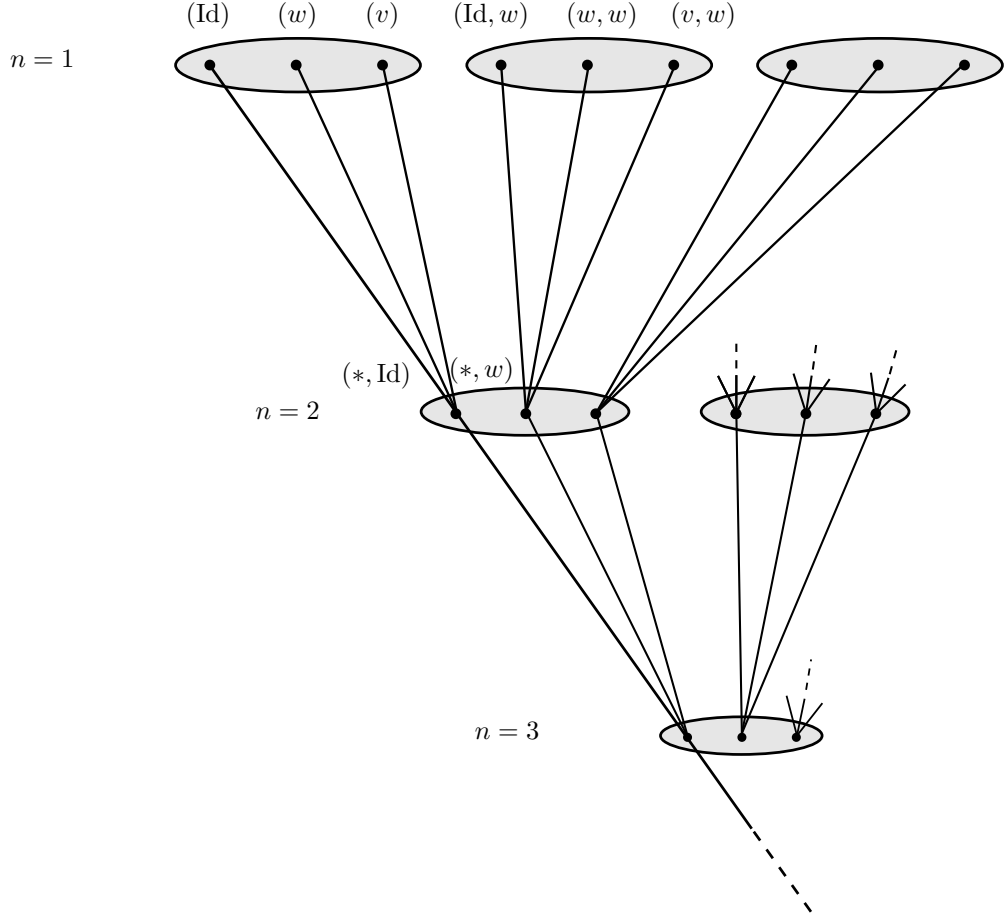
		    \caption{A portion of the simplicial tree we get from $\G$. Each cone above a vertex is indexed by $F_{\infty}$. The vertex labels only show the support of the prescribing functions.} 
	    \label{fig:combmodel}
\end{figure}

We next describe the action of $\PMap(\G)$ on $\mathcal{T}$. First given any $\phi \in \PMap(\G)$ we can write $\phi = \wm{w_{1},I_{1}}\circ \cdots \circ \wm{w_{m},I_{m}} \circ \phi_{c}$ for some $m$ where $\phi_{c}$ is the restriction of $\phi$ to the core graph of $\Gamma$, $I_{i}$ is an interval in the ray $R_{i}$, and $w_{i} \in F_{\infty}$. Note that in this notation we allow some of the $w_{i}$ to be the identity element. Set $w_i = \Id$ for $i>m$. Now given any function $f: \Z^+ \cap [n,\infty) \to F_\infty$ we define
\begin{align*}
    (\phi \cdot f)(i) = w_{i}\cdot\phi_{c*}(f(i)),
\end{align*}
for all $i \in \Z_{\ge n}$. Now we define the action of $\PMap(\G)$ on the vertex set of $\mathcal{T}$ by $\phi \cdot (n,f) = (n,\phi \cdot f)$. In order to verify that this defines a group action we need the following 1-ended and possibly non-compactly supported version of Lemma \ref{LEM:conjugatetongue}.

\begin{LEM}[General Conjugation Rule] \label{LEM:conjtongueextension}
    Suppose $\G$ is a locally finite, infinite graph with $|E_{\ell}(\G)| = 1$. Let $\wm{w,I} \in \PMap(\G)$ be a word map supported outside of the core graph, $\G_{c}$, and $\psi \in \PMap(\G)$ be (not necessarily compactly supported and) totally supported on $\G_{c}$. Then 
    \begin{align*}
        \psi \circ \wm{w,I} \circ \psi^{-1} = \wm{\psi_{*}(w),I}.
    \end{align*}
\end{LEM}

\begin{proof}
    As $|E_\ell(\G)|=1$, by Corollary \ref{COR:closurecompactsup} we can approximate $\psi$ by a sequence $\{\psi_{n}\}$ with $\lim_{n\rightarrow \infty} \psi_{n} = \psi$ and each of the $\psi_{n}$ being totally supported on a compact subgraph of $\G_{c}$. Thus we can apply Lemma \ref{LEM:conjugatetongue} to each of these $\psi_{n}$ so that 
    \begin{align*}
        \psi_{n} \circ \wm{w,I} \circ \psi_{n}^{-1} = \wm{(\psi_{n})_{*}(w),I}
    \end{align*}
    for each $n$. Now note that the sequence $(\psi_{n})_{*}(w)$ must be eventually constant so that for large enough $n$, $(\psi_{n})_{*}(w) = \psi_{*}(w)$. So, we have 
    \begin{align*}
        \psi_{n} \circ \wm{w,I} \circ \psi_{n}^{-1} = \wm{\psi_{*}(w),I}
    \end{align*}
    for large enough $n$. Finally, taking limits on both sides and using the continuity of group multiplication gives the desired result. 
\end{proof}

This lemma is exactly what allows us to verify that the action defined above is actually a group action on the vertex set. Indeed, given any two $\phi,\psi \in \PMap(\G)$ we can write for large enough $n$,
\begin{align*}
    \phi &= \wm{w_{1},I_{1}} \circ\cdots \circ  \wm{w_{n},I_{n}} \circ \phi_{c} \\
    \psi &= \wm{w'_{1},I_{1}} \circ\cdots \circ  \wm{w'_{n},I_{n}} \circ \psi_{c}
\end{align*}
again with some $w_{i}$ and $w'_{i}$ potentially the identity. Then by \Cref{LEM:conjtongueextension} we have
\begin{align*}
    \psi \circ \phi &= \wm{w'_{1},I_{1}} \circ\cdots \circ  \wm{w'_{n},I_{n}} \circ \wm{\psi_{c*}(w_{1}),I_{1}} \circ\cdots \circ  \wm{\psi_{c*}(w_{n}),I_{n}} \circ \psi_{c} \circ \phi_{c} \\
    &= \wm{w_1'\psi_{c*}(w_1),I_1} \circ \ldots \circ \wm{w_n'\psi_{c*}(w_n),I_n} \circ (\psi_c \circ \phi_c),
\end{align*}
where the second equality is from Composition rule (\Cref{LEM:compositionTongues}).

Now we can verify that this gives a group action on $V(\mathcal{T})$:
\begin{align*}
    (\psi \cdot (\phi \cdot f))(i) &= w'_{i}\cdot\psi_{c*}(w_{i} \cdot \phi_{c*}(f(i))) \\
    &= w_{i}'\psi_{c*}(w_{i}) \cdot (\psi_{c}\circ \phi_{c})_{*}(f(i)) \\
    &= ((\psi \circ \phi) \cdot f)(i).
\end{align*}

Similarly, one can also check that this action preserves adjacency in $\mathcal{T}$ and so acts by graph isomorphisms. 

Now we prove that $\PMap(\G)$ acts on $\cT$ parabolically, where the fixed point is the unique end point at $\infty$ in \Cref{prop:1-endedTree}.

\begin{PROP}\label{prop:parabolic}
    $\PMap(\G)$ acts transitively on the level sets of $\mathcal{T}$. 
\end{PROP}

\begin{proof}
    We check that given any $(n,f)$ there exists some $\phi \in \PMap(\G)$ such that $\phi \cdot (n,\Id) = (n,f)$. Indeed, taking $\phi = \prod_{i=1}^{\infty}\wm{f(i),I_{i}}$ will work. Note that the fact that $f$ has finite support guarantees that this is a well-defined mapping class as all but finitely many of the elements in the product will be identity word maps.
\end{proof}

\section{Flux Maps: Graphs with $|E_{\ell}| \ge 2$}\label{sec:TwoEnds} 

Let $\Gamma$ be an infinite, locally finite graph with at least two ends accumulated by loops. We show every such graph has non-CB pure mapping class groups, inspired by the work of Aramayona-Patel-Vlamis \cite{APV2020} in the infinite-type surface case. Similar maps also appeared in Durham-Fanoni-Vlamis \cite{DFV2018} where they were used to show that the mapping class group of the two-ended ladder surface is not CB. We first need some background on free factors.

\subsection{Free Factors}
\label{ss:FreeFactors}

\begin{DEF}
    Let $G$ be a group. Then $A<G$ is a \textbf{free factor} of $G$ if there exists some $P<G$ such that $G = A * P$. 
\end{DEF}

\begin{LEM} \label{LEM:subfreefact}
    Let $C$ be a free group and $A<B<C$ with $A$ a free factor of $C$. Then $A$ is also a free factor of $B$.  
\end{LEM}

\begin{proof}
Because $C$ contains $A$ as a free factor, we can realize $A$ and $C$ as a pair of graphs $\Delta \subset \Gamma$, where $\pi_1(\G,p) \cong C$ for some $p \in \Delta$, and the isomorphism restricts to $\pi_1(\Delta,p) \cong A$.
Consider the cover corresponding to the subgroup $B$, denoted $\rho: (\G_B,\tilde{p}) \to (\G,p)$. Let $i:\Delta \arr \G$ denote the inclusion map, we have that \[ i_*(\pi_1(\Delta,p))=A<B=\rho_*(\pi_1(\G_B,\tilde{p})),\] so the inclusion lifts to $\tilde{i}: (\Delta,p) \arr (\G_B,\tilde{p})$. As $\rho \circ \tilde{i}=i$ and $i$ is injective, $\tilde{i}$ is also injective. Similarly, $\tilde{i}_*:\pi_1(\Delta,p) \to \pi_1(\G_B,\tilde{p})$ is injective so it follows that $\pi_1(\tilde{i}(\Delta),\tilde{p}) = \tilde{i}_*(\pi_1(\Delta,p)).$ Therefore, $\G_B$ contains $\tilde{i}(\Delta)$, which is a homeomorphic copy of $\Delta$, and the isomorphism $\rho_*: \pi_1(\G_B,\tilde{p}) \cong B$ restricts to the isomorphism $\rho_*: \pi_1(\tilde{i}(\Delta),\tilde{p})=\tilde{i}_*(\pi_1(\Delta,p)) \cong A$. We conclude $A$ is a free factor of $B$.
\end{proof}

\begin{DEF}
    Let $B$ be a free group and $A < B$ a free factor. Define the \textbf{corank} of $A$ in $B$, $\cork(B,A)$, to be 
    \begin{align*}
        \cork(B,A) \defeq \rk(B/\la\la A \ra\ra),
    \end{align*}
    where $\la\la A\ra\ra$ is the normal closure of $A$ in $B$.
    Equivalently, if we write $B = A * P$, then $\cork(B,A) = \rk(P)$.
\end{DEF}

\begin{LEM} \label{LEM:corkadditive}
    The function $\cork$ is additive. I.e., if $A<B<C$ with $A$ and $B$ both free factors of a free group $C$ then 
    \begin{align*}
        \cork(C,A) = \cork(C,B) + \cork (B,A).
    \end{align*}
\end{LEM}

\begin{proof}
    Firstly, note that $A$ is a free factor of $B$ by \Cref{LEM:subfreefact} so $\cork(B,A)$ is well-defined.
    The equality follows from the fact that the free product operation on groups is associative. Indeed, we have some $P_{B},P_{A}$ such that 
    \begin{align*}
        C &= B * P_{B}, \text{ and} \\
        B &= A * P_{A}.
    \end{align*}
    Thus, 
    \begin{align*}
        C = (A*P_{A})*P_{B} = A*(P_{A}*P_{B}).
    \end{align*}
    This implies that $\cork(C,A) = \rk(P_{A}*P_{B})$. If either $P_{A}$ or $P_{B}$ has infinite rank then so does $P_{A}*P_{B}$. If $P_{A}$ and $P_{B}$ both have finite rank we can apply Grushko's Theorem \cite{Grushko1940} to see that 
    \[
        \rk(P_{A}*P_{B}) = \rk(P_{A}) + \rk(P_{B}) = \cork(B,A)+\cork(C,B). \qedhere
    \]
\end{proof}

\subsection{Constructing Flux Maps}
\label{ssec:fluxMap}

\begin{THM}\label{thm:twoendnotCB}
    Let $\Gamma$ be a graph with at least two ends accumulated by loops. Then $\PMap(\Gamma)$ is not CB.
\end{THM}

We prove this by finding a \emph{flux} map given any two distinct ends accumulated by loops.
Let $\PPHE(\Gamma)$ be the group of proper homotopy equivalences of $\Gamma$ that induce the identity map on the end space of $\Gamma$.
For any partition of the ends of $\Gamma$ into two clopen sets, each of which contains an end accumulated by loops, we will define a flux map $\Phi: \PMap(\G)\arr \Z$. We again will always be using standard forms of our graphs. In particular, note that any edge that is not a loop in a standard form graph is separating.

Fix a line $\gamma$ in the maximal tree of $\Gamma$ whose ends correspond to two distinct ends accumulated by loops. Pick any edge of $\gamma$, and let $x_0$ be the midpoint. Then $\Gamma \setminus \{x_0\}$ will induce our desired partition $C_L \cup C_R$ of the end space of $\Gamma$. More precisely, $\Gamma \setminus \{x_{0}\}$ has two components, $\Gamma_{L}$ and $\Gamma_{R}$ each of which has end spaces $C_{L}$ and $C_{R}$, respectively. Let $T_{L} \subset T$ be the maximal tree of $\overline{\Gamma_{L}}$. Define for each $n \in \Z$:

\begin{align*}
    \Gamma_{n} &\defeq \begin{cases} 
    \overline{\Gamma_{L} \cup B_{n}(x_{0})} &\text{ if }n \geq 0 \\
    \left(\Gamma_{L} \setminus B_{n}(x_{0})\right) \cup T_{L} &\text{ if }n < 0 
    \end{cases},
\end{align*}
where $B_{n}(x_{0})$ is the open metric ball of radius $n$ about $x_{0}$. Note that $\Gamma_{0} = \Gamma_{L}$. See \Cref{fig:fluxsetup} for an example of these sets.

\begin{figure}[ht!]
	    \centering
	    \includegraphics[width=\textwidth]{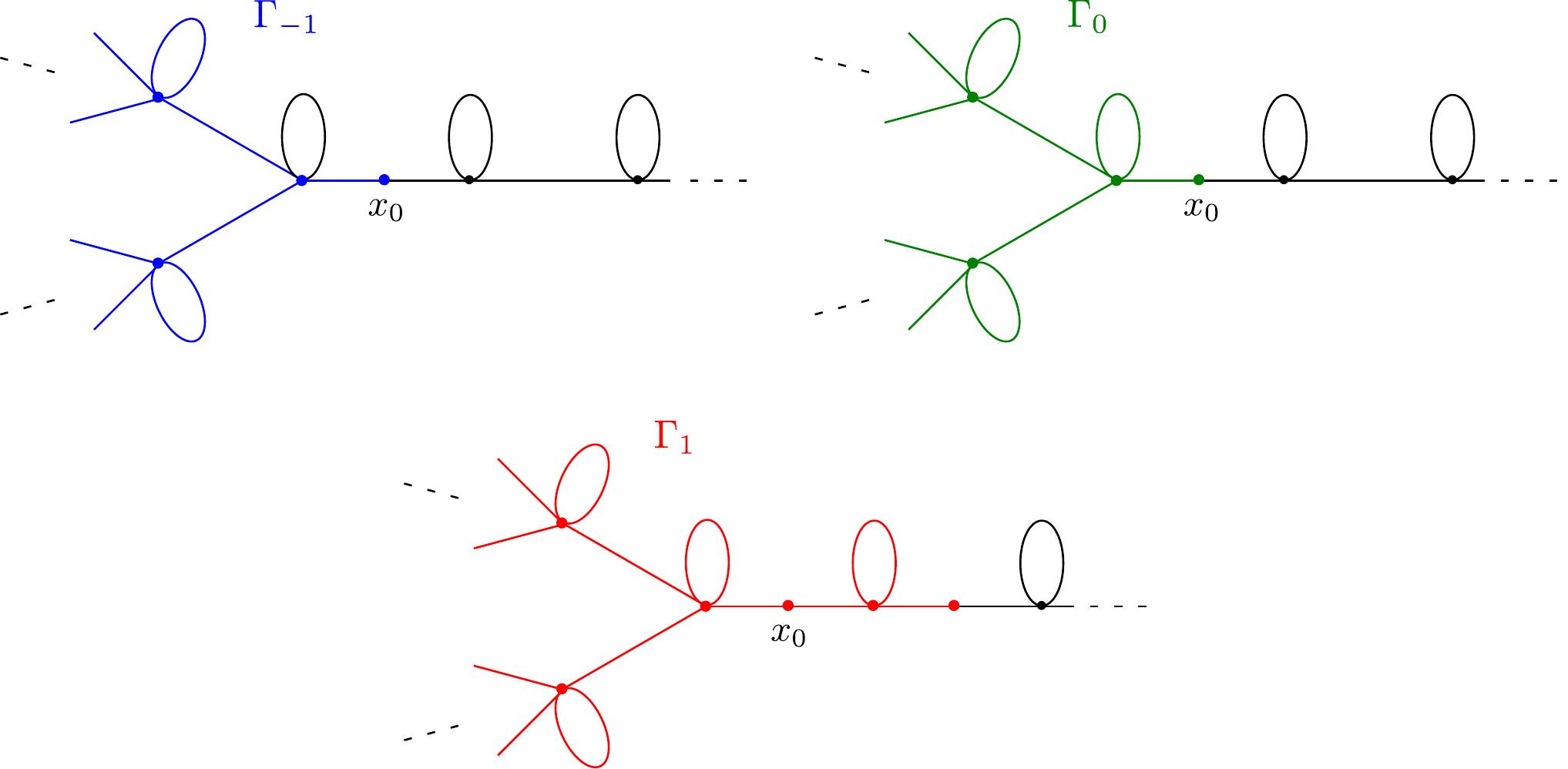}
		\caption{Examples of $\G_{-1}$, $\G_0$ and $\G_{1}$ in blue, green, and red respectively.} 
	    \label{fig:fluxsetup}
\end{figure}

Now for each $n \in \Z$, $\Gamma_{n}$ determines a free factor $A_{n} = \pi_{1}(\Gamma_{n},x_{0})$ of the infinite-rank free group $F=\pi_{1}(\Gamma,x_{0})$. Observe that $A_{n}$ has infinite rank and corank within $F$. We also note that these subgraphs and corresponding free factors are totally ordered. That is, if $n,m \in \Z$ with $n<m$ then $\Gamma_{n} \subset \Gamma_{m}$ and $A_{n} \le A_{m}$, which further implies that $A_n$ is a free factor of $A_m$ by \Cref{LEM:subfreefact}.

\begin{LEM} \label{LEM:vertexistence}
    Let $f\in \PPHE(\G)$. Then for any $n \in \Z$, there exists some $m \in \Z$ such that $\Gamma_{m}$ contains both $\Gamma_{n}$ and $f(\Gamma_{n})$.
\end{LEM}

\begin{proof}
    It suffices to find $m>n$ such that $\Gamma_m \supset f(\Gamma_n)$. Note first that $f(\Gamma_n)$ has end space equal to $C_{L}$ since $f$ is pure. Then there exists a common neighborhood $M$ of $C_L$ contained in both $\Gamma_n$ and $f(\Gamma_n)$. Since $f(\Gamma_n) \setminus M$ is bounded and the collection $\{\Gamma_m \setminus M\}_{m>n}$ exhausts $\Gamma \setminus M$, it follows that there exists some $m>n$ such that $f(\Gamma_n) \setminus M \subset \Gamma_m \setminus M$. By construction we also have $M \subset \Gamma_n \subset \Gamma_m$, and it follows that $f(\Gamma_n) \subset \Gamma_m$.
\end{proof}

\begin{RMK}
    From here on we write $\cork(A_{m},A_{n})$ and $\cork(A_{m},f_{*}(A_{n}))$ to implicitly include a choice of basepoint that realizes these free factors as based fundamental groups. That is, we can take $x_{0}$ as above, and then we have $\cork(A_{m},A_{n}) = \cork(\pi_{1}(\Gamma_{m},x_{0}),\pi_{1}(\Gamma_{n},x_{0}))$ and $\cork(A_{m},f_{*}(A_{n})) = \cork(\pi_{1}(\Gamma_{m},f(x_{0})),f_{*}(\pi_{1}(\Gamma_{n},x_{0})))$.
\end{RMK}

\begin{COR} \label{COR:admisexist}
    Let $f \in \PPHE(\G)$. Then for any $n \in \Z$, there exists some $m>n$ so that 
    \begin{enumerate}[(i)]
        \item
            $A_{n}$ and $f_{*}(A_{n})$ are free factors of $A_{m}$, and
        \item
            both $\cork(A_{m},A_{n})$ and $\cork(A_{m},f_{*}(A_{n}))$ are finite.
    \end{enumerate}
\end{COR}

\begin{proof} 
    For (i) we apply  \Cref{LEM:vertexistence} to find an $m \in \Z$ such that both $\Gamma_n$ and $f(\Gamma_n)$ are contained in $\Gamma_m$. Since $A_n = \pi_1(\Gamma_n)$ and $f_*(A_n) < \pi_1(f(\Gamma_n))$, both $A_n$ and $f_*(A_n)$ are subgroups of $A_m$. Moreover, as $f_*$ is a $\pi_1$-isomorphism, both $A_n$ and $f_*(A_n)$ are free factors of $\pi_1(\Gamma)$. Hence, by  \Cref{LEM:subfreefact}, we conclude $A_n$ and $f_*(A_n)$ are free factors of $A_m$.
    In particular, the quantities $\cork(A_m,A_n)$ and $\cork(A_m,f_*(A_n))$ are well-defined.
    
    For (ii) we first let $g$ be a proper homotopy inverse for $f$. By  \Cref{LEM:vertexistence} we can find some $m$ such that $\Gamma_{n},f(\Gamma_{n})$ and $g(\Gamma_{n})$ are contained in $\Gamma_{m}$. Note that $\cork(A_{m},A_{n})$ is finite by definition. 
    
    Suppose, for the sake of contradiction, that $\cork(A_{m},f_{*}(A_{n}))$ is infinite. Then we have that for every integer $i < n$, $A_{i}$ is not a subgroup of $f_{*}(A_{n})$. Otherwise, $\cork(A_{m},f_{*}(A_{n})) < \cork(A_{m},A_{i}) < \infty$. For each $i<n$, there exists a basis element $a_{i_{j}}$ of $A_{i}$ such that $a_{i_{j}} \not\in f_*(A_n)$. Since $\bigcap_{i < n} A_{i} = \emptyset$, we can pass to an infinite sequence of \textit{distinct} basis elements $\{a_{i_{j}}\}$ such that $a_{i_{j}} \notin f_{*}(A_{n})$. 
    Therefore $g_{*}(a_{i_{j}}) \notin A_{n}$ for all $i_{j}$. Let $\alpha_{i_{j}}$ denote a loop representing $a_{i_{j}}$ in $\Gamma_{m}$. Note that $g_{*}(a_{i_{j}}) \notin A_{n}$ implies that $g(\alpha_{i_{j}}) \not\subset \Gamma_{n}$ so that $g(\alpha_{i_{j}}) \cap \overline{\Gamma_{m}\setminus \Gamma_{n}} \neq \emptyset$. However, $\overline{\Gamma_{m}\setminus \Gamma_{n}}$ is compact and $\overline{\G_m \setminus \G_n} \cap \G_n$ is finite, so there must exist some point $x \in \overline{\Gamma_{m}\setminus \Gamma_{n}}$ such that $g^{-1}(x)$ is infinite, contradicting the fact that $g$ was a proper map. Thus we conclude that $\cork(A_{m},f_{*}(A_{n}))$ is finite. 
\end{proof}

\begin{DEF}
    Given $f \in \PPHE(\Gamma)$ we say that a pair of integers, $(m,n)$, with $n<m$, is \textbf{admissible} for $f$ if 
    \begin{enumerate}[(i)]
        \item
            $A_{n}$ and $f_{*}(A_{n})$ are free factors of $A_{m}$, and
        \item
            both $\cork(A_{m},A_{n})$ and $\cork(A_{m},f_{*}(A_{n}))$ are finite.
    \end{enumerate}
\end{DEF}

 \Cref{COR:admisexist} shows that for any $f \in \PPHE(\G)$ and $n \in \Z$, there exist $m \in \Z$ such that $(m,n)$ is admissible for $f$. For a map $f \in \PPHE(\Gamma)$ and an admissible pair $(m,n)$ for $f$, we let
\begin{align*}
    \phi_{m,n}(f) := \cork(A_{m},A_{n}) - \cork(A_{m},f_{*}(A_{n})). 
\end{align*}

\begin{LEM} \label{LEM:ind}
    This quantity is independent of the choice of admissible pair $(m,n)$. That is, if $(m,n)$ and $(m',n')$ are two admissable pairs for the map $f \in \PPHE(\Gamma)$ then $\phi_{m,n}(f) = \phi_{m',n'}(f)$. 
\end{LEM}

\begin{proof} 
    Let $f \in \PPHE(\Gamma)$. We first consider the case with admissible pairs $(m,n)$ and $(m',n)$ for $f$ with $m<m'$. Then by the additivity of $\cork$ we have 
    \begin{align*}
        \cork(A_{m'},A_{m}) &= \cork(A_{m'},A_{n}) - \cork(A_{m},A_{n}), \text{ and} \\
        \cork(A_{m'},A_{m}) &= \cork(A_{m'},f_{*}(A_{n})) - \cork(A_{m},f_{*}(A_{n})).
    \end{align*}
    Now
    \begin{align*}
        &\phantom{==}\phi_{m',n}(f) - \phi_{m,n}(f) \\
        &= \cork(A_{m'},A_{n})-\cork(A_{m},A_{n}) - \cork(A_{m'},f_{*}(A_{n})) - \cork(A_{m},f_{*}(A_{n})) \\
        &= \cork(A_{m'},A_{m}) - \cork(A_{m'},A_{m}) = 0.
    \end{align*}
    
    Next, let $(m,n)$ and $(m',n')$ be any two admissible pairs for $f$, without loss of generality we assume $m<m'$. By definition, $(m',n)$ must also be admissible for $f$. Then we can apply the above argument to reduce to considering just the pairs $(m',n)$ and $(m',n')$. Suppose that $n<n'$. Then we have $A_{n} < A_{n'}$ and $f_{*}(A_{n}) < f_{*}(A_{n'})$. Once again by additivity we have
    \begin{align*}
        \cork(A_{n'},A_{n}) &= \cork(A_{m'},A_{n}) - \cork(A_{m'},A_{n'}), \text{ and} \\
        \cork(f_{*}(A_{n'}),f_{*}(A_{n})) &= \cork(A_{m'},f_{*}(A_{n})) - \cork(A_{m'},f_{*}(A_{n'})).
    \end{align*}
    Note also that the function $\cork$ is invariant with respect to group isomorphisms so that $\cork(f_{*}(A_{n'}),f_{*}(A_{n})) = \cork(A_{n'},A_{n})$. Thus as before we have
    \begin{align*}
        &\phantom{==}\phi_{m',n}(f) - \phi_{m',n'}(f) \\
        &= \cork(A_{m'},A_{n})-\cork(A_{m'},A_{n'}) - \cork(A_{m'},f_{*}(A_{n})) - \cork(A_{m'},f_{*}(A_{n'})) \\
        &= \cork(A_{n'},A_{n}) - \cork(f_{*}(A_{n'}),f_{*}(A_{n})) = 0 \qedhere
    \end{align*}
\end{proof}

This allows us to define a function
\begin{align*}
    \phi:\PPHE(\Gamma) \rightarrow \Z \\
    f \mapsto \phi_{m,n}(f)
\end{align*}
where $(m,n)$ is any admissible pair for $f$. 

\begin{PROP}
\label{prop:flux}
    The map $\phi$ is a homomorphism. 
\end{PROP}

\begin{proof}
    First note that $\phi(\Id)=0$. Let $f,g \in \PPHE(\Gamma)$ and let $n \in \Z$.  By  \Cref{COR:admisexist} we can find some $m$ so that $(m,n)$ is simultaneously admissible for all three maps $f,g$, and $fg$. Now we have 
    \begin{align*}
        \phi_{m,n}(fg) &= \cork(A_{m},A_{n}) - \cork(A_{m},(fg)_{*}(A_{n})) \\
        &= \cork(A_{m},A_{n}) - \cork(A_{m},f_{*}(A_{n})) + \cork(A_{m},f_{*}(A_{n})) - \cork(A_{m},(fg)_{*}(A_{n})) \\
        &= \phi_{m,n}(f) + \cork(f_{*}^{-1}(A_{m}),A_{n}) - \cork(f_{*}^{-1}(A_{m}),g_{*}(A_{n}))\\
        &= \phi(f) + \phi(g).
    \end{align*}
    Note that the last step follows by picking some $m'$ such that $A_{m'}$ contains $f_{*}^{-1}(A_{m})$, $A_n$ and $g_*(A_n)$ as free factors and applying the same argument used to prove  \Cref{LEM:ind}. 
\end{proof}

\begin{LEM}
    \label{lem:sameflux_homotopic}
    If $f,g \in \PPHE(\Gamma)$ are properly homotopic, then $\phi(f)=\phi(g)$. 
\end{LEM}

\begin{proof} 
    We first claim that if a map $h \in \PPHE(\Gamma)$ is properly homotopic to the identity, then $\phi(h) = 0$. Let $(m,n)$ be an admissible pair for $h$ and $H_t$ be a proper homotopy of $\G$ so that $H_0 = \Id$ and $H_1 = h$, and define $\beta(t) = H_t(x_0)$. We will prove $\cork(A_m, h_* (A_n)) = \cork(A_m,A_n)$.
    If necessary, enlarge $\G_m$ so that it contains the image of $\beta$. The induced map $h_*:\pi_1(\G,x_0) \arr \pi_1(\G,h(x_0))$ satisfies 
    \[c_{\beta}\circ h_*=\Id: \pi_1(\G,x_0) \arr \pi_1(\G,x_0)\]
    where $c_{\beta}:\pi_1(\G,h(x_0))\arr \pi_1(\G,x_0)$ is conjugation by the path $\beta$. Then
    \begin{align*}
        \cork(A_m,h_*(A_n)) &=\cork\left(\pi_1(\G_m,h(x_0)), h_*(\pi_1(\G_n,x_0)) \right) \\
        &=  \cork\left(c_{\beta}\left(\pi_1(\G_m,h(x_0))\right), c_{\beta}h_*(\pi_1(\G_n,x_0)) \right) \\
        &=\cork(\pi_1(\G_m,x_0),\pi_1(\G_n,x_0)) \\
        &=\cork(A_m,A_n),
    \end{align*}
    as desired. Note that we required the path $\beta$ to be contained in $\G_m$ in order to write $c_{\beta}(\pi_1(\G_m,h(x_0)))=\pi_1(\G_m,x_0)$. 
    
    Now suppose $f$ and $g$ are properly homotopic. Then there exists a proper homotopy inverse $\overline{g}$ of $g$ such that $f\overline{g} \simeq id$. By the first assertion and  \Cref{prop:flux}, we have $\phi(f) + \phi(\overline{g}) = \phi(f\overline{g})= 0$, so $\phi(f) = -\phi(\overline{g})$. Also, by definition $g\overline{g} \simeq id$, so $\phi(g) + \phi(\overline{g}) = \phi(g\overline{g})=0$. Hence, $\phi(f) = -\phi(\overline{g}) = \phi(g)$, concluding the proof.
\end{proof}

Thus we obtain a well-defined homomorphism, which we call a \textbf{flux map}:
\begin{align*}
    \Phi:\PMap(\Gamma) \rightarrow \Z \\
    [f] \mapsto \phi(f).
\end{align*}

Finally, to see that flux maps are nontrivial, we use the loop shifts defined in \Cref{ss:loopshifts}. 
We say that a loop shift, $h$, \textbf{crosses} a partition $\cP = C_{L} \sqcup C_{R}$ of $E(\Gamma)$ if $h^{+}$ and $h^{-}$ are contained in different partition sets.

\begin{PROP}
    The homomorphism $\Phi$ satisfies:
    \begin{enumerate}[(i)]
        \item  
            $\Phi(f) = 0$ for all $f \in \PMapcc{\Gamma}$,
        \item
            $\Phi([h]) = \pm 1$ where $h$ is a loop shift which crosses the partition used to define $\Phi$. 
    \end{enumerate}
\end{PROP}

\begin{proof}
    (i) We first let $g \in \PMapc(\Gamma)$. Then, after potentially modifying $g$ by a proper homotopy, $g$ is totally supported on some compact subset $K \subset \Gamma$. We can then find some $n$ such that $(n,n)$ is an admissable pair for $g$ and $\Gamma_{n} \cap K$ is a (possibly empty) tree. Thus $g_{*}$ is the identity map on $A_{n}$. 
    
    Next we apply a theorem of Dudley \cite{dudley1961}
    which states that any homomorphism from a Polish group to $\Z$ is continuous to conclude that $\Phi(f) = 0$ for any $f \in \PMapcc{\Gamma}$. Note that $\PMapcc{\Gamma}$ is a closed subgroup of $\Map(\Gamma)$ and thus Polish. 
    
    Property (ii) follows from the definition of the loop shift. Assume that $h^{-} \in C_{L}$ and $h^{+} \in C_{R}$. Let $m >0$ be such that $\Gamma_{m}$ contains one more loop of $\rho(\Lambda)$ than $\Gamma_{0}$. This is possible because $h$ crosses the partition $C_{L} \sqcup C_{R}$. Then we have 
    \[
    \Phi([h]) = \phi(h) = \cork(A_m,A_0)-\cork(A_m,h_*(A_0)) = 1-0 = 1.
    \]
    Note that if instead $h^{-} \in C_{R}$ and $h^{+} \in C_{L}$ the same argument would show that $\Phi([h]) = -1$. 
\end{proof}

\begin{proof}[Proof of \Cref{thm:twoendnotCB}]
    By Dudley's automatic continuity property \cite{dudley1961} the map $\Phi$ is continuous. Thus, we get a continuous action of $\PMap(\Gamma)$ on the metric space $\Z$ with unbounded orbits. 
\end{proof}
\begin{RMK}
\label{rmk:fluxonlarger}
    We could construct the flux map on any subgroup $H$ of $\Map(\G)$ that fixes the two ends accumulated by the loops but not necessarily fixes the other ends. Following the same argument, we can show that $H$ is not CB.
\end{RMK}

The proof of \Cref{thm:twoendnotCB} shows that we have nontrivial homomorphisms to $\Z$ so that $H^{1}(\PMap(\Gamma);\Z)\neq 0$. However, with a more delicate choice of flux maps, we can get a better lower bound on the rank of $H^{1}(\PMap(\Gamma);\Z)$.

\begin{PROP}
    \label{prop:cohomology}
    If $n=|E_\ell(\G)|\ge 2$ and finite, then
    $\rk(H^1(\PMap(\G);\Z)) \ge n-1$. If $|E_{\ell}(\G)| = \infty$ then $H^{1}(\PMap(\G);\Z) = \bigoplus_{i=1}^{\infty} \Z$.
\end{PROP}

To prove this, we refine the notation for flux map as follows.
If $\cP = C_{L} \sqcup C_{R}$ is a partition of the ends of $E(\Gamma)$ into two sets such that $C_{L} \cap E_{\ell}(\Gamma) \neq \emptyset$ and $C_{R} \cap E_{\ell}(\G) \neq \emptyset$ then we denote the resulting flux map as $\Phi_{\cP}$. 

\begin{proof}[Proof of \Cref{prop:cohomology}]
    Let $n=|E_\ell(\G)|$ be finite. Identify $E_\ell(\G)$ with the $n$-set $\{0,1,\ldots,n-1\}$. 
    Then for $i=1,\ldots,n-1$, there exists pairwise disjoint neighborhoods $U_{i}$ of each of the $i$ in $E(\G)$. Define the partition $\cP_i=U_{i} \sqcup \left(E_\ell(\G)-U_{i}\right)$ and denote by $h_i$ a loop shift associated to a line joining the ends $\{0,i\}$. Then by construction each $h_i$ crosses $\cP_i$ and it follows that:
    \[
     \Phi_{\cP_i}(h_j) = \begin{cases}
        1, & \text{if } i = j\\
        0, & \text{if } i\neq j.
    \end{cases}
    \]
    This implies $\Phi_{\cP_1},\ldots,\Phi_{\cP_{n-1}}$ are linearly independent in $H^1(\PMap(\G);\Z)$, so we conclude $\rk(H^1(\PMap(\G);\Z)) \ge n-1$.
    
    If $\lvert E_{\ell}(\G) \rvert = \infty$ we similarly enumerate a collection of ends in $E_{\ell}(\Gamma)$ as $\{0,1,\ldots\}$ and can find pairwise disjoint neighborhoods $U_{i}$ in $E(\Gamma)$. This is possible since end spaces are totally disconnected. We then similarly define our partitions and see that the associated flux maps are linearly independent. This gives a lower bound on the rank of $H^{1}(\PMap(\Gamma);\Z)$. To see that the cohomology cannot be larger than countably infinite, note that $\PMap(\G)$ has a countable dense subset. Then since any homomorphism to $\Z$ must be continuous there are only countably many unique homomorphisms to $\Z$. 
\end{proof}
    
\section{Locally CB Classification}\label{ssec:LocalCB}

In this section we use the tools developed above to also give a full classification of which graphs have \emph{locally coarsely bounded (locally CB)} pure mapping class groups. Recall that because $\PMap(\G)$ are topological groups, they are locally CB if a neighborhood of the identity is CB. 

\begin{PROP} \label{prop:finitelocCB}
    Let $\G$ be a locally finite, infinite graph of \emph{finite} rank, then $\{\Id\}$ is an open set in $\PMap(\G)$. In particular, $\PMap(\G)$ is discrete and locally CB.
\end{PROP}

\begin{proof}
    Take $K=\G_c$, the core graph of $\G$.
    Then $\cV_K=\{\Id\}$ because the complementary components of $K$ are each trees, and every pure mapping class totally supported on a tree is properly homotopic to the identity (\Cref{prop:rank0}).
\end{proof}

\begin{PROP} \label{prop:infendloopslocCB}
    Let $\G$ be a locally finite, infinite graph with \emph{infinitely many} ends accumulated by loops, then $\PMap(\G)$ is not locally CB.
\end{PROP}

\begin{proof}
    Let $K$ be any compact set in $\G$. Because $\G$ has infinitely many ends accumulated by loops there is at least one component of $\G\setminus K$ with two or more ends accumulated by loops, call these ends $e_{-}$ and $e_{+}$. Let $\cP = C_{L} \sqcup C_{R}$ be any partition of $E(\G)$ that separates $e_{-}$ and $e_{+}$ and let $h$ be a loop shift totally supported on $\G\setminus K$ with $h^{-} = e_{-}$ and $h^{+} = e_{+}$. Then $h \in \cV_{K}$ and $\Phi_{\cP}(h) = 1$ so that $\Phi_{\cP}\vert_{\cV_{K}}$ is a nontrivial homomorphism on $\cV_K$ to $\Z$. Furthermore, this restriction is continuous again by Dudley's automatic continuity \cite{dudley1961} as $\cV_{K}$ is a clopen subgroup of $\PMap(\G)$ (hence Polish). Therefore $\cV_K$ is not CB.
\end{proof}

\begin{PROP}\label{prop:finiteendloopsloccb}
    Let $\Gamma$ be a locally finite, infinite graph with infinite rank and $\lvert E_{\ell}(\G) \rvert<\infty$. Then $\PMap(\G)$ is locally CB if and only if $\G\setminus \G_{c}$ has only \emph{finitely many} components $T_{1},\ldots,T_{m}$ such that $|E(T_{i})| = \infty$.
\end{PROP}

\begin{proof}
    Suppose first $\Gamma \setminus \Gamma_c$ has infinitely many components with infinite end spaces. Then given any compact set $K$, there is a component $A$ of $\G \setminus K$ that has infinite rank and $A \setminus \Gamma_{c}$ having infinite end space. We can apply \Cref{prop:lengthfunction} with $T=A \setminus \G_c \subset \G \setminus K$ to build a length function on $\PMap(\G)$ that is unbounded on $\cV_{K}$. Thus no identity neighborhood is CB.

    Conversely, suppose $\Gamma \setminus \Gamma_c$ has finitely many components $T_1,\ldots,T_m$ such that $|E(T_i)|=\infty$ for $i=1,\ldots,m$.
    Also, for $i=1,\ldots,m$ let $x_i:= \overline{T_i} \cap \G_c$ (some of them might be the same). Then $\G \setminus \{x_1,\ldots,x_m\}$ will contain
    $T_1,\ldots T_m$ as disjoint components. Let $K$ be the minimal spanning tree of $\{x_1,\ldots,x_m\}$.
    Note
    \[
        E(\G)\setminus E_\ell(\G) = E(\G \setminus \G_c)
    \]
    and $T_1,\ldots,T_m \subset \G \setminus \G_c$, so it follows that $E(T_i) \subset E(\G) \setminus E_\ell(\G)$ for all $i =1,\ldots,m$.
    
    Moreover, as $n:=|E_\ell(\G_c)|=|E(\G_c)|<\infty$, we may enlarge $K$ so that we can ensure $\G \setminus K$ decomposes as
    \[
        \G \setminus K = \left(T_1 \sqcup \ldots \sqcup T_m\right) \sqcup \left(\G_{1} \sqcup \ldots \sqcup \G_{n}\right) \sqcup \G'
    \] where $\G_{1},\ldots,\G_{n}$ are graphs with only a single end accumulated by loops, and $\G'$ is some (possibly not connected) subgraph of $\G$. Since $K$ is compact $\G'$ has only finitely many components, and by construction of $K$, it follows that $|E(\G')|<\infty$, and $|E_\ell(\G')|=0$. Enlarging $K$ more to include all the loops in $\G'$, we may even assume $\rk(\G') = 0$, so it is a forest. See \Cref{fig:locCB}.
    \begin{figure}[ht!]
        \centering
        \includegraphics[width=.8\textwidth]{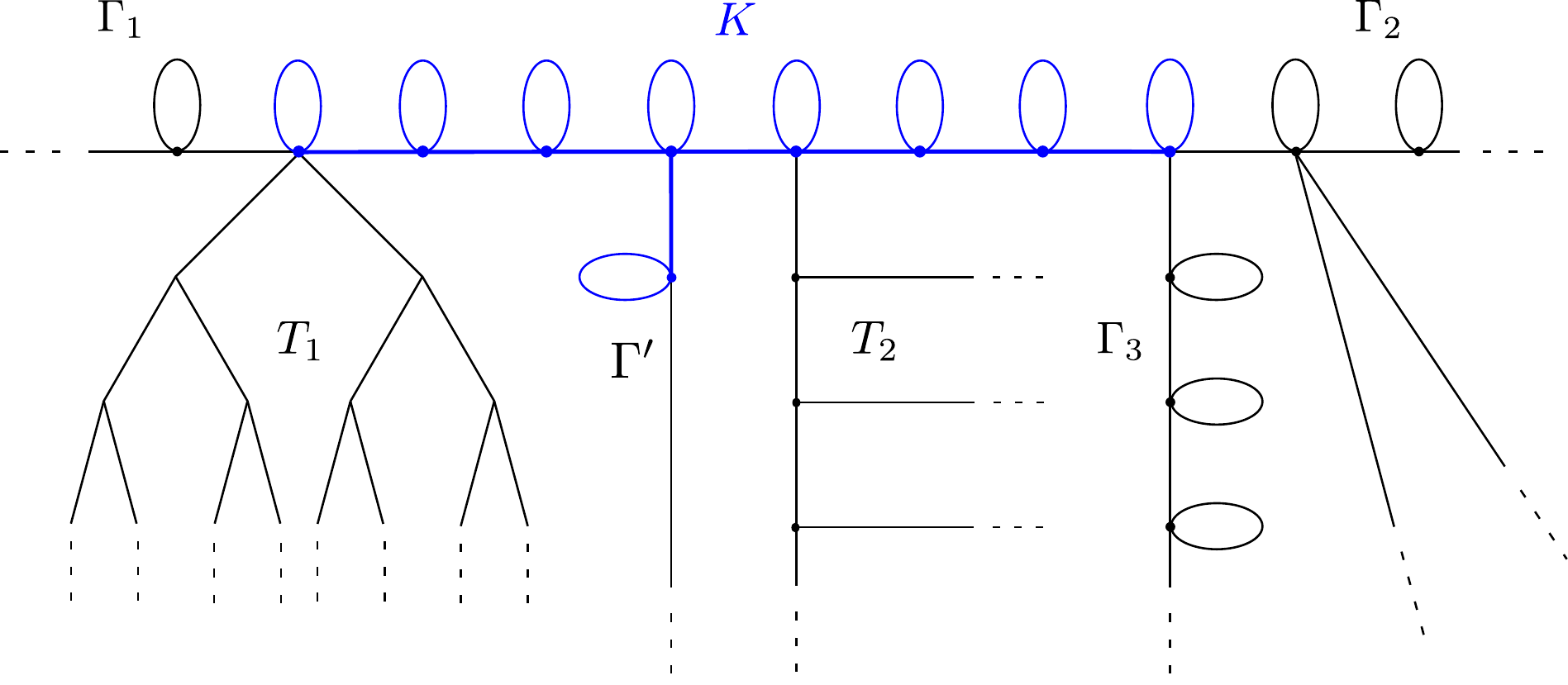}
        \caption{The complementary components of $K$ are subgraphs with infinite end space $T_1,\ldots,T_m$, monster graphs $\G_{1},\ldots,\G_{n}$, or a forest $\G'$.}
        \label{fig:locCB}
    \end{figure}

    Finally, with this choice of $K$ it follows that every element $g \in \cV_{K}$ has to fix the complementary components setwise.
    Note the pure mapping class group of each of trees $T_1,\ldots,T_m$ or $\G'$ is trivial.
    Hence we have that $\cV_K \cong \prod_{i=1}^n\PMap(\G_i)$.
    On the other hand, $\PMap(\G_{i})$ is CB by \Cref{THM:oneendcb}. Since the product of CB groups is CB (cf. \cite[Lemma 3.36]{rosendal2022}), it follows that $\cV_K$ is CB, concluding that $\PMap(\G)$ is locally CB.
\end{proof}

\section{Asymptotic Dimension} \label{sec:asymdim}  
  \begin{DEF}[Asymptotic Dimension]
    \label{def:asdim}
        A pseudo-metric space $(X,d)$ has \textbf{asymptotic dimension $\le n$}, denoted by $\asdim(X,d)\le n$, if for every $R>0$ there exist a uniformly bounded cover $\cV$ of $X$  such that every $R$-ball in $X$ intersects at most $n+1$ elements of $\cV$. We say $\asdim(X,d)=n$ if $\asdim(X,d) \le n$ but $\asdim(X,d) \not\le n-1$.
    \end{DEF}
    
    The asymptotic dimension is well-defined over coarse equivalence classes. 
    
    \begin{LEM}[{\cite[Proposition 3.B.19]{cornulier2016}}]
       \label{lem:asdimCE}
         If $f:X \to Y$ is a coarse embedding between pseudo-metric spaces, then $\asdim X \le \asdim Y$. In particular, if $f$ is a coarse equivalence, then $\asdim X = \asdim Y$.
         Hence, the asymptotic dimension is invariant under a coarse equivalence.
     \end{LEM}
    
     For a locally coarsely bounded $\PMap(\G)$, we know $\PMap(\G)$ has a well-defined coarse equivalence type by \Cref{prop:WDCEpmap}. Hence, along with \Cref{lem:asdimCE}, the asymptotic dimension of a locally coarsely bounded $\PMap(\G)$ is well-defined. 
     
\subsection{Pure Mapping Class Groups of Zero Asymptotic Dimension}
\label{ss:asdim0}
     As an application of \Cref{prop:PformsLeaves}, we show:
     
     \begin{THM}
        \label{thm:asdim0}
         Let $\G$ be a locally finite, infinite graph with $\PMap(\G)$ locally CB. If $|E_\ell|=1$ then, the asymptotic dimension of $\PMap(\G)$ is $0$.
     \end{THM}
     
    \begin{RMK}
    \label{rmk:asdim}
    As we have seen in \Cref{SEC:Length}, a pseudo-metric space $(\hat{X},\hat{d})$ induces a metric space $(X,d)$ by collapsing all the pairs of points of distance zero. The induced metric space $(X,d)$ is canonical in the sense that it is the smallest metric space that preserves positive distances in $\hat{d}$.
    
    We show that in fact $\asdim (\hat{X},\hat{d}) = \asdim(X,d)$. Let $q:(\hat{X},\hat{d}) \to (X,d)$ be the quotient map. Pick $\hat{x} \in \hat{X}$, and $x \in X$ with $q(\hat{x})=x$, and consider $R$-balls 
    $B_R(\hat{x}) \subset (\hat{X},\hat{d})$ and $B_R(x) \subset (X,d)$. For any open set $\hat{U} \subset \hat{X}$, we have:
    \[
        B_R(\hat{x}) \cap \hat{U} \neq \emptyset \quad \text{ if and only if } \quad B_R(x) \cap U \neq \emptyset,
    \]
    because any positive distance in $\hat{d}$ is preserved under $q$.
    Hence, considering \Cref{def:asdim}, $B_R(\hat{x})$ and $B_R(x)$ each intersects the same number of  uniformly bounded sets from the open covers $\{\hat{U}_\alpha\}$ and $\{U_\alpha\}$ respectively. Therefore, we can conclude that 
    \[
        \asdim (\hat{X},\hat{d}) = \asdim (X,d).
    \]
    \end{RMK}
    
    Before we begin with the proof of \Cref{thm:asdim0} we need to fix an appropriate length function on our groups. 
    
    \begin{DEF}
        A length function on a group, $\ell: G \rightarrow [0,\infty)$ is \textbf{full} if $\ell^{-1}(\{0\})$ is CB in $G$. 
    \end{DEF}
    
    We call these length functions ``full" because they will be used to capture the ``full" geometry of our groups. 
    
    \begin{LEM}\label{lem:fulllength}
        Let $\G$ be a locally finite, infinite graph with $\PMap(\G)$ locally CB, but not globally CB. If $|E_{\ell}|=1$ then $\PMap(\G)$ admits a full length function. In fact, the length function defined in \Cref{prop:lengthfunction} can be taken to be full. 
    \end{LEM}
    
    \begin{proof}
        Consider the connected components $\{T_i\}$ of $\G \setminus \G_c$. Since $\PMap(\G)$ is not CB, there exists an accumulation point in $E(\G)\setminus E_\ell(\G) = \bigsqcup E(T_i)$ so some $T_i$ has infinite end space. At the same time, $\PMap(\G)$ is locally CB and hence only finitely many of the $T_{i}$ have infinite end spaces. Thus, we can modify $\G$ by a proper homotopy equivalence in order to guarantee that $\G$ has only a single $T_{0}$ with an infinite end space. Use this $T_{0}$ to build the unbounded length function, $\ell$, as in \Cref{prop:lengthfunction}. 
        
        We claim that this length function is full. Indeed, $\ell^{-1}(\{0\})$ is exactly the subgroup of $\PMap(\G)$ that fixes $T_{0}$ up to proper homotopy. Since the end space of $\G \setminus T_{0}$ is discrete we see that $\ell^{-1}(\{0\})$ is CB by \Cref{THM:oneendcb}.
    \end{proof}
    
    \begin{LEM}\label{lem:fullCnCB}
        Let $\G$ be a locally finite, infinite graph with $|E_\ell|=1$ and with $\PMap(\G)$ locally CB, but not globally CB. Let $\ell$ be the unbounded, full, length function given by \Cref{lem:fulllength}. Then $H_{n} = \ell^{-1}([0,n])$ is CB in $\PMap(\G)$ for all $n \in \Z_{\geq 0}$. 
    \end{LEM}
    
    \begin{proof}
    Let $\Delta_n \subset \G$ be the set of points in $\G$ of distance $<n$ 
    from the core graph $\G_c$. Since $\PMap(\G)$ is locally CB, there are finitely many components of $\G \setminus \G_c$ with infinite end space. Thus, by the classification of graphs (\Cref{thm:InfGraphClass}), it follows that $\Delta_n$ is properly homotopy equivalent to one of monster graphs $\G_N$ for some $N\in \Z_{\ge 0} \cup \{\infty\}$ (See \Cref{ss:stdforms} for the definition.) Because every mapping class in $H_n = \ell^{-1}([0,n])$ can be properly homotoped to be totally supported in $\Delta_n$, for each $n$ it follows that $H_n \cong \PMap(\Delta_n) \cong \PMap(\G_N)$ for some $N \in \Z_{\ge 0} \cup \{\infty\}$. However, we know $\PMap(\G_N)$ is CB for every such $N$ by \Cref{THM:oneendcb}, concluding the proof.
    \end{proof}
     
     \begin{proof}[{Proof of \Cref{thm:asdim0}}]
       First, note that if $\PMap(\G)$ is CB, it has asymptotic dimension $0$ as it is quasi-isometric to a point. Assume that $\PMap(\G)$ is not CB. 
       
       Fix an unbounded, full, length function $\ell$ on $\PMap(\G)$ as in the previous lemma. Observe that the induced pseudo-metric $\hat{d}$ on $\PMap(\G)$ is left-invariant and coarsely proper because each metric ball is exactly given by the sets $H_{n} = \ell^{-1}([0,n])$ defined above. Thus they are CB by \Cref{lem:fullCnCB}.
       
       Then by \Cref{prop:WDCEtype} and \Cref{lem:asdimCE}, the asymptotic dimension of $\PMap(\G)$ is well-defined and can be calculated as $\asdim \PMap(\G) = \asdim (\PMap(\G),\hat{d})$.
       By \Cref{rmk:asdim}, this is further equal to $\asdim (\cP,d)$.
       However, by \Cref{prop:PformsLeaves}, $\cP$ can be identified as the leaves in some simplicial tree, so it follows that $\asdim (\cP,d) = 0$ and $\asdim (\PMap(\G),\hat{d}) =0$ as well.
     \end{proof}
    
\subsubsection{Bounded Geometry}

We take a brief detour here and discuss groups of bounded geometry. The tools we have developed, in particular, the existence of unbounded, full, length functions, allow us to show that some of our groups do \emph{not} have bounded geometry. 

\begin{DEF}[{\cite[Definition 3.6]{rosendal2022}}]
    A Polish group $G$ is said to have \emph{bounded geometry} if there is a coarsely bounded set $A\subset G$ covering every coarsely bounded set $B\subset G$ by finitely many left-translates of $A$. That is, $B \subset \mathcal{F}A$ for some finite set $\mathcal{F}\subset G$ which depends on $B$.
\end{DEF}

Generally speaking, groups of bounded geometry should be ``most closely'' related to locally compact groups. Rosendal asks the following question, which remains open as our groups fail to provide examples.

\begin{Q}[{\cite[Problem A.15]{rosendal2022}}]
     Find a non-locally compact, topologically simple, Polish group of bounded geometry that is not coarsely bounded.
\end{Q}
     
\begin{PROP}\label{prop:BddGeometry}
    Let $\G$ be a locally finite, infinite graph with $|E_{\ell}|=1$ and $\PMap(\G)$ locally CB, but not globally CB. Then $\PMap(\G)$ does not have bounded geometry.
\end{PROP}

\begin{proof}
    Let $\ell$ be an unbounded, full, length function given by \Cref{lem:fulllength}, and let $T$ be the component of $\G \setminus \G_{c}$ used to define $\ell$ and $H_n=\ell\inv\left([0,n]\right)$. Note that every CB set in $\PMap(\G)$ is contained in some $H_n$. Furthermore, each of the sets $H_{n}$ are CB by \Cref{lem:fullCnCB}.

    To show that $\PMap(\G)$ does not have bounded geometry we will show that finitely many translates of $H_n$ cannot cover $H_{n+1}$. Let $L$ be a geodesic line in $T$ which is at distance $n$ from $\G_c$.  Then the action of $H_n$ on $L$ is trivial, while the $H_{n+1}$ orbit of $L$ is infinite.  Let $\mathcal{F}$ be any finite set in $\PMap(\G)$, then $\left\lvert \mathcal{F} H_n \cdot L \right\rvert= \left\lvert \mathcal{F}\cdot L \right\rvert \leq \left\lvert \mathcal{F} \right\rvert<\infty$. Thus, $H_{n+1} \not\subset \mathcal{F}H_n$, so $\PMap(\G)$ does not have bounded geometry.
\end{proof}

\subsection{Pure Mapping Class Groups of Infinite Asymptotic Dimension} \label{ssec:asdimInf}
     
\begin{THM}
    \label{thm:asdimInf}
    Let $\G$ be a locally finite, infinite graph with $|E_\ell| \ge 2$ and $\PMap(\G)$ locally CB. Then the asymptotic dimension of $\PMap(\G)$ is $\infty$.
\end{THM}

We prove this by coarsely embedding $\Z^k$ with the word metric into $\PMap(\G)$ with a coarsely proper metric for each $k$. Note such a coarsely proper metric exists for $\PMap(\G)$ by \Cref{prop:WDCEtype}, as it is locally CB. We get the embedding from a (not necessarily coarsely proper) length function on $\PMap(\G)$, but we need a few observations beforehand to justify the construction.

\begin{LEM}\label{lem:freefactorint}
    Let $F$ be a free group, and $A$ be a free factor of $F$.
    Then for any subgroup $K$ of $F$, the intersection $K \cap A$ is a free factor of $K$.
    In particular, If $K$ is also a free factor of $F$, then so is the intersection $K \cap A$.
\end{LEM}

\begin{proof}
    Let $F=A*B$ and $K \le F$ be a subgroup. By Kurosh subgroup theorem, (\cite{Kurosch1934} and see also \cite[Theorem 2.7.12]{Sklinos2021}) $K$ admits a free decomposition as
    \[
        K = \left(\ast_{i \in I_1}(K \cap g_{i}^{-1}Ag_i)\right) \ast \left(\ast_{j \in I_2}(K \cap g_{j}^{-1}Bg_{j})\right) \ast F'
    \]
    where $I_1,I_2$ are index sets, $g_i,g_j \in G$, $F'$ is a free group, and $g_i=1$ and $g_j=1$ for some $i \in I_1$ and $j \in I_2$. With this choice of $g_i$, we have that $K \cap A$ is a free factor of $K$, as desired.

    The latter part follows from the fact that a free factor of a free factor of $F$ is a free factor of $F$.
\end{proof}

Now fix a point $x_0 \in \G$ to be the midpoint of some non-loop edge. Let $\G_A := \G_0$, defined in \Cref{ssec:fluxMap} and $\G_B:= \overline{\G\setminus \G_A}$. Then write $A = \pi_1(\G_A,x_0)$ and $B= \pi_1(\G_B,x_0)$. 

By the previous lemma together with \Cref{LEM:subfreefact}, for $f \in \PPHE(\G)$ the two subgroups $f_*(A) \cap A$ and $f_*(B) \cap B$ are free factors of $A$ and $B$ respectively. Now we can show that they actually have finite corank.

\begin{LEM}\label{lem:finiteDisplacement}
    Let $f \in \PPHE(\G).$ Then $f_{*}(A) \cap A$ has finite corank in $A$ and $f_{*}(B) \cap B$ has finite corank in $B$.
\end{LEM}

\begin{proof}
    By symmetry, it suffices to show $\cork(A,f_*(A) \cap A)<\infty$.
    Suppose, for the sake of contradiction, $\cork(A, f_*(A) \cap A) =\infty$.
    By \Cref{COR:admisexist}, there exists $A_m$ that contains $A$ and $f_*(A)$ as free factors with finite corank. By the assumption $\cork(A_m,f_*(A) \cap A) = \cork(A_m,A) + \cork(A,f_*(A) \cap A) = \infty$. Because $\cork(A_m,A)<\infty$, we can find an \emph{infinite} list of distinct basis elements $a_i \in A_m$, such that $a_i \not\in f_*(A)$, but this contradicts $\cork(A_m,f_*(A)) <\infty$.
\end{proof}

Since properly homotopic proper homotopy equivalences induce the same isomorphisms on $\pi_1$,
the quantities $\cork(A, f_*(A) \cap A)$ and $\cork(B,f_*(B) \cap B)$ are well-defined on $\PMap(\G)$.

\begin{DEF}
    Define the \textbf{displacement function} based at $x_0$ on $\PMap(\G)$ to be the map $\cD_{x_{0}}: \PMap(\G) \to \Z_{\ge 0}$, defined as:
    \[
        \cD_{x_0}(f) = \cork(A, f_{*}(A) \cap A) + \cork(B, f_{*}(B) \cap B).
    \]
    We will omit the basepoint $x_0$ and call such map a displacement function on $\PMap(\G)$ and write $\cD$.
\end{DEF}
 By \Cref{lem:finiteDisplacement}, $\mathcal{D}$ only admits finite values. We first show that $\cD$ satisfies the axioms for a length function other than symmetry. 

\begin{PROP}\label{prop:displacementLength}
    $\cD$ is continuous, $\cD(\Id)=0$, and satisfies the triangle inequality.
\end{PROP}

\begin{proof}
  First, $\cD(\Id)=\cork(A,A)+\cork(B,B)=0$. Let $\cV_{x_{0}}$ be a neighborhood about the identity in $\PMap(\G)$ that fixes $x_{0}$ and stabilizes the complementary components of $x_{0}$. Let $g\in\cD\inv(\{n\})$ and $h\in \cV_{x_0}$. Because $h_{*}(A)=A$ and $h_{*}(B)=B$, we have that $\cD(h)=0$ and $\cD(hg)=\cD(g)$. So, $\cV_{x_{0}}g \subset \cD^{-1}(\{n\})$ is an open neighborhood of $g$ in $\cD\inv(\{n\})$, and $\cD$ is continuous.

  Now we show that $\cD(gf) \le \cD(g) + \cD(f)$ for $f,g \in \PMap(\G)$.
  First, decompose
  \[
    A = (g_{*}f_{*}(A) \cap A) \ast H,
  \]
  for some free factor $H$ of $A$.
  By \Cref{lem:finiteDisplacement}, we can define
  \begin{align*}
    r&:= \rk (H) = \cork(A, g_{*}f_{*}(A) \cap A) < \infty, \\
    p&:= \cork(A, f_{*}(A) \cap A) < \infty, \\
    q&:= \cork(A, g_{*}(A) \cap A) < \infty.
  \end{align*}  
  Let $H=\langle a_{1},\ldots,a_{r} \rangle$.
  Further decompose $f_{*}(H)$ as
  \[
    f_{*}(H) = (f_{*}(H) \cap A) \ast H'
  \]
  for some free factor $H'$ of $f_{*}(H)$. Because $\rk(f_{*}(H)) = \rk(H)=r<\infty$, it follows that
  both $f_{*}(H) \cap A$ and $H'$ are of finite rank $\le r$. Pick bases
  \begin{align*}
    f_{*}(H) \cap A &=: \langle a_{1}',\ldots,a_{m}' \rangle \subset A, \\
    H' &=: \langle w_{m+1},\ldots,w_{r} \rangle.
  \end{align*}
   By construction $H' \cap A = 1.$ Now
  \[
    g_{*}f_{*}(H) = \langle g_{*}(a_{1}'),\ldots,g_{*}(a_{m}') \rangle \ast \langle g_{*}(w_{m+1}),\ldots,g_{*}(w_{r}) \rangle.
  \]
  By construction of $H$, we have $g_*f_*(H) \cap A =1,$ and in particular
  \begin{equation} \label{eq:corkgstar}\tag{$\star$}
     g_*\langle a_1',\ldots,a_m' \rangle \cap A = \langle g_*(a_1'),\ldots,g_*(a_m')\rangle \cap A = 1.
  \end{equation}
  Because $\langle a_1',\ldots,a_m' \rangle$ is a rank $m$ free factor of $A$, \Cref{eq:corkgstar} shows that the corank in $A$ of $g_*(A)\cap A$ is at least $m$. That is, $m\leq q$.
 
  On the other hand, because $g_*f_*(H) \cap A =1$, we also have that
  \begin{equation} \label{eq:corkfstar}\tag{$\star\star$}
      f_*\langle f_*^{-1}(w_{m+1}),\ldots,f_*^{-1}(w_r) \rangle \cap A = \langle w_{m+1},\ldots,w_{r} \rangle \cap A = 1,
  \end{equation}
  where $\langle f_*^{-1}(w_{m+1}),\ldots,f_*^{-1}(w_r) \rangle \subset H$ and is a rank $r-m$ free factor of $A$.
  By a similar argument as before, we have $r-m \le p$ by \Cref{eq:corkfstar}.

  All in all, we proved $r=(r-m)+m\le p+q$, hence
  \[
    \cork(A, g_{*}f_{*}(A) \cap A) \le \cork(A, f_{*}(A) \cap A) + \cork(A, g_{*}(A) \cap A).
  \]
  Doing the same for $\cork(B,g_{*}f_{*}(B) \cap B)$, we obtain $\cD(gf) \le \cD(g) + \cD(f)$, concluding the proof.
\end{proof}

Now we symmetrize $\cD$ to make it a length function:
\[
  |\cD|(f):= \frac12\left(\cD(f) + \cD(f^{-1})\right),
\]
for $f \in \PMap(\G)$. Symmetrizing also gives that $|\cD|(f^{-1})=|\cD|(f)$, so $|\cD|$ is a length function by \Cref{prop:displacementLength}.
\begin{COR}\label{cor:displacementMetric}
     The absolute displacement function $|\cD|$ induces a continuous, left-invariant, pseudo-metric $d_{\cD}$ on $\PMap(\G)$.
\end{COR}
\begin{proof}
    Define $d_{\cD}(g,h):= |\cD|(g^{-1}h)$.
    The continuity follows from the continuity of $\cD$, justified in \Cref{prop:displacementLength}.
\end{proof}

For graphs $\G$ with $|E_\ell| \ge 2$ and $\PMap(\G)$ locally CB,
let $d_{\max}$ be the continuous left-invariant and coarsely proper pseudo-metric on $\PMap(\G)$ obtained from \Cref{prop:LocCBCPmetric}.
Then we can compute $\asdim \PMap(\G)$ as $\asdim (\PMap(\G),d_{\max})$, and are ready to prove \Cref{thm:asdimInf}.

\begin{proof}[Proof of \Cref{thm:asdimInf}]
Fix $k \ge 1$ and we will show that we can obtain an \emph{isometric} embedding $(\Z^k,\ell^1) \hookrightarrow (\PMap(\G),d_{\cD})$ from the following class of loop shifts. The $\ell^{1}$-metric is defined on $\Z^k$ as:
\[
    \ell^1\left((e_1,\ldots,e_k),(e'_1,\ldots,e'_k)\right)=\sum_{i=1}^k|e_i-e'_i|.
\]

First, consider the ladder graph $\Lambda$ in standard form as in \Cref{fig:laddergraph}. Label the loops of $\Lambda$ by $\Z$, and for $i=0,\ldots,k-1$, let $\Lambda_i$ be the smallest connected subgraph of $\Lambda$ that consists of the loops labeled by $i$ modulo $k$. Embed $\Lambda$ into $\G$ and now let $\Lambda_i$ refer to its image in $\G$. Let $h_i$ be a loop shift supported on $\Lambda_i \subset \G$. Pick a base point $x_0 \in (v_0,v_1) \subset \Lambda$, the edge connecting the $0$th loop and $1$st loop of $\Lambda$. Then we obtain the absolute displacement function $|\cD|$ on $\G$. Note that any two loop shifts $h_i$ and $h_j$ commute, and
\[
  |\cD|(h_{0}^{e_{0}}\cdots h_{k-1}^{e_{k-1}}) = |e_{0}| + \ldots + |e_{k-1}|,
\]
by the definition of $|\cD|$.
Hence, $\{h_0,\ldots,h_{k-1}\}$ generate a group $H$ isomorphic to $\Z^k$. Therefore, the map
\[
    (\Z^k,\ell^1) \longrightarrow (H,d_{\cD}|_H), \qquad  (e_0,\cdots,e_{k-1}) \mapsto h_0^{e_0}\ldots h_{k-1}^{e_{k-1}}
\]
is an isometry and an isomorphism.
In particular, we have that the map $(\Z^{k},\ell^{1}) \to (H,d_{\cD}|_{H})$ is a coarse equivalence.

Now consider the following commutative diagram: 
\[\begin{tikzcd}
	{(\Z^{k},\ell^1)} && {(H,d_{\cD}|_H)} & {} & {(\PMap(\G),d_{\cD})} \\
	\\
	&& {(H,d_{\max}|_H)} && {(\PMap(\G),d_{\max})}
	\arrow[hook, from=1-3, to=1-5]
	\arrow[hook, from=3-3, to=3-5]
	\arrow["{\iota_1}"', from=1-3, to=3-3]
	\arrow["{\iota_2}", from=3-5, to=1-5]
	\arrow["\cong", from=1-1, to=1-3]
\end{tikzcd}\]
where $\iota_{1}$ and $\iota_{2}$ are the identity maps, and the horizontal inclusions are isometric embeddings. We will check that $\iota_{1}$ is a coarse embedding. Using \Cref{prop:WDCEtype} again, $\iota_{1}$ and $\iota_{2}$ are coarsely Lipschitz because $d_{\cD}|_{H}$ is a coarsely proper metric on $H$ and so is $d_{\max}$ of $\PMap(\G)$. Hence, $\iota_{2}^{-1}$ admits a lower control function $\Phi_{-}:[0,\infty) \to [0,\infty]$, such that for any $g,g' \in \PMap(\G)$:
\[
  \Phi_{-}(d_{\cD}(g,g')) \le d_{\max}(\iota_{2}^{-1}g, \iota_{2}^{-1}g')=d_{\max}(g,g').
\]
Therefore, for any $h,h' \in H \le \PMap(\G)$:
\[
  \Phi_{-}(d_{\cD}|_{H}(h,h')) = \Phi_{-}(d_{\cD}(h,h'))
  \le d_{\max}(h,h') = d_{\max}|_{H}(\iota_{1}(h),\iota_{1}(h')),
\]
proving that $\iota_{1}$ is coarsely expansive. Therefore $\iota_{1}$ is a coarse embedding.
All in all, the composition
\[
  (\Z^{k},\ell^{1}) \longrightarrow (H,d_{\cD}|_{H}) \overset{\iota_{1}}{\longrightarrow} (H,d_{\max}|_{H})  \lhook\joinrel\longrightarrow (\PMap(\G),d_{\max})
\]
is a coarse embedding. This yields an inequality by \Cref{lem:asdimCE}:
\[
  k = \asdim(\Z^{k},\ell^{1}) \le \asdim(\PMap(\G),d_{\max}) = \asdim(\PMap(\G)).
\]
Since $k \ge 1$ was chosen arbitrarily, we conclude $\asdim(\PMap(\G)) = \infty$. 
\end{proof}


\appendix

\section{Ultrametric Spaces are $0$-Hyperbolic}

In this appendix we give a proof of \Cref{lem:ultram0hyp}, that every ultrametric space is 0-hyperbolic. This is a well known fact \cite[Exercise 1]{Duchin2017} but as of yet we have not found a written proof in the literature. 
Recall that Gromov $\delta$-hyperbolicity can be illustrated with the Gromov product as follows: Given a metric space $(X,d)$, define a Gromov product $(x,y)_w:= \frac12 \{d(x,w) + d(y,w) - d(x,y)\}$ for $w,x,y\in X$. Then $X$ is Gromov $\delta$-hyperbolic if and only if
\[
    (x,y)_w \geq \min\{(x,z)_w, (y,z)_w \} - \delta,
\]
for all $w,x,y,z \in X$.

\begin{LEM}
\label{lem:ultram0hyp2}
Every ultrametric space is Gromov 0-hyperbolic.
\end{LEM}

\begin{proof}
    Suppose $(X,d)$ is ultrametric space.
    We first paraphrase what it means to be 0-hyperbolic: For every $w,x,y,z\in X$,
    \begin{flalign*}
        &\phantom{{}\Leftrightarrow{}}(x,y)_w \ge \min \left\{(x,z)_w, (y,z)_w\right\} \\ 
        &\Leftrightarrow d(x,w)+d(y,w)-d(x,y) \ge \min\left\{d(x,w)+d(z,w)-d(x,z), \ d(y,w)+d(z,w)-d(y,z)\right\} \\
        &\Leftrightarrow d(x,w)+d(y,w)-d(x,y) \ge d(x,w)+d(z,w)-d(x,z) \quad \text{OR}
        \\ &\phantom{{}\Leftrightarrow{}} d(x,w)+d(y,w)-d(x,y) \ge d(y,w)+d(z,w)-d(y,z) \\
        &\Leftrightarrow d(x,z)+d(y,w) \ge d(x,y) +d(z,w) \quad \text{OR} \quad d(y,z)+d(x,w) \ge d(x,y) + d(z,w).
    \end{flalign*}

    Now, for the sake of contradiction, suppose there exist $w,x,y,z\in X$ such that:
    \begin{equation} \tag{$\ast$}
    \begin{aligned}
    \label{ineq:toward0hyp}
    &d(x,z)+d(y,w) < d(x,y) + d(z,w) \quad \text{AND} \\ &d(y,z)+d(x,w) < d(x,y) + d(z,w).
    \end{aligned}
    \end{equation}
    For the notational simplicity, write:
    \begin{align*}
        A_1 := d(x,z), \qquad A_2 := d(y,w) \\
        B_1 := d(y,z), \qquad B_2 := d(x,w) \\
        C_1 := d(x,y), \qquad C_2 := d(z,w)
    \end{align*}
    Refer to \Cref{fig:ultratriangle}. 
    \begin{figure}[ht!]
        \centering
        \includegraphics[width=0.3\textwidth]{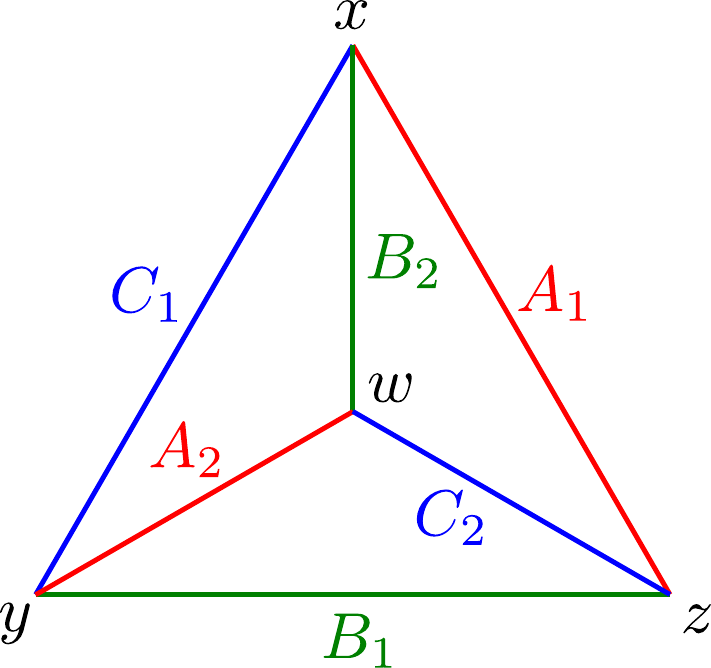}
        \caption{Pictorial representation of notations for the distances of pairs of $x,y,z,w\in X$ in the proof of \Cref{lem:ultram0hyp2}.}
        \label{fig:ultratriangle}
    \end{figure}
    Then our assumption \eqref{ineq:toward0hyp} can be rewritten as:
    \begin{align}
    \label{ineq:0hyperbolic1}
        &A_1 + A_2 < C_1 + C_2 \qquad \text{AND} \tag{i}\\
    \label{ineq:0hyperbolic2}
        &B_1 + B_2 < C_1 + C_2. \tag{ii}
    \end{align}
    Note that the ultrametric property is equivalent to saying that every triangle in \Cref{fig:ultratriangle} is isosceles, with the two identical sides no smaller than the third.
    From this observation, it breaks into two cases in regards to the triangle $\triangle xyz$.
    \begin{CASE}
    $C_1 \le A_1 = B_1$.
    \end{CASE}
    We proceed as follows:
    \begin{align*}
        &A_2 < C_2, \qquad B_2 < C_2 &&\because \text{\eqref{ineq:0hyperbolic1}, and \eqref{ineq:0hyperbolic2}} \\
        &B_1 = C_2 > A_2  &&\because\text{Triangle $\triangle yzw$} \\
        &A_1 = C_2 > B_2  &&\because\text{Triangle $\triangle xzw$} \\
        &C_1 > A_2, \qquad C_1 > B_2 &&\because \text{\eqref{ineq:0hyperbolic1}, and \eqref{ineq:0hyperbolic2}}.
    \end{align*}
    Now in Triangle $\triangle xyw$, we have $C_1$ is the only large side and the other two $A_2,B_2$ are strictly smaller. This contradicts $\triangle xyw$ being isosceles with the two identical sides are the largest among the three sides.
    
    \begin{CASE}
    $C_1 = A_1 > B_1$ \qquad \text{OR} \qquad $C_1 = B_1 >A_1$.
    \end{CASE}
    Note the two situations in this case are symmetric by interchanging $A_i$'s with $B_i$'s, so it suffices to consider when $C_1 = A_1 > B_1$. Now:
    \begin{align*}
        &A_2 < C_2 &&\because \text{\eqref{ineq:0hyperbolic1}} \\
        &B_1 = C_2 > A_2  &&\because\text{Triangle $\triangle yzw$} \\
        &C_1 > B_2 &&\because \text{\eqref{ineq:0hyperbolic2}} \\
        &A_2 = C_1 > B_2  &&\because\text{Triangle $\triangle xyw$} \\
        &A_1 < C_2 &&\because \text{\eqref{ineq:0hyperbolic1}} \\
        &B_2 = C_2 > A_1  &&\because\text{Triangle $\triangle xzw$}.
    \end{align*}
    However, now we have a contradiction that $A_1 > B_1 > A_2 > B_2 > A_1$.
\end{proof}

\bibliography{bib}
\bibliographystyle{plain}

\end{document}